\documentclass[preprint,3p,times]{elsarticle}

\usepackage[toc,page,title,titletoc,header]{appendix}
\usepackage{algorithm}
\usepackage{algorithmic}
\usepackage{amsmath}
\usepackage{amssymb}
\usepackage{amsthm}
\usepackage{enumerate}
\usepackage{enumitem}
\usepackage{mathrsfs}
\usepackage{graphicx}
\usepackage{epsfig}
\usepackage{tikz,pgfplots,tikz-3dplot}
\usepackage{epstopdf}
\usepackage{microtype}
\usetikzlibrary{fit}
\biboptions{sort&compress}

\newtheorem{thm}{Theorem}[section]

\newtheorem{rem}{Remark}
\renewcommand{}

\newcommand{\nn}{\nonumber}

\newcommand{\mS}{\mathcal{S}}

\usepackage{cases}

\def\epsilon{\varepsilon} 
\newcommand{\mat}[1]{\boldsymbol{#1}}

\allowdisplaybreaks
\pgfplotsset{compat=1.18}
\begin{document}
\begin{frontmatter}
\title{
Structure-preserving parametric finite element methods for anisotropic surface diffusion flow with minimal deformation formulation}
\author[1]{Yihang Guo}
\author[1]{Meng Li*}
\address[1]{School of Mathematics and Statistics, Zhengzhou University,
Zhengzhou 450001, China}
\ead{This author's research was supported by National Natural Science Foundation of China (No. 11801527,U23A2065).
Corresponding author: limeng@zzu.edu.cn. }
\begin{abstract}
High mesh quality plays a crucial role in maintaining the stability of solutions in geometric flow problems.
 Duan and Li [Duan \& Li, SIAM J. Sci. Comput. 46 (1) (2024) A587–A608] applied the minimal deformation (MD) formulation to propose an artificial tangential velocity determined by harmonic mapping to improve mesh quality.
In this work, we extend the method to anisotropic surface diffusion flows, which, similar to isotropic curvature flow, also preserves excellent mesh quality.
Furthermore, developing a numerical algorithm for the flow with MD formulation that guarantees volume conservation and energy stability remains a challenging task.
We, in this paper, successfully construct several structure-preserving algorithms, including first-order and high-order temporal discretization methods.  
Extensive numerical experiments show that our methods effectively preserve mesh quality for anisotropic SDFs, ensuring high-order temporal accuracy, volume conservation or/and energy stability.
\end{abstract}
\end{frontmatter}
\pagestyle{myheadings} \markboth{~}
{}
\section{Introduction}\label{sec:intro}
Surface diffusion is a common phenomenon in materials science and solid-state physics, involving the motion of surface atoms, atomic clusters, and molecules on the surface of solid materials \cite{Mullins57}. When the lattice orientations of the material surface vary, it leads to differences in diffusion rates and anisotropic surface energy density. This phenomenon is known as anisotropic surface diffusion, which plays an important role in material science, such as solid-state dewetting \cite{Naffouti17,Wang15,Zhao20,Zhao2020p}, crystal growth of nanomaterials 
 \cite{cahn91,gomer1990} and morphology development of alloys \cite{asaro1972interface}.
 
\begin{figure}[h]
    \centering
    \includegraphics[width=0.6\linewidth]{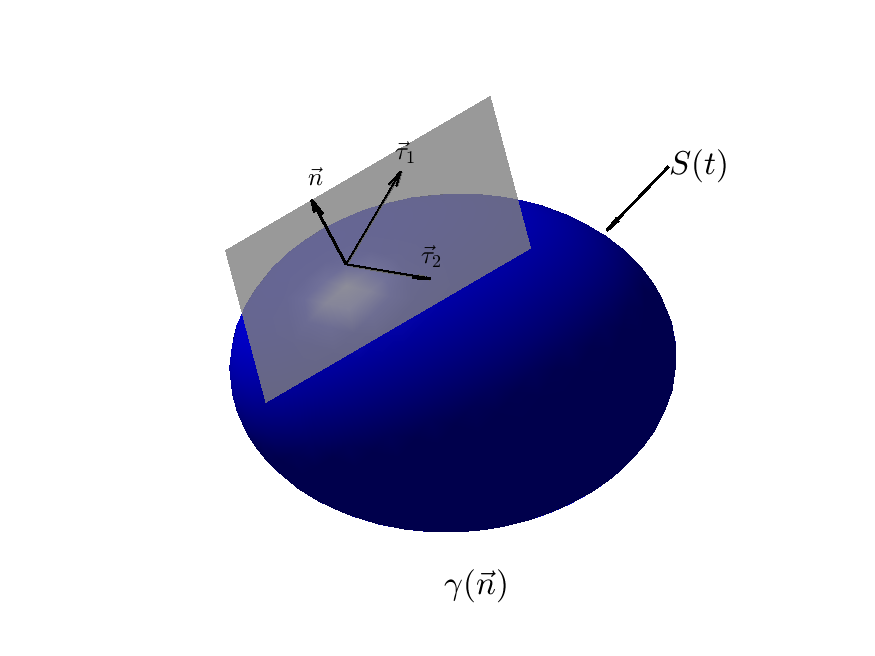}
    \caption{An illustration of surface diffusion on a closed surface $\mS(t)$ with anisotropic surface energy density $\gamma(\vec{n})$ in three dimensions, where $\vec{n}$ is the unit outward normal vector, $\vec{\tau}_1$ and $\vec{\tau}_2$ are a set of orthonormal basis for the local tangential plane.}
    \label{fig:illustration of surface diffusion}
\end{figure}

 As depicted in Figure \ref{fig:illustration of surface diffusion}, $\mS:= \mS(t)$ is a closed and orientable surface in three dimensions (3D), $\vec{n}=(n_1,n_2,n_3)^T \in \mathbb{S}^2$ represents the unit outward normal vector of $\mS(t)$ and $\gamma(\vec{n})$ denotes the anisotropic surface energy density function. The total surface energy of $\mS$ is defined as   
\begin{equation}\label{eqn:energy}
     W(\mS):= W(\mS(t)) = \int_{\mS(t)} \gamma(\vec{n})\, dA,
 \end{equation}
 where $dA$ represents the surface integral. The anisotropic surface diffusion of the closed surface is governed by the following geometric flow \cite{jiang2012phase,Mullins57,Naffouti17,thompson12solid}:
 \begin{equation}\label{eqn:govern}
     \left(\partial_t \vec{X}\circ \vec{X}^{-1}\right) \cdot \vec{n}=\Delta_{\mS} \mu, \qquad \text{on} \quad \mS(t),
 \end{equation}
 where $\vec{X}(\vec{\rho},t): \mS^0 \to \mathbb{R}^3$ is the parameterization of $\mS(t)$ with $\mS^0$ representing the initial surface of $\mS(t)$, $\Delta_{\mS}:=\nabla_{\mS} \cdot \nabla_{\mS}$ represents the surface Laplace-Beltrami operator with $\nabla_{\mS}$ being the surface gradient of $\mS(t)$, and $\mu$ is the chemical potential (or the weighted mean curvature $H_\gamma$) obtained by using the thermodynamic variation of $W(\mS)$:
 \begin{equation}\label{eqn:mu_def}
     \mu = \frac{\delta W(\mS)}{\delta \mS} = \lim_{\epsilon \to 0} \frac{W(\mS^\epsilon)-W(\mS)}{\epsilon}, 
 \end{equation}
 with $\mS^\epsilon$ being a small perturbation of $\mS$ \cite{Jiang16,Jiang19a}.
 
  Let $\gamma(\vec{p})$ be a homogeneous extension with respect to $\gamma(\vec{n})$, which can be defined by 
  \begin{equation}\label{eqn:extension_gamma}
      \gamma(\vec{p}):= 
      \begin{cases}
          |\vec{p}|\gamma\left(\frac{\vec{p}}{|\vec{p}|}\right), \qquad \forall \vec{p} = (p_1,p_2,p_3)^T \in \mathbb{R}^3_*:=\mathbb{R}^3\setminus \{0\},
 \\
          0, \quad \qquad \qquad \vec{p} = 0,
      \end{cases}
  \end{equation}
 where $|\vec{p}| = \sqrt{p_1^2+p_2^2+p_3^2}$. From \cite{cahn1974vector}, the Cahn-Hoffman $\vec{\xi}$-vector  can be formulated as 
 \begin{equation}\label{eqn:Cahn_vector}
     \vec{\xi} = (\xi_1,\xi_2,\xi_3)^T:=\vec{\xi}(\vec{n})=\nabla \gamma(\vec{p})\big|_{\vec{p} = \vec{n}}.
 \end{equation}
Obviously, it holds $\vec{\xi}\cdot \vec{n} = \gamma(\vec n)$, and the chemical potential $\mu$ can be represented as \cite{wheeler1999cahn} 
 \begin{equation}\label{eqn:mu_xi}
     \mu = \nabla_s \cdot \vec{\xi}.
 \end{equation}

There is extensive research on numerical methods for anisotropic surface diffusion flow (SDF), such as the marker-particle method \cite{Wong00,du2010tangent}, the finite difference method \cite{Deckelnick05,Bansch04}, the crystalline method \cite{Klinger11shape}, the discontinuous method and the parametric finite element method (PFEM) \cite{barrett2007parametric,Barrett08Ani,Barrett19,Li2020energy,Bao17}. 
Among these methods for simulating SDF, the BGN method proposed by Barrett, Garcke, and Nürnberg is considered the most efficient and accurate. 
Compared with other PFEMs, the BGN method applied to SDF demonstrates the ability to improve mesh quality without the need for mesh redistribution techniques.
For the evolution of one-dimensional curves, it has been demonstrated that when any three vertices are not colinear, the BGN method satisfies the equal-mesh distribution \cite{Barrett07,Zhao20,Zhao2021}, which partially explains why the tangential velocity generated by the BGN method can improve mesh quality in the evolution of surface diffusion of one-dimensional curves. 
Moreover, Hu and Li in \cite{hu22evolving} explained that the BGN method succeeds in improving mesh quality by minimizing the instantaneous rate of deformation of the surface in the continuous limiting situation. 
However, in approximating the SDF of the surface, when the time step is chosen sufficiently small in the BGN method, the nodes tend to cluster and the mesh gets worse, leading to solution instability. 

 In \cite{duan24}, Duan and Li proposed a novel artificial motion designed to minimize the deformation energy, given by
 \begin{equation}\label{eqn:de_energy}\notag
    E\left(\mS(t)\right)=\frac{1}{2}\int_{\mS^0} \left|\nabla_{\mS^0} X(\cdot,t)\right|^2 \, dA,
\end{equation}
under the constraint $(\partial_t \vec{X}\circ \vec{X}^{-1})\cdot \vec{n} = \vec{v}\cdot \vec{n}$, where $\vec{v}$ denotes the velocity. Equivalently, we can introduce a Lagrange multiplier $\kappa: \mS(t) \to \mathbb{R}$ to transform the constrained optimization problem into solving the following system: 
 \begin{equation}\label{eqn:harmonic_formula}
    \begin{cases} 
     \left(\partial_t \vec{X}\circ \vec{X}^{-1}\right) \cdot \vec{n} = \vec{v}\cdot \vec{n}, \qquad &\text{on} \quad \mS(t), \\
      \left(\kappa \vec{n}\right)\circ \vec{X} = \Delta_{\mS^0} \vec{X}, \qquad  \quad \qquad &\text{on} \quad \mS^0.
    \end{cases} 
 \end{equation}
 This system is referred to as the MD formulation, which arises a harmonic map from the initial surface $\mS^0$ to the surface $\mS(t)$.
Therefore, the map $\vec{X}(\cdot,t)$ in the MD formulation can map the triangles on $\mS^0$ to similar triangles on the evolving surface $\mS(t)$, which resolves the mesh distortion that occurs in the BGN method during isotropic SDF and mean curvature flows (MCFs) when the time step is sufficiently small. 
Inspired by the MD formulation, we apply this method to numerically solve the anisotropic SDF.

As is well known, the evolution of the closed surface \(\mS(t)\), governed by the partial differential equation \eqref{eqn:govern}, possesses two fundamental geometric properties: volume conservation and energy decay.
Establishing numerical methods that preserve these two properties is of great significance.
In \cite{Zhao2021,bao2023symmetrized1,bao2024unified}, Bao et al. proposed several structure-preserving PFEMs for isotropic/anisotropic SDF within the BGN method.
%
Although the MD formulation method can maintain good mesh quality during evolution, constructing a structure-preserving algorithm for this method is still a challenge. 
Very recently, Gao and Li in \cite{gao2025geometric} designed geometric-structure preserving methods for computing surface evolution under isotropic mean curvature flow and surface diffusion with MD formulations. The proposed method achieves first-order accuracy in time and can preserve surface area decrease and volume conservation. In this work, we extend and develop higher-order temporal methods in this paper to address more complex anisotropic SDF problems, which not only preserve energy stability but also maintain volume conservation properties.

In recent years, there are two popular methods for energy dissipative systems: the invariant energy quadratization (IEQ) method and the scalar
auxiliary variable (SAV) approach \cite{Shen18,huang2020highly,shen2019new,huang2022new}. Huang proposed a new SAV approach in \cite{huang2022new}, which can maintain unconditional energy stability for the modified energy under the backward differentiation formula of order k (BDFk) scheme, while also providing a rigorous error estimate for the BDFk ($1\le k \le5$) scheme. Additionally, the Lagrange multiplier (LM) approaches mentioned in \cite{Garcke2024structure,gao2025geometric,cheng2020new} are also effective alternatives. 
Building upon the SAV and LM approaches, we propose several types of numerical schemes based on the MD formulation, with the following key innovations: 
\begin{itemize}
    \item We develop fully discrete PFEMs for the anisotropic SDF \eqref{eqn:govern} and \eqref{eqn:mu_xi}  with tangential velocity to improve mesh quality based on the MD formulation
in \eqref{eqn:harmonic_formula}. For time integration, we adopt the BDFk methods ($1\leq k\leq 4$), abbreviated as MD-BDFk. This method significantly enhances mesh quality compared to the BGN-BDFk methods, particularly in certain specific cases. By carefully designing the normal vector in the MD-BDF1 method, we further develop a volume-conserving method (VC-MD-BDF1 method). Additionally, the MD-BDF2 method proposed in this paper demonstrates improved mesh performance in simulating the evolution of anisotropic SDF, as compared to the extrapolated linearization method in \cite{duan24}.
\item By combining the SAV approach with the MD formulation, we construct high-order time-accurate, energy-stable PFEMs for anisotropic SDFs, referred to as the SAV-MD-BDFk methods.  Compared to the SAV-BGN-BDFk methods, the SAV-MD-BDFk methods offer better mesh quality in certain cases. 
Additionally, we can easily prove the energy stability of the SAV-MD-BDFk methods, with the discrete energy approximating the original energy. 
Furthermore, we can prove the boundedness of the original energy. 
Similarly, by modifying the normal vector in the SAV-MD-BDF1 method, we further develop an approximate volume-conserving scheme, termed the VC-SAV-MD-BDF1 method.
    \item In order to construct a scheme that preserves both volume conservation and energy stability, we combine the LM approach with the MD formulation to design the structure-preserving LM-MD-BDFk methods. Different from the recent work in \cite{Garcke2024structure}, the LM-MD-BDFk methods preserve the structural properties of anisotropic SDFs while significantly improving mesh quality. Additionally, different from \cite{gao2025geometric}, we not only address anisotropic SDFs but also achieve higher-order time discretizations that can ensure volume conservation.
  We note that, in contrast to the SAV-MD-BDFk methods, the LM-MD-BDFk methods are nonlinear schemes that require iteration. While this increases the computational cost, it effectively preserves volume conservation. 
The LM-MD-BDFk methods include three types: the VC-LM-MD-BDFk methods with volume conservation ($1\leq k\leq 4$), the ES-LM-MD-BDFk methods with energy stability (\(k = 1, 2\)), and the SP-LM-MD-BDFk methods that satisfy both volume conservation ($1\leq k\leq 4$) and energy stability (\(k = 1, 2\)).

  \item All of the methods discussed above can address certain issues, but they also have limitations, as detailed in Remark 8. Building upon the advantages of both approaches, we propose the novel LM-SAV-MD-BDFk methods, which can simultaneously achieve approximate volume conservation and energy stability for any $k$, while maintaining stability when reaching equilibrium. We illustrate that when \( r \) takes a larger value, the volume error of the LM-SAV-MD-BDFk method becomes comparable to that of the volume conservation method, without introducing any additional computational cost.
\end{itemize}

The rest of the paper is organized as follows. 
In Section \ref{sec:anisotropic}, we present the parametric finite element approximations for anisotropic SDFs combined with the MD formulation and adopt the BDFk scheme for time discretization.
In Section \ref{sec:SAV}, we further integrate this method with the SAV approach to propose a class of energy-stable SAV-MD-BDFk methods.
In Section \ref{sec:LM_MD}, leveraging the LM approach, we introduce the LM-MD-BDFk methods, which can preserve both energy stability and volume conservation. 
In Section \ref{sec:LM_SAV_MD}, we design novel LM-SAV-MD-BDFk methods that effectively advance the previously proposed methods. In Section \ref{sec:numerical_tests}, we conduct extensive experiments to demonstrate the efficiency, precision, structure preservation, and excellent mesh quality of our proposed schemes. Finally, we conclude the paper in Section \ref{sec:conclusion}.

\section{The MD-BDFk methods for anisotropic SDF}\label{sec:anisotropic}

In this section, with the aid of the MD formulation method \cite{duan24}, we propose a new formula for anisotropic SDF and present its weak formulation along with the corresponding fully discrete schemes.
\subsection{The new formulation}
The velocity of a surface evolving under surface diffusion is given by  
$\vec{v} = (\Delta_\mS \mu) \vec{n}.$ Then, based on the definition of chemical potential \eqref{eqn:mu_def}, by introducing the Cahn-Hoffman vector $\vec{\xi}$ and utilizing the MD formulation \eqref{eqn:harmonic_formula}, we can get a new formulation for anisotropic SDF:
\begin{subequations}\label{eqn:newMD_SDF}
\begin{alignat}{3}
&\left(\partial_t \vec{X}\circ \vec{X}^{-1}\right) \cdot \vec{n} = \Delta_\mS \mu, \quad &&\text{on} \quad \mS(t), \\
&\mu = \nabla_\mS \cdot \vec{\xi}, \quad &&\text{on} \quad \mS(t), \\
&\left(\kappa \vec{n}\right) \circ \vec{X} = \Delta_{\mS^0} \vec{X}, \quad &&\text{on} \quad \mS^0.
\end{alignat}
\end{subequations}

For a differentiable scalar-valued function $f$ on $\mS(t)$, the surface gradient $\nabla_\mS f=\nabla_{\mS(t)} f = (\underline{D}_1f,\underline{D}_2f,\underline{D}_3f)^T$ \cite{Deckelnick05}. Similarly, the surface Jacobian and surface divergence of the differential vector-valued function $\vec{f} = (f_1,f_2,f_3)^T$ on $\mS(t)$ can be defined as \cite{bao2023symmetrized1,bao2024unified}

\begin{equation}\label{eqn:difference}
  \nabla_\mS \vec{f} = \bigg(\nabla_\mS f_1,\nabla_\mS f_2,\nabla_\mS f_3\bigg)^T, \qquad  \nabla_\mS \cdot \vec{f} = \sum_{i=1}^{3} \underline{D}_i f_i. 
\end{equation}

Motivated by \cite{bao2023symmetrized1}, we introduce a symmetric matrix $\mat{Z}_k(\vec{n})$ as
\begin{equation}\label{eqn:matrix}
\mat{Z}_k(\vec{n}) = \gamma(\vec{n})\mat{I}-\vec{n}\vec{\xi}^T-\vec{\xi}\vec{n}^T+k(\vec{n})\vec{n}\vec{n}^T,  
\end{equation}
where $\mat{I}$ is a $3\times3$ identity matrix and $k(\vec{n})$ is a stability function to guarantee the positive definition of the matrix. 
The existence of a minimal stabilizing function $k_0(\vec{n})$ has been proven in \cite{bao2023symmetrized1}. By using similar method in \cite[Theorem 2.2]{bao2024unified}, the chemical potential $\mu$ can be represented as
  \begin{equation}\label{eqn:strong_formula}
      \mu = -\nabla_\mS \cdot \left(\mat{Z}_k(\vec{n}\right)\nabla_\mS \vec{\text{id}})\vec{n}. \qquad \text{on} \quad \mS(t).
  \end{equation}
Therefore, the system \eqref{eqn:newMD_SDF} can be reformulated into the following form:
\begin{subequations}\label{eqn:new2MD_SDF}
    \begin{alignat}{3}
 \label{eqn:model_MD1}      
 &\left(\partial_t \vec{X}\circ \vec{X}^{-1}\right) \cdot \vec{n} = \Delta_\mS \mu, \quad &&\text{on} \quad \mS(t), \\
      & \mu = -\nabla_\mS \cdot \left(\mat{Z}_k(\vec{n})\nabla_\mS \vec{\text{id}}\right)\vec{n}, \quad &&\text{on} \quad \mS(t), \\
     &\left(\kappa \vec{n}\right) \circ \vec{X} = \Delta_{\mS^0} \vec{X}, \quad &&\text{on} \quad \mS^0,
    \end{alignat}
\end{subequations}
where $\vec{\text{id}}$ is the identity function on $\mS(t)$.

For $t\ge 0$, let $V(t)$ and $W(t)$ respectively represent the enclosed volume and the total energy of the evolving closed surface $\mS(t)$, which can be formulated as
\begin{equation}\label{eqn:geo_physic}
     V(t):=\frac{1}{3}\int_{\mS(t)} \vec{\text{id}}\cdot \vec{n}\, dA, \qquad
    W(t):=\int_{\mS(t)}\gamma(\vec{n})\,dA. \end{equation}
Then, applying the Reynolds transport theorem, we can obtain 
\begin{subequations}\label{eqn:geometry_property}
   \begin{align}
     \label{eqn:volume_diss}  &\frac{dV(t)}{dt}=\left(\left(\partial_t \vec{X} \circ \vec{X}^{-1} \right)\cdot \vec{n},1\right)_{\mS(t)}=0, \qquad \forall t \ge 0, \\ 
       &\label{eqn:ener_diss} \frac{dW(t)}{dt}=\left(\left(\partial_t \vec{X} \circ \vec{X}^{-1}\right)\cdot \vec{n},\mu\right)_{\mS(t)} = -\left(\nabla_\mS \mu, \nabla_\mS \mu\right)_{\mS(t)}\le 0, \qquad \forall t \ge 0.
   \end{align} 
\end{subequations}

In order to obtain the weak formulation of \eqref{eqn:new2MD_SDF}, we define the functional space with respect to the closed surface $\mS(t)$:
\begin{equation}\label{eqn:L2space}
    L^2(\mS(t)):=\left\{u:\mS(t) \to \mathbb{R} \bigg|\int_{\mS(t) }|u|^2 \,dA < \infty \right\},
\end{equation}
equipped with the $L^2$-inner product
\begin{equation*}\label{eqn:inner_product}
    \left(u,v\right)_{\mS(t)}: = \int_{\mS(t)} uv\,dA, \qquad \forall u,v \in L^2(\mS(t)).
\end{equation*}
Similarly, the definition \eqref{eqn:L2space} can be directly applied to $[L^2(\mS(t))]^3$ and $[L^2(\mS(t))]^{3\times 3}$. Specially, the $L^2$-inner product of two matrix-valued functions $\mathbf{U}$, $\mathbf{V}\in[L^2(\mS(t))]^{3\times 3}$ requires the use of the matrix Frobenius inner product $\mathbf{U}:\mathbf{V} = \text{Tr}(\mathbf{V}^T \mathbf{U})$, which is defined as follows:  
\begin{equation*} \label{eqn: matrix_inner}
    \left \langle \mathbf{U},V  \right \rangle_{\mS(t)}:=\int_{\mS(t)} \mathbf{U}:\mathbf{V} \, dA, \qquad \forall \mathbf{U},\mathbf{V}\in[L^2(\mS(t))]^{3\times 3}.
\end{equation*}
Based on \eqref{eqn:L2space}, we further define the following Sobolev spaces:
\begin{subequations}\label{eqn:Sobolev}
  \begin{align}
&H^1(\mS(t)):=\bigg\{u:\mS(t) \to \mathbb{R}\big| u\in [L^2(\mS(t))], \nabla_{\mS(t)} u \in [L^2(\mS(t))]^3\bigg\},\\ &
\left[H^1(\mS(t))\right]^3:=\bigg\{\vec{u}\to \mathbb{R}^3 \big|\vec{u}\in[L^2(\mS(t))]^3,\nabla_{\mS(t)} \vec{u} \in [L^2(\mS(t))]^{3\times 3}\bigg\}.
 \end{align}
\end{subequations}

Therefore, the weak formulation of the system \eqref{eqn:new2MD_SDF} is as follows: given the initial closed surface $\mS(0) = \vec{X}(\vec{\rho},0)\in [H^1(\mS^0)]^3$, find the evolution surface $\mS(t) = \vec{X}(\vec{\rho},t)$, the chemical potential $\mu\in H^1(\mS(t))$ and the scalar function $\kappa\in L^2(\mS(t)) $ for $t >0$, such that 
\begin{subequations}\label{eqn:weakly_formula}
 \begin{align}
    &\left(\left(\partial_t \vec{X}\circ \vec{X}^{-1}\right)\cdot \vec{n}, \varphi\right)_{\mS(t)} + \left(\nabla_\mS \mu,\nabla_\mS \varphi\right)_{\mS(t)} = 0, \qquad \forall \varphi \in H^1(\mS(t)), \\
    & \left(\mu, \phi\right)_{\mS(t)}-\left(\mat{Z}_k(\vec{n})\nabla_{\mS}\vec{\text{id}},\nabla_\mS (\vec{n}\phi)\right)_{\mS(t)} = 0, \qquad \forall \phi \in H^1(\mS(t)),\\ \label{eqn:kappa_weakly}
    & \left((\kappa\vec{n})\circ \vec{X}, \vec{\omega}\right)_{\mS^0}+\left(\nabla_{\mS^0}\vec{X},\nabla_{\mS^0}\vec{\omega}\right)_{\mS^0} = 0, \qquad \forall \vec{\omega}\in [H^1(\mS^0)]^3.
  \end{align}  
\end{subequations}

For computational convenience, we can transform the integral in \eqref{eqn:kappa_weakly} to the evolving closed surface \(\mS(t)\) using the Jacobian matrix \(\mat{J} = \nabla_{\mS(t)} \vec{X}^0 (\nabla_{\mS(t)} \vec{X}^0)^T + \vec{n}\vec{n}^T\).
Since the ratio of the surface area elements between \(\mS^0\) and \(\mS(t)\) equals to \(\sqrt{\text{det}(\mat{J})}\), then we have
\begin{equation}\label{eqn:kappa_weakly2}
     \left(\kappa\vec{n}, \vec{\omega}\sqrt{\text{det}(\mat J)}\,\right)_{\mS(t)}+\left(\nabla_{\mS(t)}\vec{\text{id}}\mat J^{-1},\nabla_{\mS(t)}\vec{\omega}\sqrt{\text{det}(\mat J)}\,\right)_{\mS(t)} = 0, \qquad \forall \vec{\omega}\in [H^1(\mS(t))]^3.
 \end{equation}
 
\subsection{The MD-BDFk methods}
Let \(\Delta t > 0\) represent the time step, and define the discrete time levels as \(t_m = m \Delta t\) for all \(m \geq 0\).
For spatial discretization, the evolving surface \(\mS(t_m)\) at any time level is approximated by the polygonal surface mesh \(\mS^m := \bigcup_{l=1}^L \overline{\sigma}_l^m\), consisting of \(L\) mutually disjoint, non-degenerate triangles \(\sigma_l^m\) and \(K\) vertices \(\vec{q}_k^m\).
The 
 element $\sigma_l^m:= \Delta\{\vec{q}_{l_1}^m,\vec{q}_{l_2}^m,\vec{q}_{l_3}^m\}, 1\le l\le L$ indicates that a triangle formed by vertices $\{\vec{q}_{l_1}^m,\vec{q}_{l_2}^m,\vec{q}_{l_3}^m\}$ in an anticlockwise sequence.
The unit outward normal vector of $\sigma_l^m$ can be obtained by 
 \begin{equation*}
     \vec{n}_l^m: = \frac{\mathscr T\{\sigma_l^m\}}{|\mathscr T\{\sigma_l^m\}|}\quad\text{with}\quad \mathscr{T}\{\sigma_l^m\} = \left(\vec{q}^m_{l_2}-\vec{q}^m_{l_1}\right)\times \left(\vec{q}^m_{l_3}-\vec{q}^m_{l_1}\right).
 \end{equation*}
 We define the following finite element space
 \begin{equation}\label{eqn:finite_space}
     \mathbb{K}(\mS^m):=\left\{u\in C(\mS^m): u\big|_{\sigma_l^m} \in \mathbb{P}^1(\sigma_l^m), \quad \forall 1\le l \le L \right\},
 \end{equation}
 where $\mathbb{P}^1(\sigma_l^m)$ represents the space of polynomials on the $l$-th element with degree at most $1$. The definition \eqref{eqn:finite_space} can be directly extended to $[\mathbb{K}(\mS^m)]^3$ and $[\mathbb{K}(\mS^m)]^{3\times 3}$. Furthermore, the mass-lumped inner product $(\cdot,\cdot)_{\mS^m}^h$ can be denoted as 

\begin{equation}\label{eqn:mass_lumped}
    (f,g)_{\mS^m}^h :=\frac{1}{3}\sum_{l=1}^L \sum_{i=1}^{3}|\sigma_l^m|f\left((\vec{q}_{l_i})^-\right)g\left((\vec{q}_{l_i})^-\right),
\end{equation}
where $|\sigma_l^m| := \frac{|\mathscr T\{\sigma_l^m\}|}{2}$ denotes the area of $\sigma_l^m$ and $f\left((\vec{q}_{l_i})^-\right)$ indicates the one-limit of $f(\vec{x})$ when $\vec{x}$ approaches vertice $\vec{q}_{l_i}$ inside $\sigma_l^m$. Similar to \eqref{eqn: matrix_inner}, the definition \eqref{eqn:mass_lumped} can also be extended to the mass-lumped inner product of the matrix-valued functions $\mathbf{U}$ and $\mathbf{V}$ as follows
\begin{equation}\label{eqn:matrix_mass_lumped}
    \left \langle \mathbf{U},\mathbf{V}  \right \rangle_{\mS^m}^h:=\frac{1}{3}\sum_{l=1}^L \sum_{i=1}^{3}|\sigma_l^m|\mathbf{U}\left((\vec{q}_{l_i})^-\right):\mathbf{V}\left((\vec{q}_{l_i})^-\right).
\end{equation}

According to \eqref{eqn:mass_lumped} and \eqref{eqn:geo_physic}, we can represent the discrete enclosed volume and surface energy with $V\left(\vec{X}^{m}\right)$ and $W\left(\vec{X}^{m}\right)$ respectively, and define them as 
\begin{subequations}\label{eqn:discrete_geo}
    \begin{align}
        &  \label{eqn:disc_vol}V\left(\vec{X}^{m}\right):=\frac{1}{3}\int_{\mS^m} \vec{\text{id}}\cdot\vec{n}^m\,dA = \frac{1}{9}\sum_{l=1}^{L}\sum_{i=1}^3 |\sigma_l^m|\left(\vec{q}^m_{l_i}\cdot \vec{n}^m_l\right), \qquad \forall m \ge 0, \\ & \label{eqn:disc_ener}
W\left(\vec{X}^{m}\right):=\int_{\mS^m}\gamma(\vec{n}^m)\,dA = \sum_{l=1}^L |\sigma_l^m|\gamma(\vec{n}^m_l),\qquad \forall m \ge 0.
    \end{align}
\end{subequations}
By the definition of the unit vector $\vec{n}_l^m$ on $\sigma_l^m$, we define the averaged normal vector $\vec{v}_k^m \in [\mathbb{K}(\mS^m)]^3$ on vertice $\vec{q}_k^m$ as
\begin{equation}\label{eqn:averaged_vector}
    \vec{v}_k^m := \frac{\sum_{\sigma_l^m \in \mat{T}_k^m} |\sigma_l^m| \vec{n}_l^m}{\bigg|\sum_{\sigma_l^m \in \mat{T}_k^m} |\sigma_l^m| \vec{n}_l^m\bigg|}\quad \text{with}\quad \mat{T}_k^m =\left\{\sigma_l^m \big| \vec{q}_k^m \in \overline{\sigma}_l^m \right\}.
\end{equation}

By applying PFEM discretization in space and the backward-Euler method in time, we obtain the fully discrete scheme (referred to as the MD-BDF1 method) of the weak formulation \eqref{eqn:weakly_formula}: given the initial closed surface $\mS^0:=\vec{X}^0(\cdot)\in [\mathbb{K}(\mS^0)]^3$, find the solution $(\vec{X}^{m+1},\vec{\chi}^{m+1},\mu^{m+1},\kappa^{m+1})\in [\mathbb{K}(\mS^0)]^3 \times [\mathbb{K}(\mS^m)]^3 \times \mathbb{K}(\mS^m) \times \mathbb{K}(\mS^m)$, such that
 \begin{subequations}\label{eqn:first_full}
     \begin{align}
         & \left(\frac{\vec{\chi}^{m+1}-\vec{\text{id}}}{\Delta t}\cdot \vec{v}^m, \varphi\right)_{\mS^m}^h + (\nabla_\mS \mu^{m+1},\nabla_\mS \varphi)_{\mS^m}^h = 0, \qquad \forall \varphi \in \mathbb{K}(\mS^m),\label{eqn:first_fulla}\\
         &(\mu^{m+1}, \phi)_{\mS^m}^h-\left(\mat{Z}_k(\vec{n}^m)\nabla_{\mS}\vec{\chi}^{m+1},\nabla_\mS (\vec{v}^m\phi)\right)_{\mS^m}^h = 0, \qquad \forall \phi \in \mathbb{K}(\mS^m),\label{eqn:first_fullb}\\
         &\left(\kappa^{m+1}, \vec{v}^m\cdot \vec{\omega}\sqrt{\text{det}(\mat J^m)}\right)_{\mS^m}^h+\left(\nabla_{\mS}\vec{\chi}^{m+1}(\mat J^m)^{-1},\nabla_{\mS}\vec{\omega}\sqrt{\text{det}(\mat J^m)}\right)_{\mS^m}^h = 0, \qquad \forall \vec{\omega}\in [\mathbb{K}(\mS^m)]^3, \label{eqn:first_fullc}\\ 
         &\vec{X}^{m+1} = \vec{\chi}^{m+1}\circ \vec{X}^m,\label{eqn:first_fulld}
     \end{align}
 \end{subequations}
 where $\vec{\chi}^{m+1}$ is a parameterization of the evolving surface $\mS^{m+1}=\vec{\chi}^{m+1}(\vec{X}^{m})$.  The discrete surface gradient of $f\in \mathbb{K}(\mS^m)$ can be calculated as
\begin{equation}\label{eqn:gradient1}
  \nabla_{\mS} f \big|_{\sigma_l^m} :=f(\vec{q}_{l_1}^m)\frac{(\vec{q}_{l_2}^m-\vec{q}_{l_3}^m)\times \vec{n}_l^m}{|(\vec{q}_{l_2}^m-\vec{q}_{l_1}^m)\times (\vec{q}_{l_3}^m-\vec{q}_{l_1}^m)|} + f(\vec{q}_{l_2}^m)\frac{(\vec{q}_{l_3}^m-\vec{q}_{l_1}^m)\times \vec{n}_l^m}{|(\vec{q}_{l_2}^m-\vec{q}_{l_1}^m)\times (\vec{q}_{l_3}^m-\vec{q}_{l_1}^m)|} +f(\vec{q}_{l_3}^m)\frac{(\vec{q}_{l_1}^m-\vec{q}_{l_3}^m)\times \vec{n}_l^m}{|(\vec{q}_{l_1}^m-\vec{q}_{l_2}^m)\times (\vec{q}_{l_3}^m-\vec{q}_{l_1}^m)|}.
\end{equation}

We note that \eqref{eqn:first_full} cannot be proved to be volume-conservative. Motivated by  \cite{bao2023symmetrized1,gao2025geometric}, we define the following time-weighted approximation:
\begin{equation}\label{eqn:normal_vectors}
     \vec{n}_l^{m+\frac{1}{2}}:=\frac{\mathscr T\{\sigma_l^m\} + 4\mathscr T\{\sigma_l^{m+\frac{1}{2}}\}+\mathscr T\{\sigma_l^{m+1}\}}{6\left|\mathscr T\{\sigma_l^m\}\right|},\quad l = 1,2,\dots,L.
 \end{equation}
Then, we can obtain a volume-conserving scheme (VC-MD-BDF1 method) for anisotropic SDF: given the initial closed surface $\mS^0:=\vec{X}^0(\cdot)\in [\mathbb{K}(\mS^0)]^3$, find the solution $(\vec{X}^{m+1},\vec{\chi}^{m+1},\mu^{m+1},\kappa^{m+1})\in [\mathbb{K}(\mS^0)]^3 \times [\mathbb{K}(\mS^m)]^3 \times \mathbb{K}(\mS^m) \times \mathbb{K}(\mS^m)$, such that
 \begin{subequations}\label{eqn:first_full_VC}
     \begin{align}
         & \left(\frac{\vec{\chi}^{m+1}-\vec{\text{id}}}{\Delta t}\cdot \vec{n}^{m+\frac{1}{2}}, \varphi\right)_{\mS^m}^h + (\nabla_\mS \mu^{m+1},\nabla_\mS \varphi)_{\mS^m}^h = 0, \qquad \forall \varphi \in \mathbb{K}(\mS^m),
         \label{eqn:first_full_VCa}
         \\
         &(\mu^{m+1}, \phi)_{\mS^m}^h-\left(\mat{Z}_k(\vec{n}^m)\nabla_{\mS}\vec{\chi}^{m+1},\nabla_\mS \left(\vec{v}^m\phi\right)\right)_{\mS^m}^h = 0, \qquad \forall \phi \in \mathbb{K}(\mS^m),\label{eqn:first_full_VCb}\\
         &\left(\kappa^{m+1}, \vec{v}^m\cdot \vec{\omega}\sqrt{\text{det}\left(\mat J^m\right)}\right)_{\mS^m}^h+\left(\nabla_{\mS}\vec{\chi}^{m+1}(\mat J^m)^{-1},\nabla_{\mS}\vec{\omega}\sqrt{\text{det}(\mat J^m)}\right)_{\mS^m}^h = 0, \qquad \forall \vec{\omega}\in [\mathbb{K}(\mS^m)]^3, \label{eqn:first_full_VCc} \\ 
         & \vec{X}^{m+1} = \vec{\chi}^{m+1} \circ \vec{X}^m.
         \label{eqn:first_full_VCd}
     \end{align}
 \end{subequations}
 
Similar to \cite{Zhao2021,gao2025geometric}, the proof of volume conservation for the VC-MD-BDF1 method relies on the identity: $V^{m+1}-V^m = \left
([\vec{\chi}^{m+1}-\vec{\text{id}}]\cdot \vec{n}^{m+\frac{1}{2}},1\right)_{\mS^m}^h$. Hence, by setting $\varphi = \Delta t$ in \eqref{eqn:first_full_VCa}, we can directly obtain $V^{m+1}=V^m$.

In what follows, by using implicit-explicit BDFk schemes for time discretization and employing the PFEM for spatial discretization, we aim to design temporal high-order algorithms for the weak formulation \eqref{eqn:weakly_formula}.
Before giving the high-order BDFk schemes, we first define the map $\vec{\chi}^{m-d} := \vec{X}^{m-d} \circ (\vec{\widetilde{X}}^{m+1})^{-1} \in [\mathbb{K}(\widetilde{\mS}^{m+1})]^3, d = 0, 1\dots,k-1$
, by introducing a suitably approximated surface $\widetilde{\mS}^{m+1}$.
Then, the MD-BDFk methods are as follows: given $\vec{\chi}^m,\vec{\chi}^{m-1},\dots,\vec{\chi}^{m-k+1}\in [\mathbb{K}(\widetilde{\mS}^{m+1})]^3$, find  $(\vec{X}^{m+1},\vec{\chi}^{m+1},\mu^{m+1},\kappa^{m+1})\in [\mathbb{K}(\mS^0)]^3 \times [\mathbb{K}(\widetilde{\mS}^{m+1})]^3 \times \mathbb{K}(\widetilde{\mS}^{m+1}) \times \mathbb{K}(\widetilde{\mS}^{m+1})$, such that
 \begin{subequations} \label{eqn:MDBDFk_scheme}
    \begin{align}
\label{eqn:MD_1}&\left(\frac{\alpha_k\vec{\chi}^{m+1}-A_k(\vec{\chi}^m)}{\Delta t}\cdot \vec{\widetilde{v}}^{m+1}, \varphi\right)_{\widetilde{\mS}^{m+1}}^h + (\nabla_\mS \mu^{m+1},\nabla_\mS \varphi)_{\widetilde{\mS}^{m+1}}^h = 0, \qquad \forall \varphi \in \mathbb{K}(\widetilde{\mS}^{m+1}),\\ \label{eqn:MD_2}
         &(\mu^{m+1}, \phi)_{\mS^{m+1}}^h-\left(\mat{Z}_k(\vec{\widetilde{n}}^{m+1})\nabla_{\mS}\vec{\chi}^{m+1},\nabla_\mS \left(\vec{\widetilde{v}}^{m+1}\phi\right)\right)_{\widetilde{\mS}^{m+1}}^h = 0, \qquad \forall \phi \in \mathbb{K}(\widetilde{\mS}^{m+1}),\\ \label{eqn:MD_3}
         &\left(\kappa^{m+1}, \vec{\widetilde{v}}^{m+1}\cdot \vec{\omega}\sqrt{\text{det}\left( \widetilde{\mat J}^{m+1}\right)}\right)_{\widetilde{\mS}^{m+1}}^h+\left(\nabla_{\mS}\vec{\chi}^{m+1}\left(\widetilde{\mat J}^{m+1}\right)^{-1},\nabla_{\mS}\vec{\omega} \sqrt{\text{det}\left( \widetilde{\mat J}^{m+1}\right)}\right)_{\widetilde{\mS}^{m+1}}^h = 0, \qquad \forall \vec{\omega}\in [\mathbb{K}(\widetilde{\mS}^{m+1})]^3,
         \\ 
         & \vec{X}^{m+1} = \vec{\chi}^{m+1} \circ \vec{\widetilde{X}}^{m+1},
    \end{align}     
\end{subequations}
where the $ \widetilde{\mat J}^{m+1}$ is the Jacobian matrix between the initial surface $\mS^0$ and the approximated surface $\widetilde{\mS}^{m+1}$, which can be calculated by 
\begin{equation}\label{eqn:discrete_Jacob}
    \widetilde{\mat J}_l^{m+1} = \nabla_{\widetilde{\mS}^{m+1}}\vec{Y}^{0}\left(\nabla_{\widetilde{\mS}^{m+1}}\vec{Y}^{0}\right)+\vec{\widetilde{n}}_l^{m+1}(\vec{\widetilde{n}}_l^{m+1})^T, \qquad l = 1,2,\dots L.
\end{equation}
The $\vec{Y}^0\in [\mathbb K(\widetilde{\mS}^{m+1})]^3$ is a finite element function with its values at the vertices of $\widetilde{\mS}^{m+1}$ being the vertices of $\mS^0$. Meanwhile, the selection of \(\alpha_k\) and \(A_k(\vec{\chi}^m)\) varies with the order \(k\):
\begin{itemize}
    \item $k = 1: \quad 
    \alpha_1 = 1, \qquad A_1(\vec{\chi}^m)=\vec{\chi}^m$;
    \item $k = 2: \quad  \alpha_2 = \frac{3}{2}, \qquad A_2(\vec{\chi}^m)=2\vec{\chi}^m-\frac{1}{2}\vec{X}^{m-1}$;
    \item $k = 3: \quad \alpha_3 = \frac{11}{6}, \qquad A_3(\vec{\chi}^m)=3\vec{\chi}^m-\frac{3}{2}\vec{\chi}^{m-1}+\frac{1}{3}\vec{\chi}^{m-2}$;
     \item $k = 4: \quad \alpha_4 = \frac{25}{12}, \qquad A_4(\vec{\chi}^m)=4\vec{\chi}^m-3\vec{\chi}^{m-1}+\frac{4}{3}\vec{\chi}^{m-2}-\frac{1}{4}\vec{\chi}^{m-3}$.
\end{itemize}

\begin{rem}
For the MD-BDFk methods, we need to prepare the initial data \(\vec{X}^0, \vec{X}^1, \dots, \vec{X}^{k-1}\), which should serve as the \(k\)-th order approximations to \(\vec{X}(\cdot,0), \vec{X}(\cdot,t_1), \dots, \vec{X}(\cdot,t_{k-1})\) at the respective time points.
In general, it is sufficient to apply the MD-BDF1 method to obtain the initial data of the MD-BDFk methods ($k> 1$) with an appropriate time step size of $\widetilde{\Delta t} \sim \Delta t^{k-1}$ by $\Delta t /\widetilde{\Delta t}$ steps. 
For the MD-BDF1 method, $\widetilde{\mS}^{m+1}$ can be directly selected as $\mS^m$, while for the high-order MD-BDFk methods ($k\geq 2$), we can predict $\widetilde{\mS}^{m+1}$ by solving the lower-order method. For instance, $\widetilde{\mS}^{m+1}$ for the MD-BDF3 method can be obtained from the solution of the MD-BDF2 method. 
 \end{rem}

 \begin{rem}\label{rem:structure}
Similar to the isotropic case discussed in \cite{duan24,gao2025geometric}, the mesh quality of the MD-BDFk methods shows significant improvement over the BGN method in certain special anisotropic scenarios.
However, proving the volume conservation and energy stability properties of the MD-BDFk methods remains a significant challenge. 
Although we can demonstrate the volume conservation of the VC-MD-BDF1 method, it is still limited to first-order time discretization. Recently, using the LM approach, Gao and Li \cite{gao2025geometric} developed a structure-preserving algorithm based on the MD formulation for mean curvature flow and SDF, but this work is still limited to first-order time discretization and isotropic case. 
In the following three sections, we design high-order time discretization schemes with structure-preserving properties by combining the SAV or/and LM methods, namely the SAV-MD-BDFk methods, the LM-MD-BDFk methods and the LM-SAV-MD-BDFk methods. Each of these schemes has its own advantages. 
For example, the SAV-MD-BDFk method is a linear scheme, which significantly improves computational efficiency compared to the nonlinear LM-MD-BDFk method, and energy stability can be theoretically proven for $1\leq k\leq 4$. However, the volume conservation property of the SAV-MD-BDFk methods cannot be proven, and only an approximate volume-preserving result can be obtained. 
The LM-MD-BDFk methods, on the other hand, can achieve volume conservation for $1\leq k\leq 4$, but the proof of energy stability is limited to first- and second-order schemes. 
The methods become unstable as they approach equilibrium. 
By fully considering the advantages of both methods, we further design the improved LM-SAV-MD-BDFk methods, which can simultaneously achieve approximate volume conservation and energy stability for \(1\leq k\leq 4\), while maintaining stability as they approach equilibrium. 

 \end{rem}

 \section{The SAV-MD-BDFk methods for anisotropic SDF}\label{sec:SAV}
In this section, we develop a class of unconditionally energy-stable schemes by combining the SAV approach \cite{huang2022new} with the MD formulation.

\subsection{SAV-MD-BDFk methods and VC-MD-BDFk method}
According to the \eqref{eqn:geo_physic}, we define a new energy function $W_c(t)$:
\begin{equation}\label{eqn:new_ener}
    W_c(t):=W(t)+C,
\end{equation}
where $C\in \mathbb{R}$  is a positive constant introduced for theoretical purposes. Based on \eqref{eqn:ener_diss}, there holds  $\frac{dW_c(t)}{dt} = \frac{dW(t)}{dt}= -(\nabla_\mS \mu, \nabla_\mS \mu)_{\mS(t)}$. On the definition of $W_c(t)$, we further introduce the following time-dependent variables:
\begin{equation}\label{eqn:two_SAV}
    R(t):=W_c(t), \qquad \zeta(t):=\frac{R(t)}{W_c(t)}\equiv 1.
\end{equation}
Combining \eqref{eqn:new2MD_SDF} and \eqref{eqn:two_SAV}, a new formulation can be introduced as follows:
\begin{subequations}\label{eqn:new3_MD}
    \begin{alignat}{3}
      &\left(\partial_t \vec{X}\circ \vec{X}^{-1}\right) \cdot \vec{n} = \Delta_\mS \mu, \quad  &&\text{on} \quad \mS(t),\\
      & \mu = -\nabla_\mS \cdot (\mat{Z}_k(\vec{n})\nabla_\mS \vec{\text{id}})\vec{n},\quad  &&\text{on} \quad \mS(t), \\
     & \left(\kappa \vec{n}\right)\circ \vec{X} = \Delta_{\mS^0} \vec{X},\quad  &&\text{on} \quad \mS^0,\\& \label{eqn:R_dis}
     \frac{dR(t)}{dt}=-\zeta(t)\left(\nabla_\mS \mu, \nabla_\mS \mu \right)_{\mS(t)},\quad  &&\text{on} \quad \mS(t).
    \end{alignat}
\end{subequations}

We also employ the BDFk schemes for time discretization and apply the PFEM for spatial discretization to obtain the following full discrete scheme (denoted as SAV-MD-BDFk methods). Given \(\vec{\chi}^m, \vec{\chi}^{m-1}, \dots, \vec{\chi}^{m-k+1} \in [\mathbb{K}(\widetilde{\mS}^{m+1})]^3\) and \(R^m\), the solution is to find \(\left(\vec{\overline{X}}^{m+1}, \vec{\overline{\chi}}^{m+1}, \overline{\mu}^{m+1}, \overline{\kappa}^{m+1}\right) \in [\mathbb{K}(\mS^0)]^3 \times [\mathbb{K}(\widetilde{\mS}^{m+1})]^3 \times \mathbb{K}(\widetilde{\mS}^{m+1}) \times \mathbb{K}(\widetilde{\mS}^{m+1})\), \(\zeta^{m+1}\), \(R^{m+1}\), and \(\left(\vec{X}^{m+1}, \mu^{m+1}, \kappa^{m+1}\right) \in [\mathbb{K}(\mS^0)]^3 \times \mathbb{K}(\widetilde{\mS}^{m+1}) \times \mathbb{K}(\widetilde{\mS}^{m+1})\), such that 
\begin{subequations} \label{eqn:SAV_scheme}
    \begin{align}
      \label{eqn:SAV_1}&\left(\frac{\alpha_k\vec{\overline{\chi}}^{m+1}-A_k(\vec{\chi}^m)}{\Delta t}\cdot \vec{\widetilde{v}}^{m+1}, \varphi\right)_{\widetilde{\mS}^{m+1}}^h + \left(\nabla_\mS \overline{\mu}^{m+1},\nabla_\mS \varphi\right)_{\widetilde{\mS}^{m+1}}^h = 0, \qquad \forall \varphi \in \mathbb{K}(\widetilde{\mS}^{m+1}),\\ \label{eqn:SAV_2}
         &\left(\overline{\mu}^{m+1}, \phi \right)_{\widetilde{\mS}^{m+1}}^h-\left(\mat{Z}_k \left(\vec{\widetilde{n}}^{m+1}\right)\nabla_{\mS}\vec{\overline{\chi}}^{m+1},\nabla_\mS \left(\vec{\widetilde{v}}^{m+1}\phi\right)\right)_{\widetilde{\mS}^{m+1}}^h = 0, \qquad \forall \phi \in \mathbb{K}(\widetilde{\mS}^{m+1}),\\ \label{eqn:SAV_3}
         &\left(\overline{\kappa}^{m+1}, \vec{\widetilde{v}}^{m+1}\cdot \vec{\omega}\sqrt{\text{det}\left( \widetilde{\mat J}^{m+1}\right)}\right)_{\widetilde{\mS}^{m+1}}^h+\left(\nabla_{\mS}\vec{\overline{\chi}}^{m+1}\left(\widetilde{\mat J}^{m+1}\right)^{-1},\nabla_{\mS}\vec{\omega}\sqrt{\text{det}\left( \widetilde{\mat J}^{m+1}\right)}\right)_{\widetilde{\mS}^{m+1}}^h = 0, \qquad \forall \vec{\omega}\in [\mathbb{K}(\widetilde{\mS}^{m+1})]^3,
         \\ & \label{eqn:SAV_7} \vec{\overline{X}}^{m+1} = \vec{\overline{\chi}}^{m+1}\circ \vec{\widetilde{X}}^{m+1},
         \\& \label{eqn:SAV_4}
    \frac{R^{m+1}-R^m}{\Delta t} = -\zeta^{m+1}\left(\nabla_\mS \overline{\mu}^{m+1},\nabla_\mS \overline{\mu}^{m+1}\right)_{\overline{\mS}^{m+1}}^h \quad \text{with} \quad \zeta^{m+1}=\frac{R^{m+1}}{W_c\left(\vec{\overline{X}}^{m+1}\right)}, \\& \label{eqn:SAV_5}
    \eta^{m+1} = 1-(1-\zeta^{m+1})^r, \\ & \label{eqn:SAV_6}
    \vec{X}^{m+1} = \eta^{m+1}\vec{\overline{X}}^{m+1}, \quad \mu^{m+1}=\eta^{m+1}\overline{\mu}^{m+1} \quad \text{and} \quad \kappa^{m+1}=\eta^{m+1}\overline{\kappa}^{m+1}.  
    \end{align}
\end{subequations}

To better understand the above methods, we introduce the following algorithm for efficient computation.
\begin{algorithm}[htbp]
\caption{SAV-MD-BDFk methods}
\label{alg:3}
\begin{algorithmic}[1]
\REQUIRE The evolving surface over the preceding $k$ time intervals: $\vec{X}^m, \cdots, \vec{X}^{m-k+1}$; The approximating surface of $\mS(t_{m+1})$: $\vec{\widetilde{X}}^{m+1}$; The modified  energy at time $t_m$:  $R^m$.

\STATE Substituting $\vec{\widetilde{X}}^{m+1}, \vec{\chi}^m, \cdots, \vec{\chi}^{m-k+1}$ to \eqref{eqn:SAV_1}-\eqref{eqn:SAV_3}, we firstly get $\vec{\overline{\chi}}^{m+1},  \overline{\mu}^{m+1}, 
\overline{\kappa}^{m+1}$. 
\STATE $\vec{\overline{X}}^{m+1}:\mS^0 \to \overline{\mS}^{m+1}$ is obtained through \eqref{eqn:SAV_7}.
\STATE Observing from \eqref{eqn:SAV_4}, $R^{m+1}$ is calculated by
    \begin{equation}\label{eqn:R}
        R^{m+1} = \frac{R^mW_c(\vec{\overline{X}}^{m+1})}{W_c(\vec{\overline{X}}^{m+1})+\Delta t \left(\nabla_{\mS} \overline{\mu}^{m+1},\nabla_{\mS}\overline{\mu}^{m+1}\right)_{\overline{\mS}^{m+1}}^h}.
    \end{equation}
   Then $\zeta^{m+1}$ is obtained using \eqref{eqn:SAV_4}.
\STATE According to \eqref{eqn:SAV_5} and \eqref{eqn:SAV_6}, we get $\eta^{m+1}$, $\vec{X}^{m+1}$, $\mu^{m+1}$ and $\kappa^{m+1}$.
\end{algorithmic}
\end{algorithm}
\begin{rem}\label{rem:zeta_mean}
The introduction of \( R(t) \) results in an unconditionally stable energy \( R^m \) for the SAV-MD-BDFk methods, and the stability of this energy is independent of the BDFk schemes. Moreover, \( \zeta^m \) serves two main purposes: 1) It helps prove the boundedness of the original energy \( W_c(\vec{X}^m) \) (see Theorem 3.2); 2) It plays a crucial role in the construction of a high-order temporal scheme, as explained below:
\begin{itemize}
    \item From \eqref{eqn:SAV_4}, it can be observed that since the first-order Euler method is used to discretize \( R(t) \), \( R^m \) is a first-order approximation of \( W_c(\vec{\overline{X}}^m) \). Therefore, \( \zeta^m \) is a first-order approximation of 1.
    \item To preserve high-order temporal accuracy, the value of \( r \) must be chosen appropriately. By multiplying \eqref{eqn:SAV_1} by \( \eta^{m+1} \), we have
\begin{equation}\label{eqn:eta_multi}
\left(\frac{\alpha_k\vec{\chi}^{m+1}-\eta^{m+1}A_k(\vec{\chi}^m)}{\Delta t}\cdot \vec{\widetilde{v}}^{m+1}, \varphi\right)_{\widetilde{\mS}^{m+1}}^h + (\nabla_\mS \mu^{m+1},\nabla_\mS \varphi)_{\widetilde{\mS}^{m+1}}^h = 0.
\end{equation}
Then, noting that $\eta^{m+1}$ is a $r$-order approximation to 1 according to \eqref{eqn:SAV_5}, we can directly obtain
\begin{equation}
\left(\frac{\alpha_k\vec{\chi}^{m+1}-A_k(\vec{\chi}^m)}{\Delta t}\cdot \vec{\widetilde{v}}^{m+1}, \varphi\right)_{\widetilde{\mS}^{m+1}}^h + (\nabla_\mS \mu^{m+1},\nabla_\mS \varphi)_{\widetilde{\mS}^{m+1}}^h = O(\Delta t^{r-1}).
\end{equation}
Therefore, the temporal accuracy of the BDFk schemes is unaffected as long as \( r \geq k+1 \).
\end{itemize}
\end{rem}
Similar to the VC-MD-BDF1 method, we can also develop a first-order approximate volume-conserving scheme for the SAV method (denoted by VC-SAV-MD-BDF1 method): given the initial closed surface $\mS^0:=\vec{X}^0(\cdot)\in [\mathbb{K}(\mS^0)]^3$, find the solution $(\vec{\overline{X}}^{m+1},\vec{\overline{\chi}}^{m+1},\overline{\mu}^{m+1},\overline{\kappa}^{m+1})\in [\mathbb{K}(\mS^0)]^3 \times [\mathbb{K}(\mS^m)]^3 \times \mathbb{K}(\mS^m) \times \mathbb{K}(\mS^m)$, and $\left(\vec{X}^{m+1},\mu^{m+1},\kappa^{m+1}\right)\in [\mathbb{K}(\mS^0)]^3 \times \mathbb{K}(\mS^m) \times \mathbb{K}(\mS^m)$, such that
\begin{subequations} \label{eqn:VCSAV_scheme}
    \begin{align}
      \label{eqn:VCSAV_1}&\left(\frac{\vec{\overline{\chi}}^{m+1}-\vec{\text{id}}}{\Delta t}\cdot \vec{n}^{m+\frac{1}{2}}, \varphi\right)_{\mS^{m}}^h + (\nabla_\mS \overline{\mu}^{m+1},\nabla_\mS \varphi)_{\mS^{m}}^h = 0, \qquad \forall \varphi \in \mathbb{K}(\mS^{m}),\\ \label{eqn:VCSAV_2}
         &\left(\overline{\mu}^{m+1}, \phi \right)_{\mS^m}^h-\left(\mat{Z}_k \left(\vec{n}^{m}\right)\nabla_{\mS}\vec{\overline{\chi}}^{m+1},\nabla_\mS \left(\vec{v}^{m}\phi\right)\right)_{\mS^{m}}^h = 0, \qquad \forall \phi \in \mathbb{K}(\mS^{m}),\\ \label{eqn:VCSAV_3}
         &\left(\overline{\kappa}^{m+1}, \vec{v}^{m}\cdot \vec{\omega}\sqrt{\text{det}\left( \mat J^{m}\right)}\right)_{\mS^{m}}^h+\left(\nabla_{\mS}\vec{\overline{\chi}}^{m+1}\left(\mat J^{m}\right)^{-1},\nabla_{\mS}\vec{\omega}\sqrt{\text{det}\left(\mat J^{m}\right)}\right)_{\mS^{m}}^h = 0, \qquad \forall \vec{\omega}\in [\mathbb{K}(\mS^{m})]^3,
         \\ & \label{eqn:VCSAV_4} \vec{\overline{X}}^{m+1} = \vec{\overline{\chi}}^{m+1}\circ \vec{X}^{m},
         \\& \label{eqn:VCSAV_5}
    \frac{R^{m+1}-R^m}{\Delta t} = -\zeta^{m+1}(\nabla_\mS \overline{\mu}^{m+1},\nabla_\mS \overline{\mu}^{m+1})_{\overline{\mS}^{m+1}}^h \quad \text{with} \quad \zeta^{m+1}=\frac{R^{m+1}}{W_c\left(\vec{\overline{X}}^{m+1}\right)}, \\& \label{eqn:VCSAV_6}
    \eta^{m+1} = 1-(1-\zeta^{m+1})^r, \\ & \label{eqn:VCSAV_7}
    \vec{X}^{m+1} = \eta^{m+1}\vec{\overline{X}}^{m+1}, \quad \mu^{m+1}=\eta^{m+1}\overline{\mu}^{m+1} \quad \text{and} \quad \kappa^{m+1}=\eta^{m+1}\overline{\kappa}^{m+1}.
    \end{align}
\end{subequations}

\subsection{Some properties of SAV methods}
For the SAV-MD-BDFk methods and the VC-SAV-MD-BDF1 method, we can obtain the following stability result and demonstrate the boundedness of the original energy. 
\begin{thm}\label{thm:ener_stable}
    For the SAV-MD-BDFk methods and the VC-SAV-MD-BDF1 method, given the energy $R^m\ge0$,  there hold
    \begin{itemize}
        \item [(i)]  The energy $R^{m+1}$ and 
        $\zeta^{m+1}$ are non-negative, i.e.,
        \begin{align}
            R^{m+1}\ge0,\qquad   \zeta^{m+1}\ge0, \qquad \forall m\geq 0.
        \end{align}
        \item [(ii)] The energy $R^{m+1}$ is unconditionally stable in the sense of 
        \begin{equation}\label{eqn:stable}
            R^{m+1}-R^{m}=-\Delta t \zeta^{m+1}\left(\nabla_{\mS}\overline{\mu},\nabla_{\mS}\overline{\mu}\right)_{\overline{\mS}^{m+1}} \le 0,\qquad \forall m\geq 0.
        \end{equation}
        \item [(iii)] The original energy $W(\vec{X}^m)$ is bounded, i.e., 
        \begin{equation}\label{eqn:bounded}
            W\left(\vec{X}^m\right)\le M_r^2,\qquad \forall m\geq 0, 
        \end{equation}
        where $M_r$ is a positive constant that  depends on $r$. 
    \end{itemize}
\end{thm}
\begin{proof}
From \eqref{eqn:R} and \eqref{eqn:SAV_4}, we can directly get $R^{m+1}\ge0$. In addition, $\zeta^{m+1}$ can be formulated as
\begin{equation} \label{eqn:zeta}
    \zeta^{m+1} = \frac{R^{m+1}}{W_c\left(\vec{\overline{X}}^{m+1}\right)} =\frac{R^m}{W_c\left(\vec{\overline{X}}^{m+1}\right)+\Delta t \left(\nabla_{\mS} \overline{\mu}^{m+1},\nabla_{\mS}\overline{\mu}^{m+1}\right)_{\overline{\mS}^{m+1}}^h},
\end{equation}
 which means that $\zeta^{m+1}\ge 0$. Observing from \eqref{eqn:SAV_4}, \eqref{eqn:stable} naturally holds.

 Due to \eqref{eqn:stable}, we have $R^{m+1}\le R^{m} \le \cdots \le R^0$. Let's denote $M = R^0$, which implies that $R^m \le M$. Without loss of generality, we assume $C\geq 1$. Then, it can be conclude from \eqref{eqn:zeta} that
\begin{equation}
    0 \le \zeta^{m+1}\le \frac{M}{W_c\left(\vec{\overline{X}}^{m+1}\right)} \le \frac{M}{W\left(\vec{\overline{X}}^{m+1}\right)+1}.
\end{equation}
On the definition of $\eta^{m+1}$, it is a $r$-th degree polynomial in regard to $\zeta^{m+1}$. Hence, we can further obtain the boundedness for $\eta^{m+1}$:
\begin{equation}\label{eqn:bound_eta}
    \left|\eta^{m+1}\right| = \left|P_{r-1}(\zeta^{m+1})\zeta^{m+1}\right|\le \frac{M_r}{W\left(\vec{\overline{X}}^{m+1}\right)+1},
\end{equation}
where $M_r$ is a constant that depends on $r$. According to \eqref{eqn:n_relate}, \eqref{eqn:area_relate}, \eqref{eqn:bound_eta} and \eqref{eqn:disc_ener}, we get
\begin{equation}
    W\left(\vec{X}^{m+1}\right)=(\eta^{m+1})^2 W\left(\vec{\overline{X}}^{m+1}\right)= \frac{M_r^2 W\left(\vec{\overline{X}}^{m+1}\right)}{\left(W\left(\vec{\overline{X}}^{m+1}\right)+1\right)^2} \le M_r^2.
\end{equation}
Hence, the boundedness of the original energy has been proven.
\end{proof}

\begin{rem}
    From the proof process mentioned above, it can be seen that the introduction of $\vec{\overline{X}}^{m+1}$ and $\zeta^{m+1}$ plays an important role in proving the boundedness of the energy. What's more, this property is unrelated to \eqref{eqn:SAV_1} and only involves \eqref{eqn:SAV_4}-\eqref{eqn:SAV_6}, which means the boundedness of the original energy holds for the SAV-MD-BDFk methods and the VC-SAV-MD-BDF1 method. More interestingly, the introduction of the SAV method does not alter the structure of the MD-BDFk methods, thus ensuring that the mesh quality is well preserved. 
\end{rem}

The VC-SAV-MD-BDF1 method does not strictly preserve the volume. From \eqref{eqn:VCSAV_1} to \eqref{eqn:VCSAV_4}, we can obtain $V\left(\vec{\overline{X}}^{m+1}\right)=V\left(\vec{X}^m\right)$. Based on this equation, the following approximate volume conservation result can be proven.
\begin{thm} \label{thm:VC-SAV-MD-BDF1}
Let $\left(\vec{X}^{m+1},\mu^{m+1},\kappa^{m+1}\right)$ be the numerical solution of the VC-SAV-MD-BDF1 method. Then, there holds 
\begin{equation}
    V\left(\vec{X}^{m+1}\right)-V\left(\vec{X}^m\right) = O(\Delta t^r), \qquad m \ge 0.
\end{equation}
\end{thm}
\begin{proof}
    From \eqref{eqn:VCSAV_7}, we can get the relationship between $|\sigma_l^{m+1}|$ and $|\overline{\sigma}_l^{m+1}|$:
\begin{equation}\label{eqn:area_relate}
    \left|\sigma_l^{m+1}\right| = (\eta^{m+1})^2\frac{\left|(\vec{\overline{q}}_{l_2}-\vec{\overline{q}}_{l_1})\times (\vec{\overline{q}}_{l_3}-\vec{\overline{q}}_{l_1})\right|}{2} = (\eta^{m+1})^2 \left|\overline{\sigma}_l^{m+1}\right|, \qquad \forall 1\le l \le L.
\end{equation}
Moreover, noticing that the definition of the element normal vector $\vec{n}_l^m$, we have 
\begin{equation}\label{eqn:n_relate}
    \vec{n}_l^{m+1} = \frac{(\eta^{m+1})^2(\vec{\overline{q}}_{l_2}-\vec{\overline{q}}_{l_1})\times (\vec{\overline{q}}_{l_3}-\vec{\overline{q}}_{l_1})}{(\eta^{m+1})^2|(\vec{\overline{q}}_{l_2}-\vec{\overline{q}}_{l_1})\times (\vec{\overline{q}}_{l_3}-\vec{\overline{q}}_{l_1})|}=\vec{\overline{n}}_l^{m+1}, \qquad \forall 1\le l \le L.
\end{equation}
Therefore, based on \eqref{eqn:area_relate} and \eqref{eqn:n_relate}, we can directly get 
\begin{equation}\label{eqn:vol_relate}
    V\left(\vec{X}^{m+1}\right) = (\eta^{m+1})^3V\left(\vec{\overline{X}}^{m+1}\right)=(\eta^{m+1})^3V\left(\vec{X}^{m}\right).
\end{equation}
Noting that $\eta^0 = 1$ and $\eta^m = 1+ O(\Delta t ^r), \forall m>0$, we obtain
\begin{align} \label{eqn:iden}
   \notag \lim_{\Delta t \to 0} \frac{V\left(\vec{X}^{m+1}\right)-V\left(\vec{X}^{m}\right)}{\Delta t^t} &= \lim_{\Delta t \to 0}\frac{\left(\left(\eta^{m+1}\right)^3-1\right)\prod_{i=1}^{m}\left(\eta^i\right)^3V\left(\vec{X}^0\right)}{\Delta t^r}\\ & \notag = \lim_{\Delta t \to 0}\frac{\left(\eta^{m+1}-1\right)\left(\left(\eta^{m+1}\right)^2+\left(\eta^{m+1}\right)+1\right)\prod_{i=1}^{m}\left(\eta^i\right)^3V\left(\vec{X}^0\right)}{\Delta t ^r}\\& = \lim_{\Delta t \to 0}\frac{3\left(\eta^{m+1}-1\right)V\left(\vec{X}^0\right)}{\Delta t^r} = O(1).
\end{align}
Overall, it is straightforward to derive from \eqref{eqn:iden} that $V\left(\vec{X}^{m+1}\right)-V\left(\vec{X}^m\right) = O(\Delta t^r)$.
\end{proof}

\begin{rem}
From Theorem \ref{thm:VC-SAV-MD-BDF1}, it can be observed that the VC-SAV-MD-BDF1 method approximately preserves volume conservation. Based on the volume difference between two adjacent time steps, we can either reduce the time step or increase the value of \( r \), where increasing \( r \) does not introduce additional computational cost. 
We can select a larger value for \( r \), such that the error between \( V(\vec{X}^{m}) \) and \( V(\vec{X}^{0}) \) is close to machine precision, making it virtually identical to a volume-preserving scheme.
\end{rem}

\begin{rem}
The SAV-MD-BDFk methods are linear schemes that can enhance computational efficiency in numerical experiments compared to some nonlinear schemes. The SAV-MD-BDFk methods not only preserve energy stability but also maintain very good mesh quality throughout the evolution process. However, except for the VC-SAV-MD-BDF1 method, which can achieve approximate volume conservation, the SAV-MD-BDFk methods still lack the property of volume conservation. 
In addition, the energy \( R^m \) can be approximately equal to the original energy plus a constant \( C \); however, it remains a modified form of the energy. 
To address above issues, in the next section we combine the LM approach with the MD formulation to construct the structure-preserving LM-MD-BDFk methods. Although the LM-MD-SAV-BDFk methods are nonlinear schemes, they can maintain volume conservation for \(1\leq k\leq 4\). Additionally, both the LM-MD-BDF1 and LM-MD-BDF2 methods can preserve energy stability.
\end{rem}

\section{The LM-MD-BDFk methods for anisotropic SDF} \label{sec:LM_MD}

In this section, we combine the LM approach with the MD formulation to construct a class of  structure-preserving PFEMs for anisotropic SDF.
We introduce two new Lagrange multipliers $\lambda(t)$ and $\varrho(t)$ to \eqref{eqn:model_MD1}, given by 
\begin{equation}\label{eqn:model_LM1}
    \left(\partial_t \vec{X}\circ \vec{X}^{-1}\right) \cdot \vec{n} = \Delta_\mS \mu + a_1\lambda(t) \mu + a_2 \varrho(t), \qquad \text{on} \quad \mS(t),
\end{equation}
where $a_1,a_2 \in \{0,1\}$. In order to ensure the equivalence of the two systems, it is necessary to enforce that \eqref{eqn:volume_diss} and \eqref{eqn:ener_diss} hold. Therefore, we consider the following system:
\begin{subequations} \label{eqn:model_LM}
    \begin{alignat}{3}
       \label{eqn:VPLM_MD1} &\left(\partial_t \vec{X}\circ \vec{X}^{-1}\right) \cdot \vec{n} = \Delta_\mS \mu+ a_1\lambda(t) \mu + a_2\varrho(t), \qquad &&\text{on} \quad \mS(t),\\ \label{eqn:VPLM_MD2}
      & \mu = -\nabla_\mS \cdot (\mat{Z}_k(\vec{n})\nabla_\mS \vec{\text{id}})\vec{n}, \qquad &&\text{on} \quad \mS(t), \\
     & \label{eqn:VPLM_MD3} \kappa \vec{n} = \Delta_{\mS^0} \vec{X}, \qquad &&\text{on} \quad \mS^0,\\ & \label{eqn:VPLM_MD4}
     a_1\frac{dW(t)}{dt}= -a_1(\nabla_\mS \mu, \nabla_\mS \mu)_{\mS(t)}, \qquad &&\text{on} \quad \mS(t), \\ & \label{eqn:VPLM_MD5}
     a_2\frac{dV(t)}{dt}=0, \qquad &&\text{on} \quad \mS(t). \end{alignat}
\end{subequations}

There are some remarks about this new continuous system:
\begin{itemize}
    \item Case 1: $a_1=1,a_2=0$: Multiplying \eqref{eqn:VPLM_MD1} by $\mu$, integrating it on surface $\mS(t)$, we get
    \begin{equation}\label{eqn:ES_equal}
       \lambda(t) \left
        (\mu,\mu\right)_{\mS(t)} = \frac{dW(t)}{dt}+(\nabla_{\mS}\mu,\nabla_{\mS}\mu)_{\mS(t)}=0.
    \end{equation}
    Therefore, the new system \eqref{eqn:model_LM} is consistent with the original system \eqref{eqn:new2MD_SDF} since $\mu \neq 0$.
    \item Case 2: $a_1=0,a_2=1$: Multiplying \eqref{eqn:VPLM_MD1} by $1$, integrating it on surface $\mS(t)$, we obtain
    \begin{equation}\label{eqn:VC_equal}
       \varrho(t) \left
        (1,1\right)_{\mS(t)} = \frac{dV(t)}{dt}=0.
    \end{equation}
    In this case, the new system \eqref{eqn:model_LM} aligns with the original system \eqref{eqn:new2MD_SDF} since $(1,1)_{\mS(t)} \neq 0$.
    \item Case 3: $a_1=1,a_2=1$: Multiplying \eqref{eqn:VPLM_MD1} by $\mu$ and $1$ respectively, integrating it on $\mS(t)$, we get 
\begin{subequations}\label{eqn:equiv_LM}
    \begin{align}
        &\lambda(t)(\mu,\mu)_{\mS(t)} + \varrho(t)(\mu,1)_{\mS(t)} = -(\nabla_s \mu,\nabla_s \mu)_{\mS(t)} + \frac{dW(t)}{dt} = 0,\\
      & \lambda(t)(\mu,1)_{\mS(t)} + \varrho(t)(1,1)_{\mS(t)} = \frac{dV(t)}{dt} = 0. 
    \end{align}
\end{subequations}
For this system of homogeneous linear equations, we can obtain the coefficient matrix
\begin{equation}
    A = 
\begin{pmatrix}
 \left(\mu,\mu\right)_{\mS(t)} & \left(\mu,1\right)_{\mS(t)} \\ \left(\mu,1\right)_{\mS(t)}
  & \left(1,1\right)_{\mS(t)}
\end{pmatrix}.
\end{equation}
Therefore, to ensure that $\lambda(t)=\varrho(t)=0$, we need to satisfy $\text{det}(A) = \left(\mu,\mu\right)_{\mS(t)}\left(1,1\right)_{\mS(t)}-\left(\mu,1\right)_{\mS(t)}^2\neq 0$. According to the Cauchy-Schwarz inequality, we can conclude that the new system \eqref{eqn:model_LM} is equivalent to the original system \eqref{eqn:new2MD_SDF}, if and only if $\mu$ is not a constant with respect to space variable.
\end{itemize}
\begin{rem}
    For Case 3, it effectively addresses the shortcoming of the MD formulation in structure-preserving algorithms. However, when evolving to the equilibrium state, we have 
    \begin{align}
        \frac{dW(t)}{dt}= -a_1(\nabla_\mS \mu, \nabla_\mS \mu)_{\mS(t)}\rightarrow 0, 
    \end{align}
    which implies that the value of $\mu$ will approach a constant. This might lead to instability in the discrete solution over long periods of evolution. In comparison, the equivalent conditions required for Case 1 and Case 2 are milder, which can adapt to long-term surface evolution. However, the numerical methods for Case 1 and Case 2 cannot simultaneously satisfy both volume conservation and energy stability. These issues will be addressed in the next section. 
\end{rem}

The weak formulation of \eqref{eqn:model_LM} is as follows: given the initial closed surface $\mS(0)\in [H^1(\mS^0)]^3$, find the evolution surface $\mS(t)=\vec{X}(\vec{\rho},t)\in [H^1(\mS^0)]^3$, the chemical potential $\mu \in H^1(\mS(t))$, the scalar function $\kappa \in L^2(\mS(t))$, $\lambda(t)\in \mathbb{R}$ and $\varrho(t)\in \mathbb{R}$ for $t>0$, such that 
\begin{subequations}\label{eqn:weak_LM}
    \begin{alignat}{3}
        \label{eqn:weak_LM1}&\left(\left(\partial_t \vec{X}\circ \vec{X}^{-1}\right)\cdot \vec{n}, \varphi\right)_{\mS(t)} + \left(\nabla_\mS \mu,\nabla_\mS \varphi\right)_{\mS(t)}- a_1 \lambda(t)\left(\mu,\varphi\right)- a_2\varrho(t)\left(1,\varphi\right) = 0, \qquad &&\forall \varphi \in H^1(\mS(t)), \\
    & \label{eqn:weak_LM2}\left(\mu, \phi\right)_{\mS(t)}-\left(\mat{Z}_k(\vec{n})\nabla_{\mS}\vec{\text{id}},\nabla_\mS (\vec{n}\phi)\right)_{\mS(t)} = 0, \qquad &&\forall \phi \in H^1(\mS(t)),\\ \label{eqn:weak_LM3}
    & \left(\kappa\vec{n}, \vec{\omega}\sqrt{\text{det}(\mat J)}\,\right)_{\mS(t)}+\left(\nabla_{\mS}\vec{\text{id}}\mat J^{-1},\nabla_{\mS}\vec{\omega}\sqrt{\text{det}(\mat J)}\,\right)_{\mS(t)} = 0, \qquad &&\forall \vec{\omega}\in [H^1(\mS(t))]^3,\\ & \label{eqn:weak_LM4}
     a_1\frac{dW(t)}{dt}= -a_1\left(\nabla_\mS \mu, \nabla_\mS \mu\right)_{\mS(t)}, \qquad &&\text{on} \quad \mS(t), \\ & \label{eqn:weak_LM5}
     a_2\frac{dV(t)}{dt}=0, \qquad &&\text{on} \quad \mS(t).
    \end{alignat}
\end{subequations}

By utilizing the BDFk schemes for time discretization and the PFEM for spatial discretization, we obtain the following fully discrete scheme (respectively denoted Case 1, Case 2, Case 3 as ES-LM-MD-BDFk, VC-LM-MD-BDFk and SP-LM-MD-BDFk methods): given $\vec{\chi}^m,\vec{\chi}^{m-1},\dots,\vec{\chi}^{m-k+1}\in [\mathbb{K}(\widetilde{\mS}^{m+1})]^3$, find  $(\vec{X}^{m+1},\vec{\chi}^{m+1},\mu^{m+1},\kappa^{m+1},\lambda^{m+1},\varrho^{m+1})\in [\mathbb{K}(\mS^0)]^3 \times [\mathbb{K}(\widetilde{\mS}^{m+1})]^3 \times \mathbb{K}(\widetilde{\mS}^{m+1}) \times \mathbb{K}(\widetilde{\mS}^{m+1})\times \mathbb{R}\times \mathbb{R}$, such that
\begin{subequations}\label{eqn:fullLM_MD}
    \begin{align}
      \label{eqn:fullLM_MD1}  &  \left(\frac{\alpha_k\vec{\chi}^{m+1}-A_k(\vec{\chi}^m)}{\Delta t}\cdot \vec{\widetilde{v}}^{m+1}, \varphi\right)_{\widetilde{\mS}^{m+1}}^h + (\nabla_\mS \mu^{m+1},\nabla_\mS \varphi)_{\widetilde{\mS}^{m+1}}^h- a_1 \lambda^{m+1}\left(\mu^{m+1},\varphi\right)_{\widetilde{\mS}^{m+1}}^h- a_2\varrho^{m+1}\left(1,\varphi\right)_{\widetilde{\mS}^{m+1}}^h = 0,~~~\forall \varphi \in \mathbb{K}(\widetilde{\mS}^{m+1}),\\ \label{eqn:fullLM_MD2}
         &(\mu^{m+1}, \phi)_{\mS^{m+1}}^h-\left(\mat{Z}_k(\vec{\widetilde{n}}^{m+1})\nabla_{\mS}\vec{\chi}^{m+1},\nabla_\mS \left(\vec{\widetilde{v}}^{m+1}\phi\right)\right)_{\widetilde{\mS}^{m+1}}^h = 0, \qquad \forall \phi \in \mathbb{K}(\widetilde{\mS}^{m+1}),\\ \label{eqn:fullLM_MD3} 
         &\left(\kappa^{m+1}, \vec{\widetilde{v}}^{m+1}\cdot \vec{\omega}\sqrt{\text{det}\left( \widetilde{\mat J}^{m+1}\right)}\right)_{\widetilde{\mS}^{m+1}}^h+\left(\nabla_{\mS}\vec{\chi}^{m+1}\left(\widetilde{\mat J}^{m+1}\right)^{-1},\nabla_{\mS}\vec{\omega} \sqrt{\text{det}\left( \widetilde{\mat J}^{m+1}\right)}\right)_{\widetilde{\mS}^{m+1}}^h = 0, \qquad \forall \vec{\omega}\in [\mathbb{K}(\widetilde{\mS}^{m+1})]^3,
         \\  \label{eqn:fullLM_MD4} &a_1\frac{\alpha_k W(\vec{X}^{m+1})-A_k(W(\vec{X}^m))}{\Delta t} = -a_1(\nabla_\mS \mu^{m+1},\nabla_\mS \mu^{m+1})_{\widetilde{\mS}^{m+1}}^h, \\ & \label{eqn:fullLM_MD5}
    a_2\left(\alpha_k V(\vec{X}^{m+1}) -A_k(V(\vec{X}^m))\right) = 0,\\ \label{eqn:fullLM_MD6}
         & \vec{X}^{m+1} = \vec{\chi}^{m+1} \circ \vec{\widetilde{X}}^{m+1}.
    \end{align}
\end{subequations}
\begin{thm}
    For the ES-LM-MD-BDFk $(k=1,2)$ methods and SP-LM-MD-BDFk methods $(k=1,2)$, the energy is stable in the sense that 
    \begin{equation}\label{ener_LM}
        W\left(\vec{X}^{m+1}\right) \le W\left(\vec{X}^m\right)\le \dots \le W\left(\vec{X}^0\right), \qquad \forall m \ge 0.
    \end{equation}
    Additionally, the VC-LM-MD-BDFk methods and SP-LM-MD-BDFk methods are volume-conservative, i.e.
    \begin{equation}\label{eqn:vol_LM_equal}
        V\left(\vec{X}^{m+1}\right)=V\left(\vec{X}^m\right)=\dots=V\left(\vec{X}^0\right), \qquad \forall m \ge 0.
    \end{equation}
\end{thm}
\begin{proof}
    For the ES-LM-MD-BDF1 method and SP-LM-MD-BDF1 method, according to \eqref{eqn:fullLM_MD4}, we can obtain 
\begin{equation}
         W\left(\vec{X}^{m+1}\right) -W\left(\vec{X}^m\right) = -\Delta t \left(\nabla_\mS \mu^{m+1},\nabla_\mS \mu^{m+1}\right)_{\widetilde{\mS}^{m+1}}^h\le 0, \qquad \forall m\ge 0.
\end{equation}
Therefore, the ES-LM-MD-BDF1 method and SP-LM-MD-BDF1 method satisfy the energy stablity. 

Moreover, for the ES-LM-MD-BDF2 and SP-LM-MD-BDF2 methods, we assume \(W\left(\vec{X}^{2}\right) \leq W\left(\vec{X}^{1}\right)\), as the corresponding BDF1 method is utilized in the first step. Furthermore, we assume \(W\left(\vec{X}^{m}\right) \leq W\left(\vec{X}^{m-1}\right)\) for \(m \geq 1\) and aim to prove that \(W\left(\vec{X}^{m+1}\right) \leq W\left(\vec{X}^{m}\right)\). 
Using \eqref{eqn:fullLM_MD4}, we have 
\begin{align}
    \frac{3}{2}W\left(\vec{X}^{m+1}\right)
    -2W\left(\vec{X}^{m}\right)
    +\frac12W\left(\vec{X}^{m-1}\right)
    =\frac{3}{2}\left[W\left(\vec{X}^{m+1}\right)
    -W\left(\vec{X}^{m}\right)\right]
    -\frac{1}{2}\left[W\left(\vec{X}^{m}\right)
    -W\left(\vec{X}^{m-1}\right)\right]\leq 0,
\end{align}
which implies 
\begin{align}
    \frac{3}{2}\left[W\left(\vec{X}^{m+1}\right)
    -W\left(\vec{X}^{m}\right)\right]
    \leq \frac{1}{2}\left[W\left(\vec{X}^{m}\right)
    -W\left(\vec{X}^{m-1}\right)\right]\leq 0. 
\end{align}
Therefore, by using the mathematical induction,  the ES-LM-MD-BDF2 method and SP-LM-MD-BDF2 method also keep the energy stability. 

Additionally, we observe that \(\alpha_k\) equals the sum of the coefficients of \(A_k\). Therefore, it is straightforward to prove that volume conservation holds for any \(k\).

\end{proof}
\begin{rem}
    The property of energy stability cannot be directly extended to ES-LM-MD-BDFk ($k=3, 4$) and SP-LM-MD-BDFk ($k=3, 4$), because it is not always possible to find $a \in \mathbb{R}^+$ that satisfies the following property:
\begin{equation}
    W\left(\vec{X}^{m+1}\right)-W\left(\vec{X}^m\right)\le a\left(W\left(\vec{X}^{m}\right)-W\left(\vec{X}^{m-1}\right)\right).
\end{equation}
\end{rem}

For the nonlinear system \eqref{eqn:fullLM_MD}, we can apply Newton's iteration method to solve it effectively. In the $n$-th iteration, given  $(\vec{\chi}^{m+1,n},\mu^{m+1,n}, \kappa^{m+1,n},\lambda^{m+1,n},\varrho^{m+1,n})\in [\mathbb{K}(\widetilde{X}^{m+1})]^3 \times \mathbb{K}(\widetilde{X}^{m+1}) \times \mathbb{K}(\widetilde{X}^{m+1}) \times \mathbb{R}  \times \mathbb{R}$, find the Newton direction $(\vec{\chi}^\delta,\mu^\delta, \kappa^\delta,\lambda^\delta,\varrho^\delta) \in [\mathbb{K}(\widetilde{X}^{m+1})]^3 \times \mathbb{K}(\widetilde{X}^{m+1}) \times \mathbb{K}(\widetilde{X}^{m+1}) \times \mathbb{R}  \times \mathbb{R}$, such that 
\begin{subequations}\label{eqn:Newton_iteration}
    \begin{align}
          & \notag \left(\frac{\alpha_k\vec{\chi}^\delta}{\Delta t}\cdot \vec{\widetilde{v}}^{m+1}, \varphi\right)_{\widetilde{\mS}^{m+1}}^h + (\nabla_\mS \mu^\delta,\nabla_\mS \varphi)_{\widetilde{\mS}^{m+1}}^h - \lambda^\delta \left(\mu^{m+1,n},\varphi \right)_{\widetilde{\mS}^{m+1}}^h- \lambda^{m+1,n} \left(\mu^\delta,\varphi \right)_{\widetilde{\mS}^{m+1}}^h -\varrho^\delta \left(1,\varphi \right)_{\widetilde{\mS}^{m+1}}^h = \\   & -\left(\frac{\alpha_k\vec{\chi}^{m+1,n}-A_k(\vec{\chi}^m)}{\Delta t}\cdot \vec{\widetilde{v}}^{m+1}, \varphi\right)_{\widetilde{\mS}^{m+1}}^h - (\nabla_\mS \mu^{m+1,n},\nabla_\mS \varphi)_{\widetilde{\mS}^{m+1}}^h + \lambda^{m+1,n}\left(\mu^{m+1,n},\varphi \right)_{\widetilde{\mS}^{m+1}}^h+\varrho^{m+1,n}\left(1,\varphi \right)_{\widetilde{\mS}^{m+1}}^h,\\ 
         &\left(\mu^\delta, \phi\right)_{\widetilde{\mS}^{m+1}}^h-\left(\mat{Z}_k(\vec{\widetilde{v}}^{m+1})\nabla_{\mS}\vec{\chi}^\delta,\nabla_\mS (\vec{\widetilde{v}}^{m+1}\phi)\right)_{\widetilde{\mS}^{m+1}}^h = -(\mu^{m+1,n}, \phi)_{\widetilde{\mS}^{m+1}}^h+\left(\mat{Z}_k(\vec{\widetilde{v}}^{m+1})\nabla_{\mS}\vec{\chi}^{m+1,n},\nabla_\mS (\vec{\widetilde{v}}^{m+1}\phi)\right)_{\widetilde{\mS}^{m+1}}^h,\\  \notag 
         &\left(\kappa^\delta, \vec{\widetilde{v}}^{m+1}\cdot \vec{\omega}\sqrt{\text{det}\left(\widetilde{J}^{m+1}\right)}\right)_{\widetilde{\mS}^{m+1}}^h+\left(\nabla_{\mS}\vec{\chi}^\delta(\widetilde{J}^{m+1})^{-1},\nabla_{\mS}\vec{\omega}\sqrt{\text{det}\left(\widetilde{J}^{m+1}\right)}\right)_{\widetilde{\mS}^{m+1}}^h \\  &= -\left(\kappa^{m+1,n}, \vec{\widetilde{v}}^{m+1}\cdot \vec{\omega}\sqrt{\text{det}\left(\widetilde{J}^{m+1}\right)}\right)_{\widetilde{\mS}^{m+1}}^h-\left(\nabla_{\mS}\vec{\chi}^{m+1,n}(\widetilde{J}^{m+1})^{-1},\nabla_{\mS}\vec{\omega}\sqrt{\text{det}\left(\widetilde{J}^{m+1}\right)}\right)_{\widetilde{\mS}^{m+1}}^h,\\&  
    a_1\left(\frac{\alpha_k \nabla_\mS \vec{\chi}^\delta}{\Delta t},\mat{Z}_k(\vec{n}^{m+1,n})\nabla_\mS \vec{\chi}^{m+1,n}\right)_{\mS^{m+1,n}}^h + 2a_1\left(\nabla_\mS \mu^\delta,\mu^{m+1,n}\right)_{\widetilde{\mS}^{m+1}}^h  \nn\\
    &=-a_1\frac{\alpha_k W(\vec{X}^{m+1,n})-A_k(W(\vec{X}^m))}{\Delta t}-a_1\left(\nabla_\mS \mu^{m+1,n},\nabla_\mS \mu^{m+1,n}\right)_{\widetilde{\mS}^{m+1}}^h , \label{LM_model_full4}\\ & \label{LM_model_full5}
a_2\alpha_k\left(\vec{\chi}^\delta,\vec{v}^{m+1,n}\right)_{\mS^{m+1,n}}^h= -a_2\left(\alpha_k V\left(\vec{X}^{m+1,n}\right)-A_k\left(V\left(\vec{X}^m\right)\right)\right), \end{align}
\end{subequations}
where $W(\vec{X}^{m+1,n})$ and $V(\vec{X}^{m+1,n})$ respectively denote the energy and volume of the closed surface $\mS^{m+1,n}$. The equation \eqref{LM_model_full4} and \eqref{LM_model_full5} are derived from the variation of the free energy and volume functional, such that 
\begin{subequations}
    \begin{align}
       & \frac{\delta W(\vec{X}(t))}{\delta \vec{X}}( \vec{X}^\delta) = \lim_{\epsilon \to 0}\frac{W(\vec{X}+\epsilon \vec{X}^\delta)-W(\vec{X})}{\epsilon} = -\left(\nabla_\mS \vec{\text{id}}^\delta, \mat{Z}_k(\vec{n}) \nabla_\mS \vec{\text{id}} \right)_{\mS(t)},\\ & \frac{\delta V(\vec{X}(t))}{\delta \vec{X}}( \vec{X}^\delta) = \lim_{\epsilon \to 0}\frac{V(\vec{X}+\epsilon \vec{X}^\delta)-V(\vec{X})}{\epsilon} =(\vec{\chi}^\delta \cdot \vec{n},1)_{\mS(t)}.
    \end{align}
\end{subequations}
Then, once getting $(\vec{\chi}^\delta,\mu^\delta, \kappa^\delta,\lambda^\delta,\varrho^\delta)$, we can update the iteration via 
\begin{subequations}
    \begin{align}
        \notag \vec{X}^{m+1,n+1}&=\left(\vec{\chi}^{m+1,n}+\vec{\chi}^\delta\right)\circ \vec{\widetilde{X}}^{m+1} , \quad \mu^{m+1,n+1}=\mu^{m+1,n}+\mu^\delta, \quad  \kappa^{m+1,n+1}=\kappa^{m+1,n}+\kappa^\delta,
        \\ & \lambda^{m+1,n+1} =\lambda^{m+1,n}+\lambda^\delta, \qquad \varrho^{m+1,n+1}=\varrho^{m+1,n}+\varrho^\delta.
        \notag
    \end{align}
\end{subequations}

For each $m\ge 0$, we choose the initial guess as
\begin{equation} \notag
    \vec{\chi}^{m+1,0}=\vec{\chi}^m, \quad \mu^{m+1,0}=\mu^m,\quad \kappa^{m+1,0}=\kappa^m, \quad \lambda^{m+1,0}=\lambda^m, \quad \varrho^{m+1,0}=\varrho^m,
\end{equation}
and repeat the iteration \eqref{eqn:Newton_iteration} until the following conditions hold 
\begin{equation}
    \max\bigg\{||\vec{\chi}^\delta||_{L^\infty},||\mu^\delta||_{L^\infty},||\kappa^\delta||_{L^\infty},|\lambda^\delta|,|\varrho^\delta|\bigg\} \le \text{tol},
\end{equation}
where $\text{tol}$ is the chosen tolerance. Hence, we can obtain the solution of the nonlinear system $\eqref{eqn:fullLM_MD}$ that
\begin{equation} \notag
    \vec{X}^{m+1}=\vec{X}^{m+1,n+1}, \quad \mu^{m+1}=\mu^{m+1,n+1},\quad \kappa^{m+1}=\kappa^{m+1,n+1}, \quad \lambda^{m+1}=\lambda^{m+1,n+1}, \quad \varrho^{m+1}=\varrho^{m+1,n+1}.
\end{equation}

\begin{rem}
  When simulating the SDF using SP-LM-MD-BDFk methods, it is necessary to determine the condition for reaching equilibrium; otherwise, the Newton iteration of the scheme may fail to converge. In contrast, the VC-LM-MD-BDFk and ES-LM-MD-BDFk methods are more suitable for long-term evolution. 
\end{rem}

\begin{rem}
    We summarize the methods provided in the previous two sections and highlight the current issues:
    \begin{itemize}
        \item For the SAV-MD-BDFk methods, energy stability can be achieved for \(1 \leq k \leq 4\), but none of these methods, except for the VC-SAV-MD-BDF1 method, ensure volume conservation.
        \item For the LM-MD-BDFk methods that include three cases: ES-LM-MD-BDFk methods, VC-LM-MD-BDFk methods and SP-LM-MD-BDFk methods. The SP-LM-MD-BDFk methods can preserve both the volume conservation and energy stability, but will be instable when reaching equilibrium. The ES-LM-MD-BDFk methods and VC-LM-MD-BDFk methods can each preserve only one type of structure. Meanwhile, for the ES-LM-MD-BDFk methods, the energy stability only holds for $k=1$ and $k=2$. 
    \end{itemize}
To address the issues above and fully leverage the advantages of both methods, we introduce the improved LM-SAV-MD-BDFk methods in the next section. These methods can simultaneously achieve approximate volume conservation and energy stability for \(1\leq k\leq 4\), and they also maintain stability when reaching equilibrium.
\end{rem}

\section{The LM-SAV-MD-BDFk methods for anisotropic SDF}\label{sec:LM_SAV_MD}
In this section, we construct a high-order structure-preserving LM-SAV-MD-BDFk $(1\le k \le 4)$ methods by applying SAV approach to VC-LM-MD-BDFk methods. By taking $a_1 = 0$ and $a_2 = 1$ in system \eqref{eqn:model_LM}, and then combining the SAV approach, we can obtain the following new formulation:
\begin{subequations} \label{eqn:SAV_LM_method}
    \begin{align}
        &\left(\partial_t \vec{X}\circ \vec{X}^{-1}\right) \cdot \vec{n} = \Delta_\mS \mu+\varrho(t), \qquad &&\text{on} \quad \mS(t),\\ 
      & \mu = -\nabla_\mS \cdot \left(\mat{Z}_k(\vec{n})\nabla_\mS \vec{\text{id}}\right)\vec{n}, \qquad &&\text{on} \quad \mS(t), \\
     &  \kappa \vec{n} = \Delta_{\mS^0} \vec{X}, \qquad &&\text{on} \quad \mS^0,\\ &
     \frac{dV(t)}{dt}=0, \qquad &&\text{on} \quad \mS(t), \\ &
     \frac{dR(t)}{dt}=-\zeta(t)\left(\nabla_\mS \mu, \nabla_\mS \mu \right)_{\mS(t)},\quad  &&\text{on} \quad \mS(t).
    \end{align}
\end{subequations}

According to the formulation \eqref{eqn:SAV_LM_method}, we can obtain the LM-SAV-MD-BDFk methods by using the BDFk $(1\le k \le 4)$ schemes. The fully discrete scheme is as follows: given \(\vec{\chi}^m, \vec{\chi}^{m-1}, \dots, \vec{\chi}^{m-k+1} \in [\mathbb{K}(\widetilde{\mS}^{m+1})]^3\) and \(R^m\), find the solution \(\left(\vec{\overline{X}}^{m+1}, \vec{\overline{\chi}}^{m+1}, \overline{\mu}^{m+1}, \overline{\kappa}^{m+1},\overline{\varrho}^{m+1}\right) \in [\mathbb{K}(\mS^0)]^3 \times [\mathbb{K}(\widetilde{\mS}^{m+1})]^3 \times \mathbb{K}(\widetilde{\mS}^{m+1}) \times \mathbb{K}(\widetilde{\mS}^{m+1})\)\(\times \mathbb{R}\), \(\zeta^{m+1}\), \(R^{m+1}\), and \(\left(\vec{X}^{m+1}, \mu^{m+1}, \kappa^{m+1},\varrho^{m+1}\right) \in [\mathbb{K}(\mS^0)]^3 \times \mathbb{K}(\widetilde{\mS}^{m+1}) \times \mathbb{K}(\widetilde{\mS}^{m+1})\)\(\times \mathbb{R}\), such that
\begin{subequations}\label{eqn:fullLM_SAV_MD}
  \begin{align}
   &\left(\frac{\alpha_k\vec{\overline{\chi}}^{m+1}-A_k(\vec{\chi}^m)}{\Delta t}\cdot \vec{\widetilde{v}}^{m+1}, \varphi\right)_{\widetilde{\mS}^{m+1}}^h + (\nabla_\mS \overline{\mu}^{m+1},\nabla_\mS \varphi)_{\widetilde{\mS}^{m+1}}^h-\overline{\varrho}^{m+1}\left(1,\varphi \right)_{\widetilde{\mS}^{m+1}}^h = 0, \qquad \forall \varphi \in \mathbb{K}(\widetilde{\mS}^{m+1}),\\ &(\overline{\mu}^{m+1}, \phi)_{\widetilde{\mS}^{m+1}}^h-\left(\mat{Z}_k(\vec{\widetilde{v}}^{m+1})\nabla_{\mS}\vec{\overline{\chi}}^{m+1},\nabla_\mS (\vec{\widetilde{v}}^{m+1}\phi)\right)_{\widetilde{\mS}^{m+1}}^h = 0, \qquad \forall \phi \in \mathbb{K}(\widetilde{\mS}^{m+1}),\\  
     &\left(\overline{\kappa}^{m+1}, \vec{\widetilde{v}}^{m+1}\cdot \vec{\omega}\sqrt{\text{det}\left( \widetilde{\mat J}^{m+1}\right)}\right)_{\widetilde{\mS}^{m+1}}^h+\left(\nabla_{\mS}\vec{\overline{\chi}}^{m+1}\left(\widetilde{J}^{m+1}\right)^{-1},\nabla_{\mS}\vec{\omega}\sqrt{\text{det}\left( \widetilde{\mat J}^{m+1}\right)}\right)_{\widetilde{\mS}^{m+1}}^h = 0, \qquad \forall \vec{\omega}\in [\mathbb{K}(\widetilde{\mS}^{m+1})]^3,\\&  \label{eqn:fullLM_SAV4_MD}
    \alpha_k V(\vec{\overline{X}}^{m+1}) -A_k(V(\vec{X}^m)) = 0,\\ &
    \vec{\overline{X}}^{m+1} = \vec{\overline{\chi}}^{m+1}\circ \vec{\widetilde{X}}^{m+1},
    \\ & \label{eqn:fullLM_SAV5_MD}
    \frac{R^{m+1}-R^m}{\Delta t} = -\zeta^{m+1}(\nabla_\mS \overline{\mu}^{m+1},\nabla_\mS \overline{\mu}^{m+1})_{\overline{\mS}^{m+1}}^h \quad \text{with} \quad \zeta^{m+1}=\frac{R^{m+1}}{W_c(\vec{\overline{X}}^{m+1})}, \\& \label{eqn:fullLM_SAV6_MD}
    \eta^{m+1} = 1-(1-\zeta^{m+1})^r, \\& \label{eqn:fullLM_SAV7_MD}
    \vec{X}^{m+1} = \eta^{m+1}\vec{\overline{X}}^{m+1}, \quad \mu^{m+1}=\eta^{m+1}\overline{\mu}^{m+1}, \quad \kappa^{m+1}=\eta^{m+1}\overline{\kappa}^{m+1}, \quad   \varrho^{m+1}=\eta^{m+1}\overline{\varrho}^{m+1}.
  \end{align}    
\end{subequations}

 Observing from \eqref{eqn:fullLM_SAV5_MD} and \eqref{eqn:fullLM_SAV6_MD}, it is evident that LM-SAV-MD-BDFk methods also satisfies the properties of SAV-MD-BDFk methods, namely the unconditional stability of the modified energy $R^m$ and the boundedness of the original energy $W(\vec{X}^m)$. In addition, it is approximately volume-conservative. 
\begin{thm} \label{thm:LM-SAV-MD-BDFk}
Let $\left(\vec{X}^{m+1},\mu^{m+1},\kappa^{m+1},\varrho^{m+1}\right)$ be the numerical solution of the LM-SAV-MD-BDFk methods, the following properties hold
\begin{itemize}
    \item[(i)]  The energy is unconditionally stable, i.e.,
\begin{equation}\label{ener_SAV_LM}
       R^{m+1}\leq R^m, \qquad \forall m \ge 0,
    \end{equation}
    and the original energy is bounded, i.e., 
    \begin{align}
        W(\vec X^m)\leq M_r^2,\qquad \forall m\geq 0. 
    \end{align}
    \item[(ii)] The volume is approximately conservative, i.e., 
    \begin{equation}\label{volume_SAV_LM}
    V\left(\vec{X}^{m+1}\right)-V\left(\vec{X}^m\right) = O(\Delta t^r), \qquad \forall m \ge 0.
\end{equation}
\end{itemize}
\end{thm}
\begin{proof}
    The energy stability and boundedness can be demonstrated similar to Theorem \ref{thm:ener_stable}.
    In addition, we use the mathematical induction to prove the approximate volume conservation property. For simplicity, we consider the case of $k=2$, as the proofs for other cases are analogous.
We assume that \(\vec{X}^1\) is calculated using the LM-SAV-MD-BDF1 method, which leads to the following relationship:
\begin{align}
    V\left(\vec{\overline{X}}^{1}\right)-V\left(\vec{X}^0\right)=0.
\end{align}
Then, based on \eqref{eqn:vol_relate}, we can derive
\begin{equation}\label{eqn:LM_vol1_rel}
    V\left(\vec{X}^1\right)-V\left(\vec{X}^0\right)=\left[\left(\eta^1\right)^3-1\right]V\left(\vec{\overline{X}}^1\right) + V\left(\vec{\overline{X}}^{1}\right)-V\left(\vec{X}^0\right) = \left[\left(\eta^1\right)^3-1\right]V\left(\vec{X}^0\right).
\end{equation}
Furthermore, since $\eta^1 = 1+ O(\Delta t ^r)$, and by using \eqref{eqn:LM_vol1_rel},  we can obtain 
\begin{equation}\label{eqn:lim_1}
    \lim_{\Delta t \to 0} \frac{V\left(\vec{X}^1\right)-V\left(\vec{X}^0\right)}{\Delta t^r} = \lim_{\Delta t \to 0}\frac{\left[\left(\eta^1\right)^3-1\right]V\left(\vec{X}^0\right)}{\Delta t ^r}= O(1),
\end{equation}
which implies that \eqref{volume_SAV_LM} holds when $m = 0$. For the LM-SAV-MD-BDF2 method, when $m=1$, we can derive from \eqref{eqn:fullLM_SAV4_MD} that
\begin{align}
  \notag
    V\left(\vec{X}^{2}\right)-V\left(\vec{X}^{1}\right) &= V\left(\vec{\overline{X}}^{2}\right)-V\left(\vec{X}^{1}\right)+\left[\left(\eta^2\right)^3-1\right]V\left(\vec{\overline{X}}^{2}\right)\\ & \notag =\frac{1}{3}\left(V\left(\vec{X}^{1}\right)-V\left(\vec{X}^{0}\right)\right) + 
    \left[\left(\eta^2\right)^3-1\right] \left(V\left(\vec{X}^{1}\right) +\frac{1}{3}\left(V\left(\vec{X}^{1}\right)-V\left(\vec{X}^{0}\right)\right) \right) \\ &
    = \frac{1}{3}\left[\left(\eta^1\right)^3-1\right]V\left(\vec{X}^0\right) + \left[\left(\eta^2\right)^3-1\right]\left(\left(\eta^1\right)^3V\left(\vec{X}^{0}\right) +\frac{1}{3}\left[\left(\eta^1\right)^3-1\right]V\left(\vec{X}^0\right) \right). 
\end{align}
Therefore, similar to \eqref{eqn:lim_1}, \eqref{volume_SAV_LM} holds for $m = 1$. 
Assume that \eqref{volume_SAV_LM} holds for $1\le m \le k-1$, $k\ge 2$. By the mathematical induction, we need to prove that it also holds for $m = k$. Due to \eqref{eqn:fullLM_SAV4_MD}, the following equation for LM-SAV-MD-BDF2 method holds 
\begin{equation}\label{eqn:vol_identity}
    V\left(\vec{\overline{X}}^{m+1}\right)-
    V\left(\vec{X}^{m}\right)
    =\frac{1}{3}\left(V\left(\vec{X}^{m}\right)-V\left(\vec{X}^{m-1}\right)\right), \qquad \forall m \ge 1.
\end{equation}
According to \eqref{eqn:vol_identity} and \eqref{volume_SAV_LM}, we can infer that
\begin{align}\label{eqn:vol_k+1}
\notag
    &V\left(\vec{X}^{k+1}\right)-V\left(\vec{X}^{k}\right)= V\left(\vec{\overline{X}}^{k+1}\right)-V\left(\vec{X}^{k}\right)+\left[\left(\eta^{k+1}\right)^3-1\right] V\left(\vec{\overline{X}}^{k+1}\right)\\
    & \notag =\frac{1}{3}\left(V\left(\vec{X}^{k}\right)-V\left(\vec{X}^{k-1}\right)\right) + \left[\left(\eta^{k+1}\right)^3-1\right]\left(\frac{1}{3}\left(V\left(\vec{X}^{k}\right)-V\left(\vec{X}^{k-1}\right)\right)+V\left(\vec{X}^{k}\right)\right)\\ 
    &=
    \frac{1}{3}\left(V\left(\vec{X}^{k}\right)-V\left(\vec{X}^{k-1}\right)\right) + \left[\left(\eta^{k+1}\right)^3-1\right]\left(\frac{1}{3}\left(V\left(\vec{X}^{k}\right)-V\left(\vec{X}^{k-1}\right)\right)+\sum_{i=1}^k\left(V\left(\vec{X}^{i}\right)-V\left(\vec{X}^{i-1}\right)\right)+V\left(\vec{X}^{0}\right)\right).
\end{align}
Hence, according to \eqref{eqn:vol_k+1}, it can be deduced that $V\left(\vec{X}^{k+1}\right)-V\left(\vec{X}^k\right) = O(\Delta t^r)$. By the mathematical induction, we conclude that  \eqref{volume_SAV_LM} holds for the LM-SAV-MD-BDF2 method. 
\end{proof}

\begin{rem}
For $1 \leq k \leq 4$, the LM-SAV-MD-BDFk methods exhibit approximate volume conservation and energy stability, and are able to address the instability issues encountered by the LM-MD-BDFk methods when evolving towards equilibrium. Although the LM-SAV-MD-BDFk methods show approximate volume conservation, according to Theorem \ref{thm:LM-SAV-MD-BDFk}, we can reduce the volume error by decreasing the time step $\Delta t$ or increasing the parameter $r$, with the latter not introducing any additional computational cost. When we choose a relatively large value for $r$, the volume error becomes comparable to that of the VC-LM-MD-BDFk methods, which are strictly volume-conservative.
\end{rem}

\section{Numerical experiments}\label{sec:numerical_tests}
In this section, we demonstrate the superiority of the proposed schemes from multiple perspectives, including convergence tests, structure-preserving tests, mesh quality comparisons, and surface evolution in surface diffusion with various anisotropic energies. 
The surface energy density functions considered in the experiments are classified into the following three categories:

\begin{itemize}
    \item The ellipsoidal anisotropic surface energy:
    \begin{equation}\label{eqn:Ellipse_fun} \gamma(\vec{n})=\sqrt{{a_1^2n_1^2+a_2^2n_2^2+a_3^2n_3^2}};
    \end{equation}
    \item The 3-fold anisotropic surface energy:
    \begin{equation}\label{eqn:3fold_fun} \gamma(\vec{n})=1+\beta \left(n_1^3+n_2^3+n_3^3\right);
    \end{equation}
    \item The 4-fold anisotropic surface energy:
    \begin{equation}\label{eqn:4fold_fun} \gamma(\vec{n})=1+\beta \left(n_1^4+n_2^4+n_3^4\right). 
    \end{equation}
\end{itemize}

\textbf{Example 1} (Convergence tests) 
In order to test the convergence of the MD-BDFk methods, we calculate a right-hand term $f$ to make the evolving surface a sphere with radius $r(t):=(1+t^3)^{\frac{1}{4}}$. By computing $(f^m,\varphi)_{\widetilde{\mS}^{m+1}}^h$ through \eqref{eqn:mass_lumped}, we add it to the right-hand side of the full discrete scheme \eqref{eqn:MDBDFk_scheme}. 
To test the temporal error, we triangulate the initial surface into 6067 vertices and 12130 elements, making the mesh size sufficiently fine. Subsequently, we use different time steps $\Delta t$ and calculate the error of the numerical solution at time $T$ as follows
\begin{equation}
    e_{\Delta t}(T) := \max_{k=1,2,\dots,K}\left|\left|\vec{q}_k^{T/\Delta t}\right|-r(T)\right|.
\end{equation}
The errors of the MD-BDF1 method, MD-BDF2 method and MD-BDF3 method are depicted in Figure \ref{fig:order_test}. It is evident that with different surface energy densities, the convergence orders of these two schemes align with our desired results.
\begin{figure}[!h]
         \centering
    \includegraphics[width=0.30\textwidth]{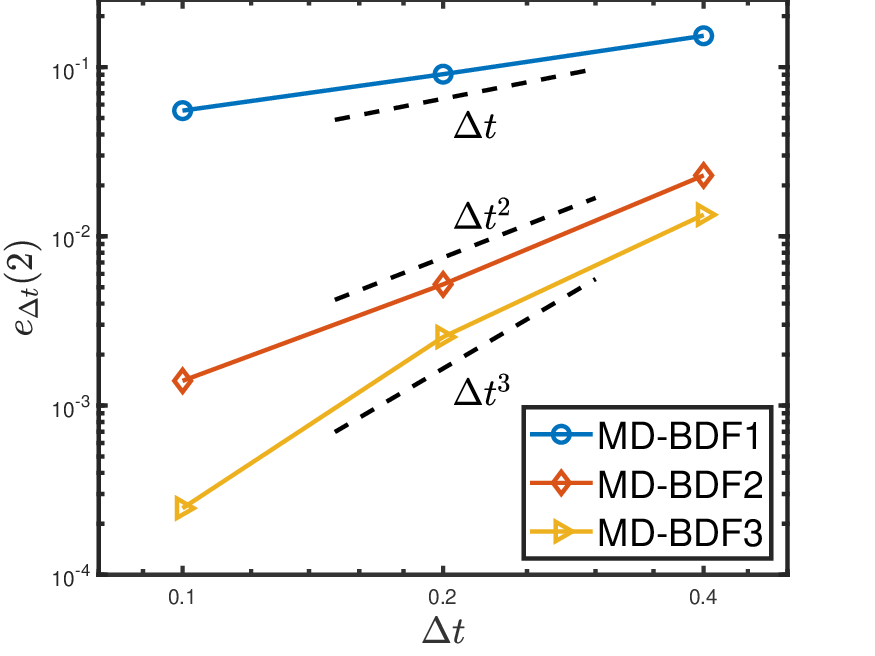}
    \includegraphics[width=0.30\textwidth]{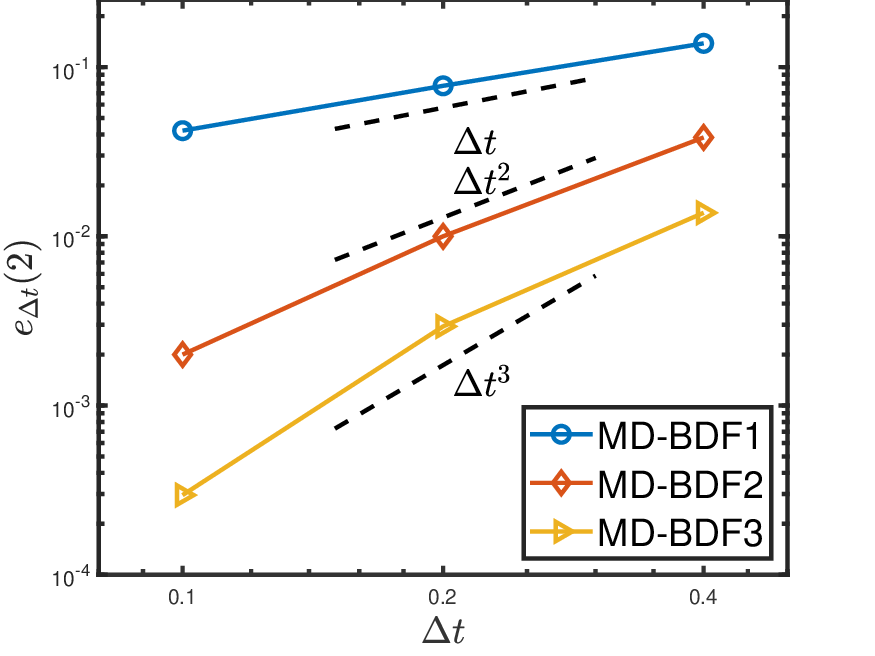}
    \includegraphics[width=0.30\textwidth]{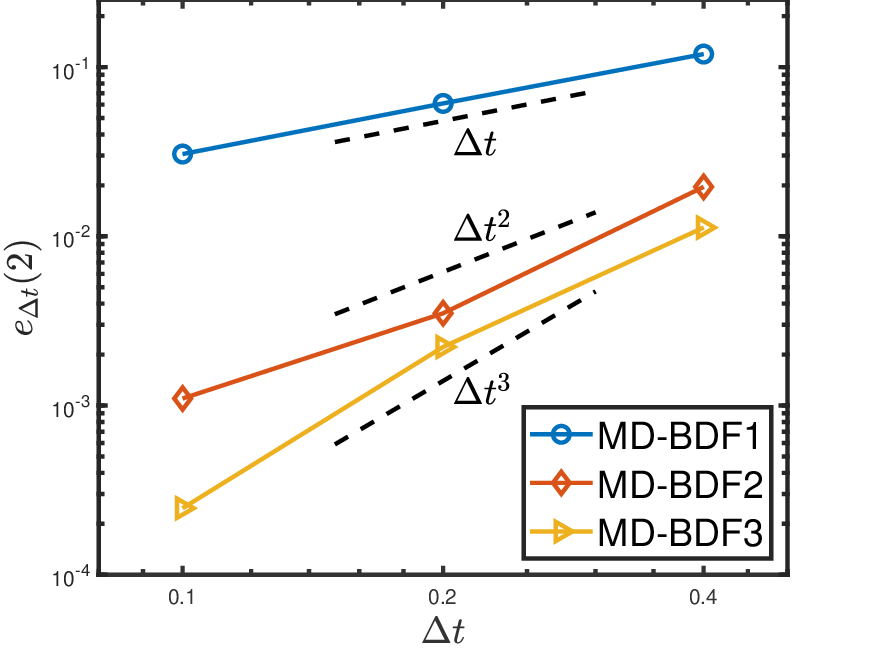}
    \caption{Convergence rates of MD-BDF1 method, MD-BDF2 method and MD-BDF3 method at times $T = 2$ with different surface energy densities: $\gamma(\vec{n})=\sqrt{{n_1^2+n_2^2+2n_3^2}}$, $\gamma(\vec{n})=1+0.125 \left(n_1^3+n_2^3+n_3^3\right)$ and $\gamma(\vec{n})=1+0.05 \left(n_1^4+n_2^4+n_3^4\right)$.}
    \label{fig:order_test}
    \end{figure}

\textbf{Example 2} (Volume conservation) 
In this experiment, we use the relative volume loss to measure the volume-conserving characteristics of the proposed methods. The initial surface is chosen as a $1 \times 1 \times 2$ ellipsoid with a mesh size $h = 1.5356\times 10^{-1}$. We define the relative volume loss $\Delta V(t)$ as
\begin{equation}\label{eqn:relate_loss}
    \Delta V(t)|_{t = t_m}:= \frac{V\left(\vec{X}^m\right)-V\left(\vec{X}^0\right)}{V\left(\vec{X}^0\right)}, \qquad \forall m \ge 0.
\end{equation}
Figure \ref{fig:vol_MDtest} illustrates the relative volume loss for the MD-BDFk methods and the VC-MD-BDF1 method. It can be observed that the VC-MD-BDF1 method exhibits a smaller volume loss throughout the evolution process, nearly achieving volume conservation. 
Figure \ref{fig:vol_SAVMDtest} shows the relative volume loss during the evolution process for the SAV-MD-BDFk methods and the VC-SAV-MD-BDF1 method.
We can observe that as \( r \) increases, the volume preservation capability of the VC-SAV-MD-BDF1 method during the evolution process is significantly improved.
Figures \ref{fig:vol_VCMDtest} and \ref{fig:vol_LMSAVMDtest} demonstrate the volume-preserving property of VC-LM-MD-BDFk methods and LM-SAV-MD-BDFk methods for $1\le k \le 4$. 
Figure \ref{fig:vol_LMSAVMDtest} demonstrates that when \( r \) is sufficiently large, the volume error of the LM-SAV-MD-BDFk methods is on the same order of magnitude as that of the VC-LM-MD-BDFk methods. Therefore, when \( r \) becomes large enough, this method effectively becomes a volume-preserving scheme, and increasing \( r \) does not introduce any additional computational cost.

\begin{figure}[!h]
         \centering
    \includegraphics[width=0.45\textwidth]{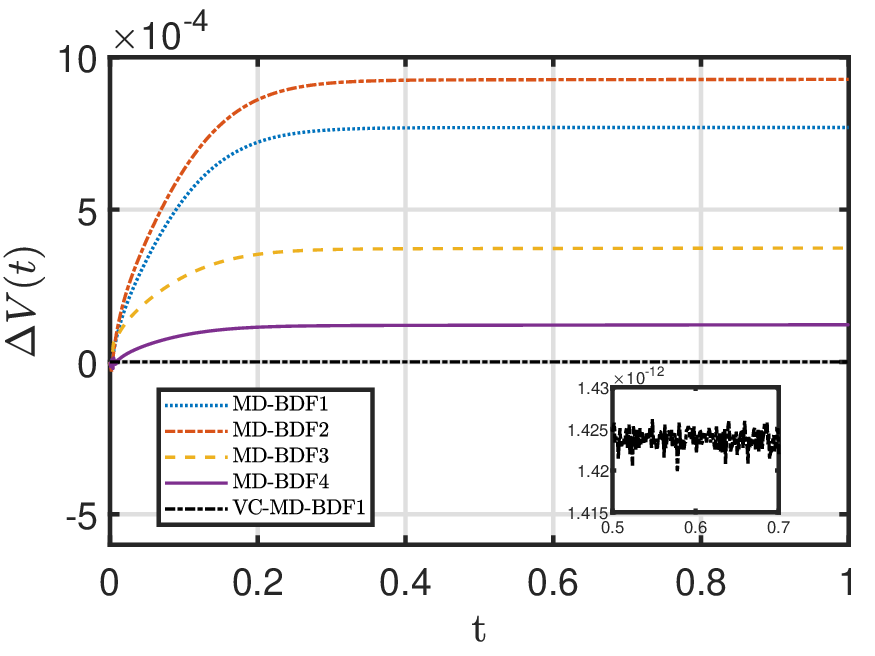}
    \includegraphics[width=0.45\textwidth]{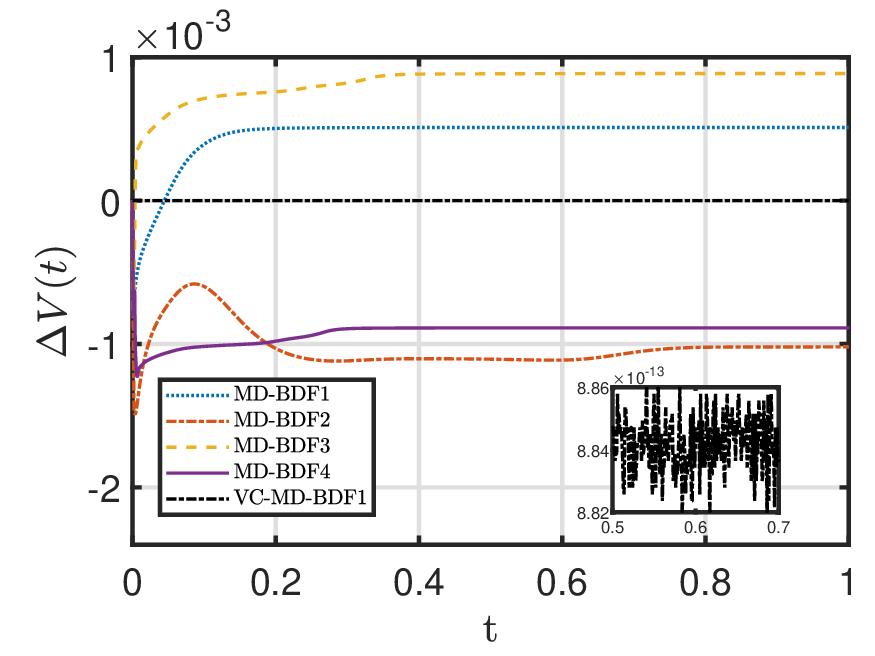}

    \caption{The relative volume loss of MD-BDFk methods under different surface energy densities: $\gamma(\vec{n})=1+0.05 \left(n_1^4+n_2^4+n_3^4\right)$ and $\gamma(\vec{n})=1+0.5 \left(n_1^4+n_2^4+n_3^4\right)$. Other parameters are chosen as $h = 1.5356\times 10^{-1}$ and $\Delta t = 10^{-3}$. }
    \label{fig:vol_MDtest}
    \end{figure}
\begin{figure}[!h]
         \centering
    \includegraphics[width=0.45\textwidth]{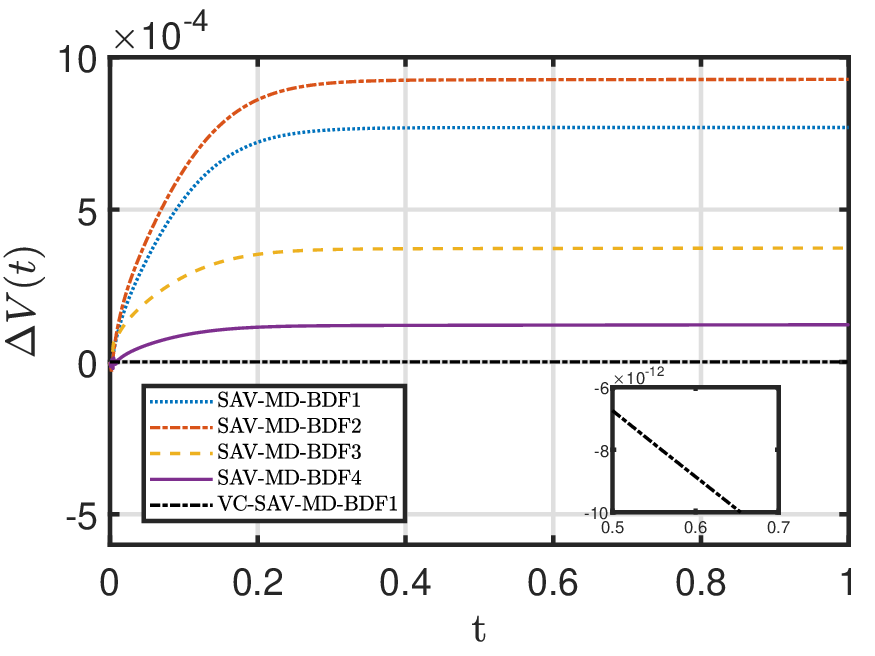}
    \includegraphics[width=0.45\textwidth]{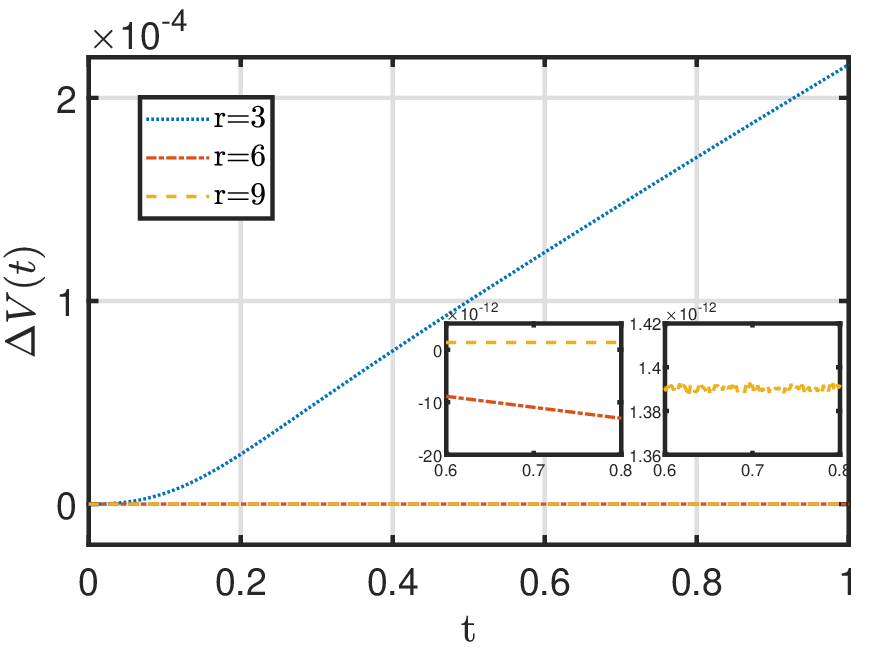}
    \includegraphics[width=0.45\textwidth]{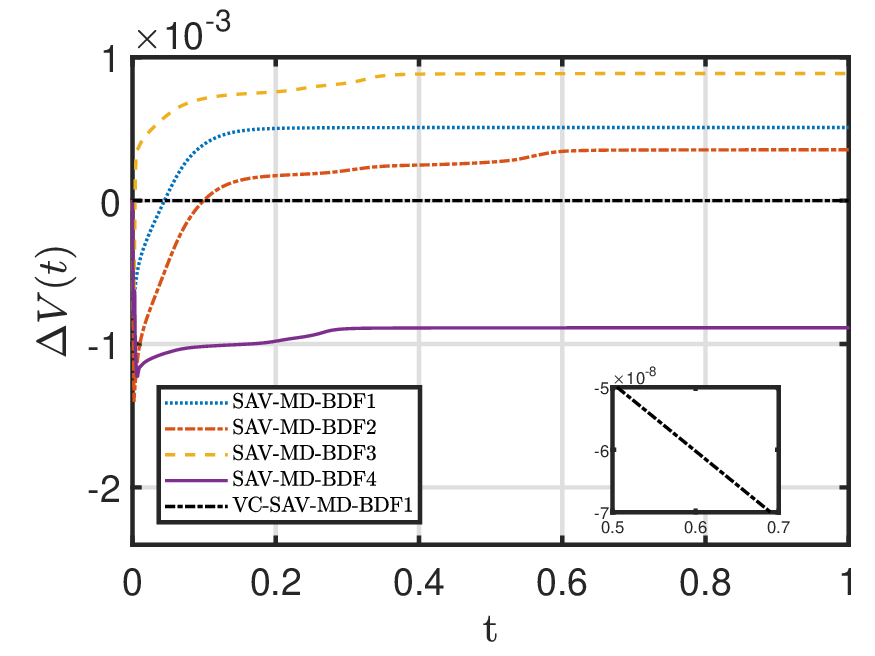}
    \includegraphics[width=0.45\textwidth]{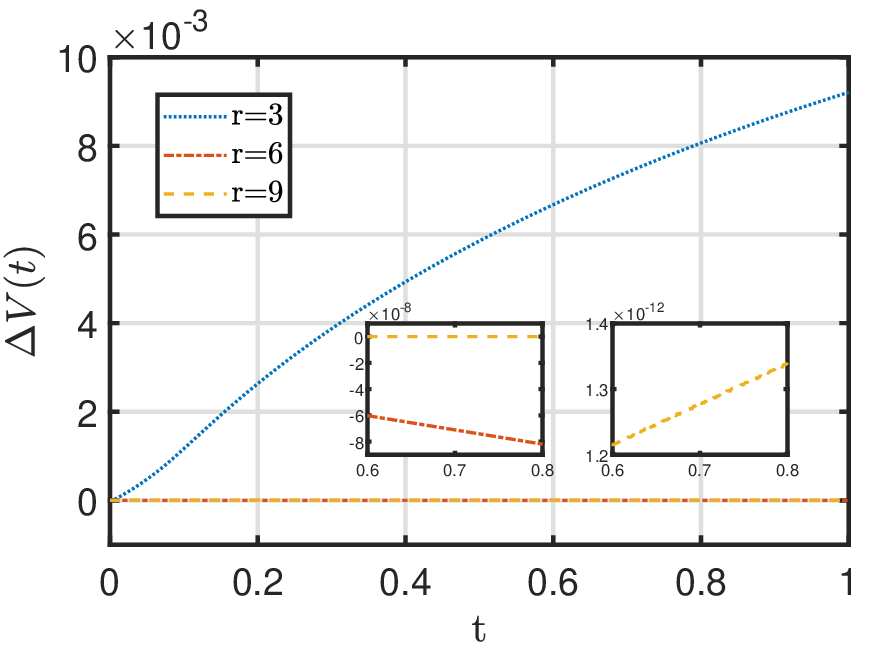}
    
    \caption{The relative volume loss of SAV-MD-BDFk methods under different surface energy densities: $\gamma(\vec{n})=1+0.05 \left(n_1^4+n_2^4+n_3^4\right)$ (the first row) and $\gamma(\vec{n})=1+0.5 \left(n_1^4+n_2^4+n_3^4\right)$ (the second row). The first column: $r=5$. Other parameters are chosen as $h = 1.5356\times 10^{-1}$ and $\Delta t = 10^{-3}$.}
    \label{fig:vol_SAVMDtest}
\end{figure}

 \begin{figure}[!h]
         \centering
    \includegraphics[width=0.45\textwidth]{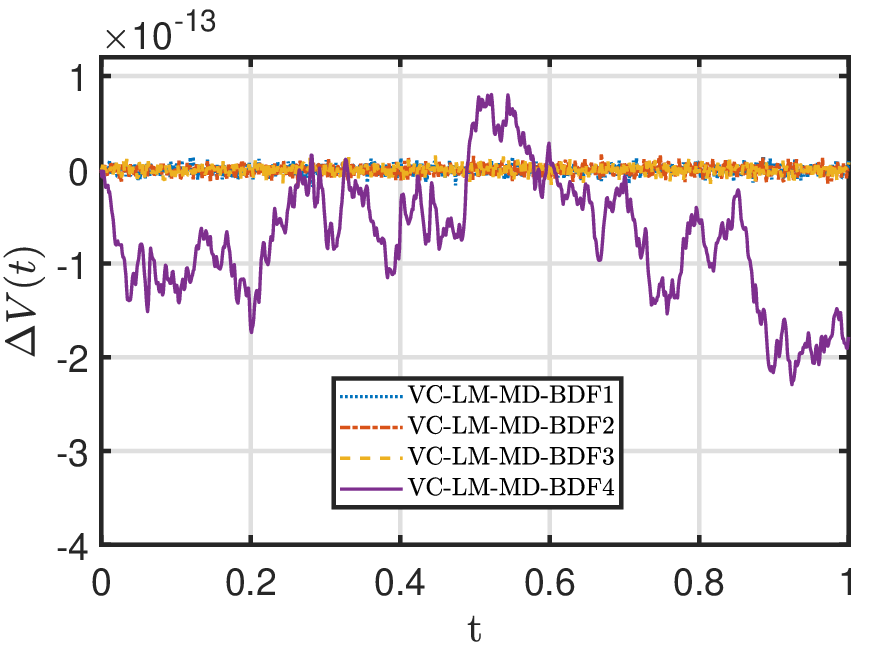}
    \includegraphics[width=0.45\textwidth]{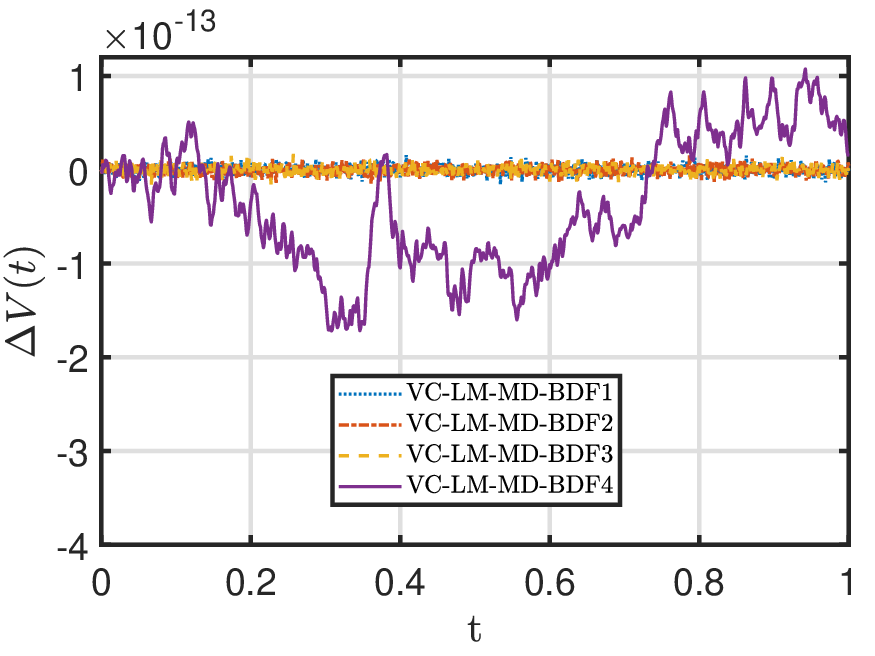}

    \caption{The relative volume loss of VC-LM-MD-BDFk methods under different surface energy densities: $\gamma(\vec{n})=1+0.05 \left(n_1^4+n_2^4+n_3^4\right)$ and $\gamma(\vec{n})=1+0.5 \left(n_1^4+n_2^4+n_3^4\right)$. Other parameters are chosen as $h = 1.5356\times 10^{-1}$ and $\Delta t = 10^{-3}$. }
    \label{fig:vol_VCMDtest}
    \end{figure}   

\begin{figure}[!h]
         \centering
    \includegraphics[width=0.45\textwidth]{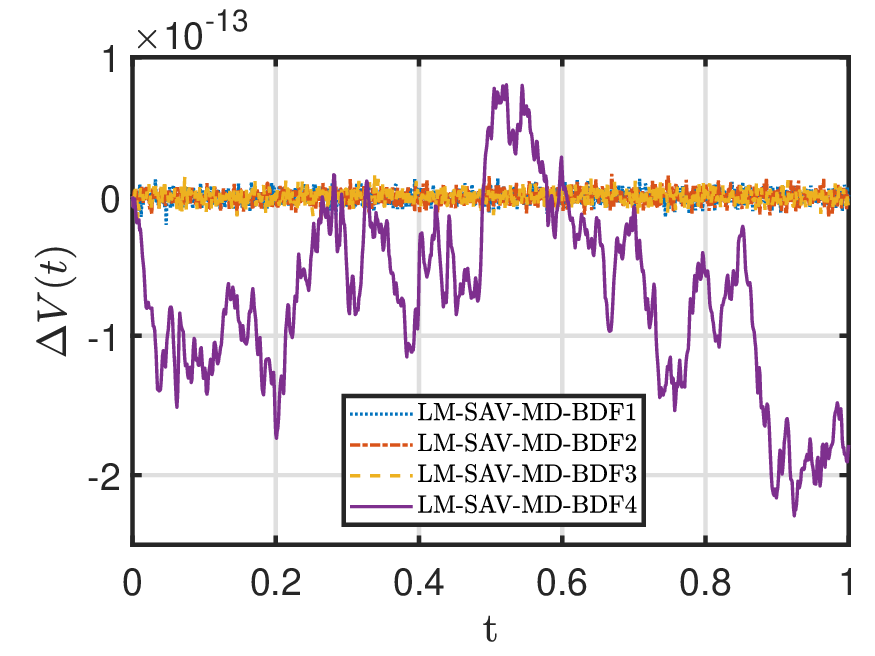}
    \includegraphics[width=0.45\textwidth]{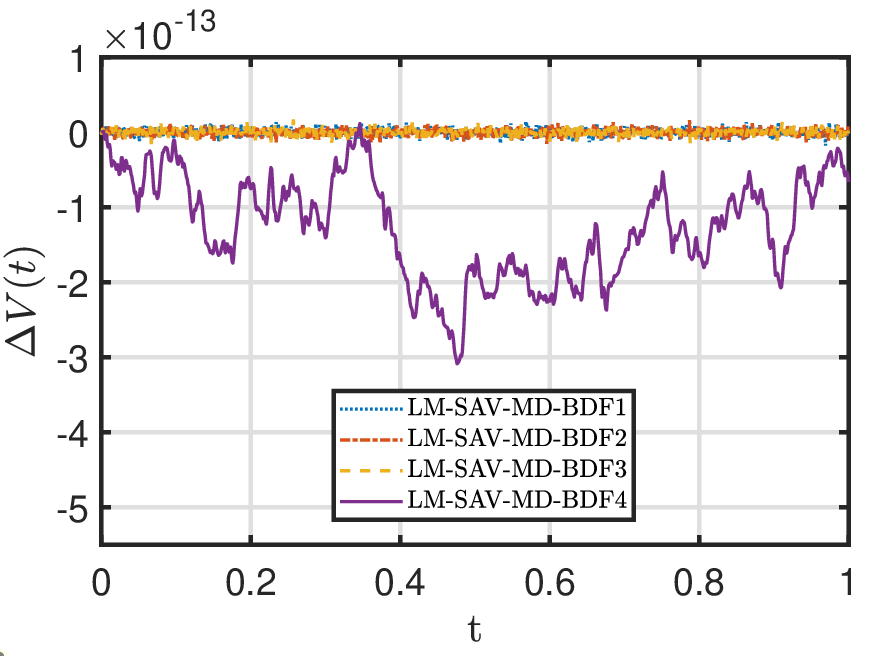}

    \caption{The relative volume loss of LM-SAV-MD-BDFk methods under different surface energy densities: $\gamma(\vec{n})=1+0.05 \left(n_1^4+n_2^4+n_3^4\right)$ and $\gamma(\vec{n})=1+0.5 \left(n_1^4+n_2^4+n_3^4\right)$. Other parameters are chosen as $h = 1.5356\times 10^{-1}$, $\Delta t = 10^{-3}$ and $r = 9$. }
    \label{fig:vol_LMSAVMDtest}
\end{figure}
    
\textbf{Example 3} (Energy stability) In this example, we continue to use the initial surface as in Example 2.  To verify the energy stability of the proposed methods, we use the normalized energy $W^h(t)/W^h(0)$, where $W^h(t)$ is defined as
\begin{equation}
    W^h(t)|_{t=t_m}:= W\left(\vec{X}^{m}\right),\qquad \forall m \ge 0.
\end{equation}
In addition, to verify the effectiveness of the SAV approach, we define the modified energy $R(t)$ and the auxiliary scalar $\zeta(t)$ as
\begin{equation}
    R(t)|_{t = t_m}:=R^m \qquad \text{and} \qquad \zeta(t)|_{t=t_m}:=\zeta^m, \qquad \forall m\ge 0.
\end{equation}
In this example, we conduct the following related experiments:
\begin{itemize}
     \item  Figure \ref{fig:MD_ener_test} shows the evolution of the normalized energy of MD-BDFk methods over time. 
It can be observed that, although the MD-BDFk methods do not theoretically possess the property of energy stability, their energy still dissipates over time. 
Moreover, to verify the energy stability property of ES-LM-MD-BDFk methods for the BDFk $(k=1,2)$ method, we plot the normalized energy over time in Figure \ref{fig:ESMDener_test}. It is evident that the ES-LM-MD-BDFk $(k=1,2)$ methods possess the property of energy stability.
\item For the SAV-MD-BDFk and LM-SAV-MD-BDFk methods, we plot the evolution of the normalized energy, normalized modified energy, and auxiliary scalar over time in Figures \ref{fig:enerSAVMD_test} and \ref{fig:LMSAVMDener_test}, respectively. 
Although the theoretical proof can only be provided for the stability of the modified energy, our numerical experiments demonstrate that both the original energy $W^h(t)$ and the modified energy $R(t)$ are dissipative over time. 
In addition, the auxiliary scalar $\zeta(t)$ remains very close to $1$ during the evolution process, which indicates that the modified energy $R(t)$ and the original energy $W^h(t)$ are very close to each other during the evolution.
\end{itemize}
\begin{figure}[!h]
         \centering
    \includegraphics[width=0.45\textwidth]{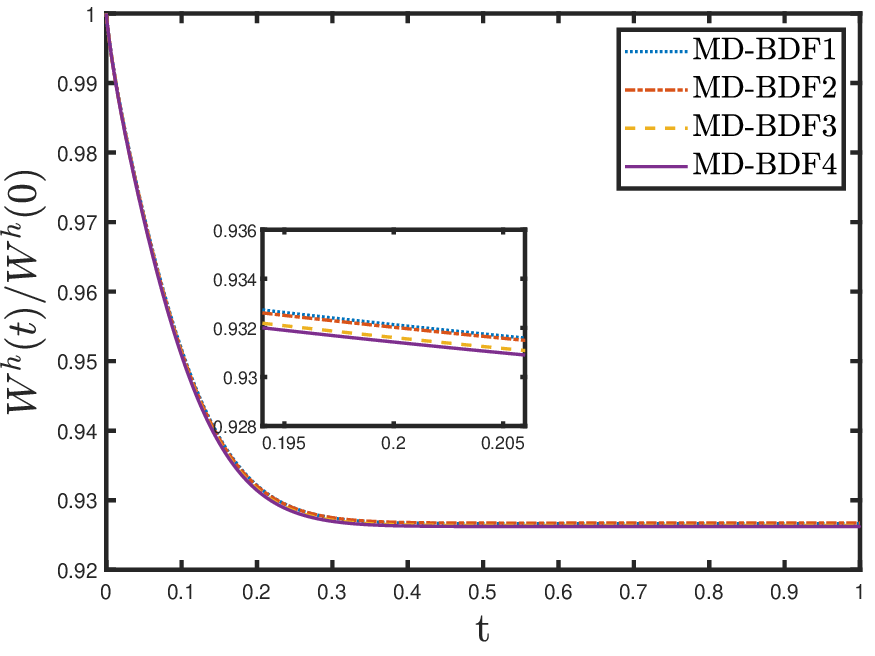}
    \includegraphics[width=0.45\textwidth]{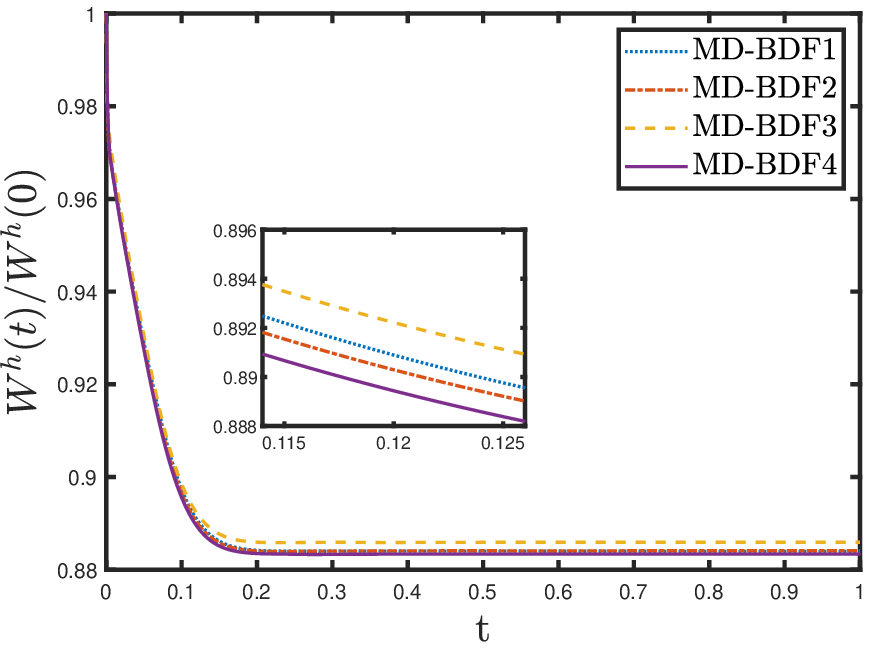}
    \caption{The normalized energy for MD-BDFk methods with different surface energy densities: $\gamma(\vec{n})=1+0.05 \left(n_1^4+n_2^4+n_3^4\right)$ and $\gamma(\vec{n})=1+0.5 \left(n_1^4+n_2^4+n_3^4\right)$. Other parameters are chosen as $h = 1.5356\times 10^{-1}$ and $\Delta t = 10^{-3}$.}
    \label{fig:MD_ener_test}
    \end{figure}
    
\begin{figure}[!h]
         \centering
    \includegraphics[width=0.30\textwidth]{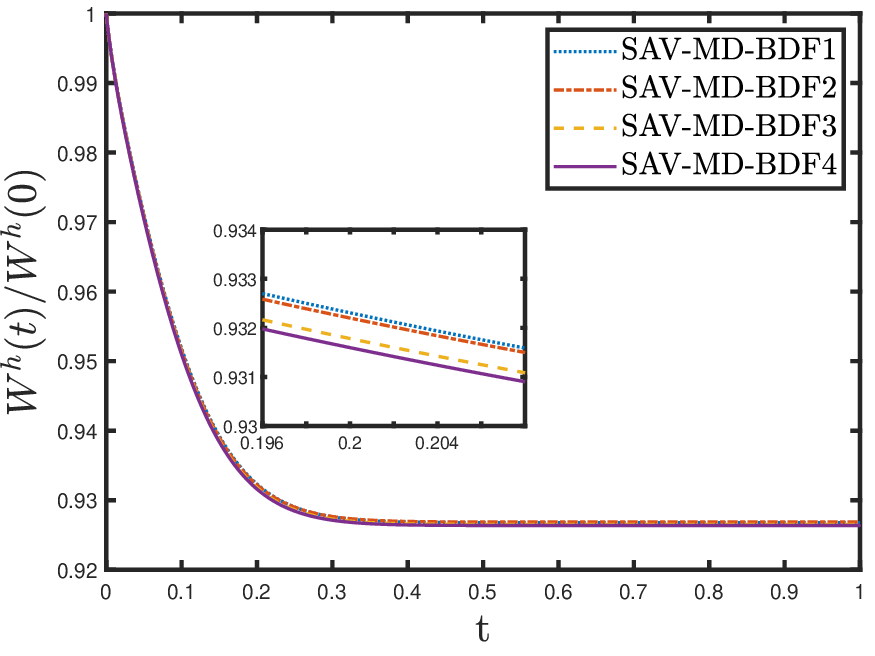}
    \includegraphics[width=0.30\textwidth]{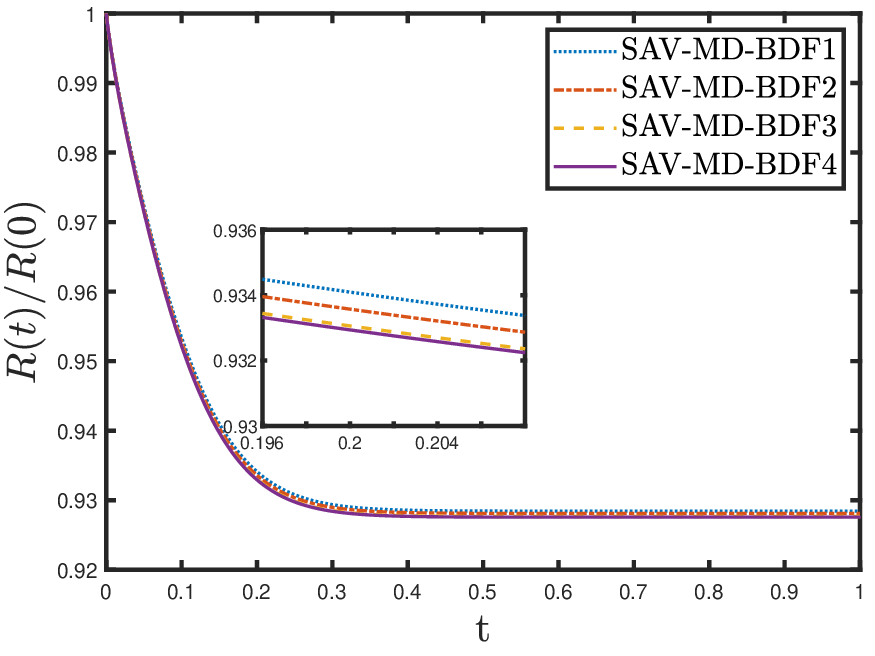}
    \includegraphics[width=0.30\textwidth]{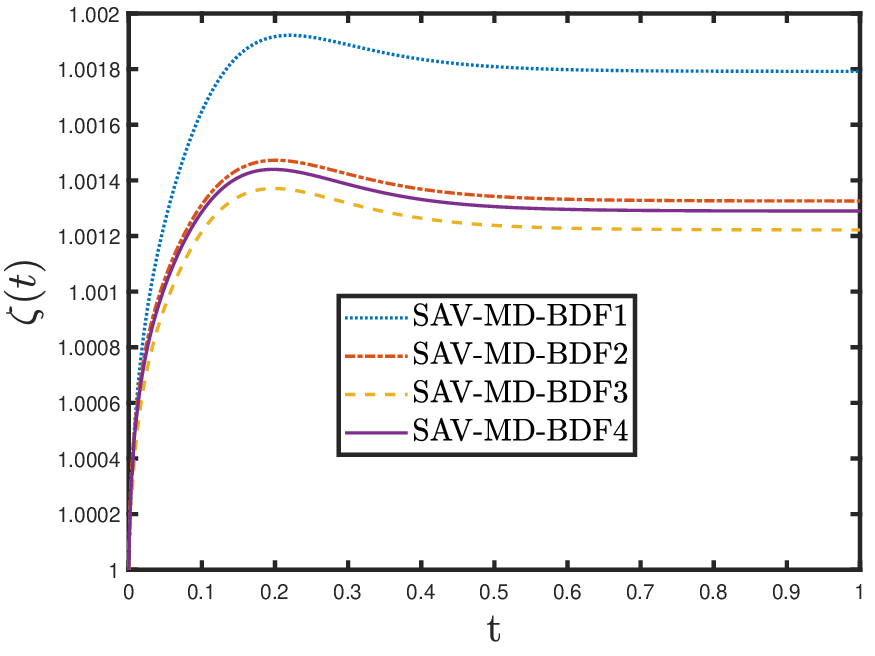}
    \includegraphics[width=0.30\textwidth]{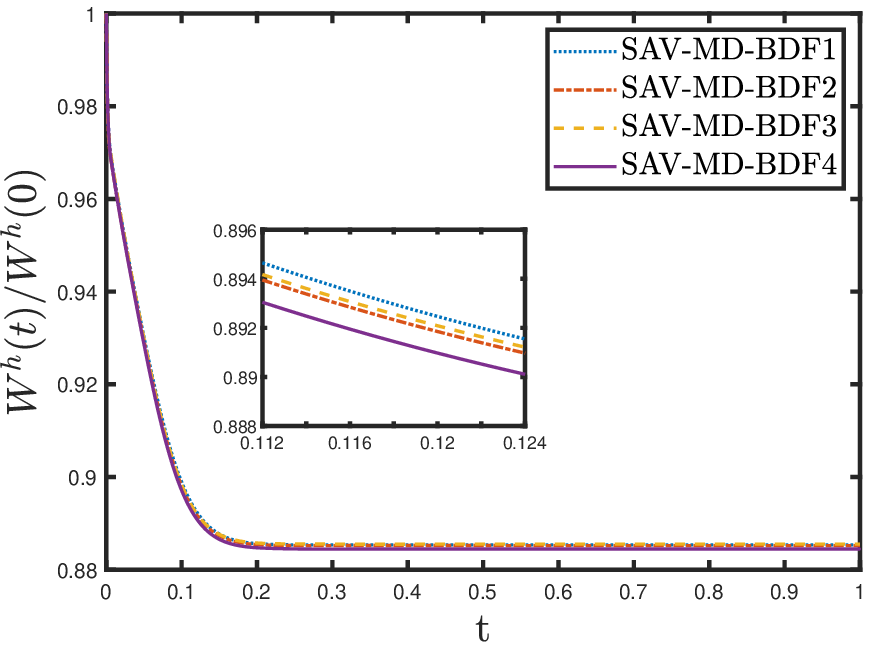}
    \includegraphics[width=0.30\textwidth]{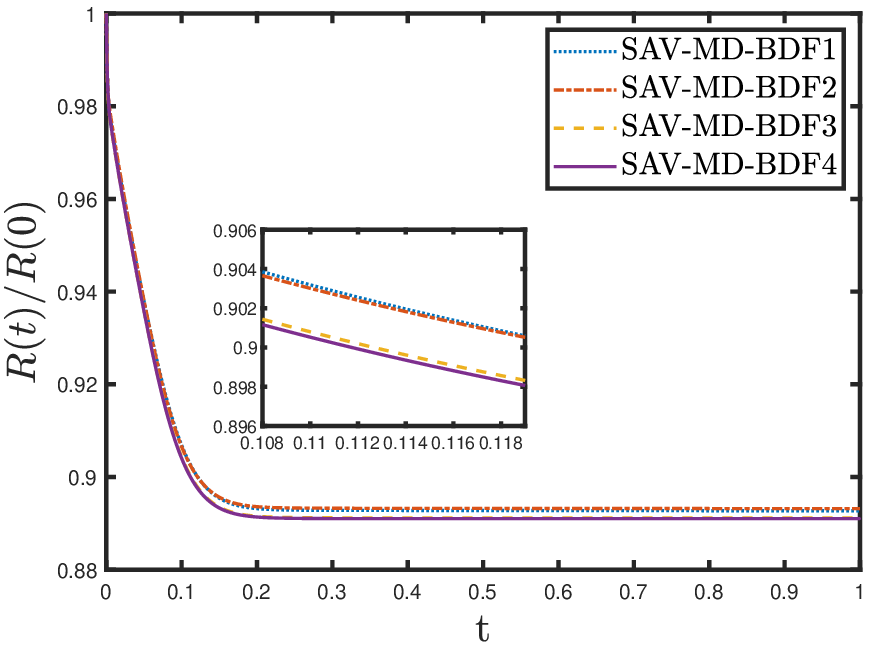}
    \includegraphics[width=0.30\textwidth]{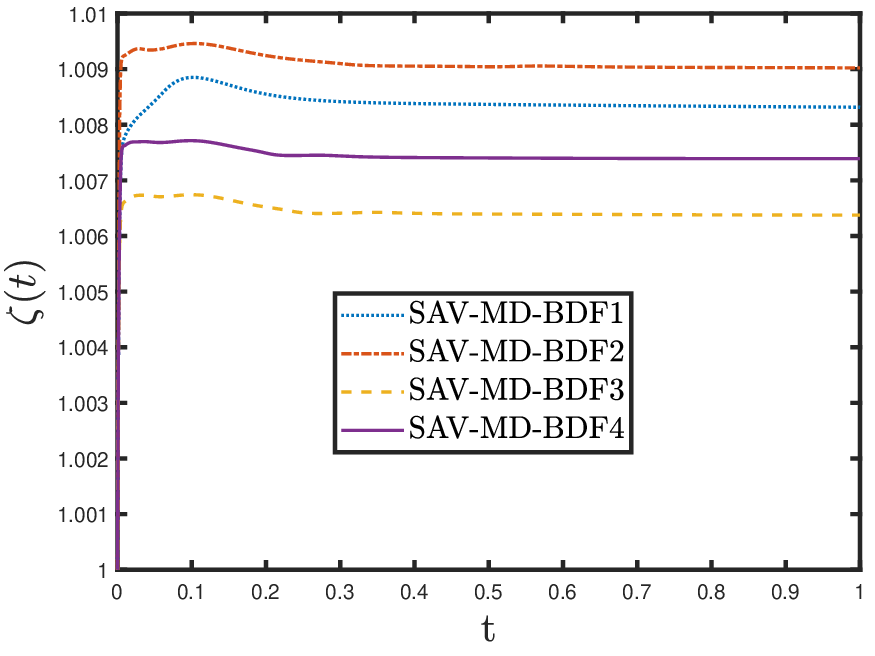}
    \caption{ Energy-related quantities of SAV-MD-BDFk methods with different energy densities: $\gamma(\vec{n})=1+0.05 \left(n_1^4+n_2^4+n_3^4\right)$ (the first row) and $\gamma(\vec{n})=1+0.5 \left(n_1^4+n_2^4+n_3^4\right)$ (the second row). The left colume: Normalized energy $W^h(t)/W^h(0)$; the middle colume: Normalized modified energy $R(t)/R(0)$; the right colume: Auxiliary scalar $\zeta(t)$. Other parameters are chosen as $h = 1.5356\times 10^{-1}$, $\Delta t = 10^{-3}$ and $r = 5$.}
    \label{fig:enerSAVMD_test}
    \end{figure}
    
\begin{figure}[!h]
         \centering
    \includegraphics[width=0.45\textwidth]{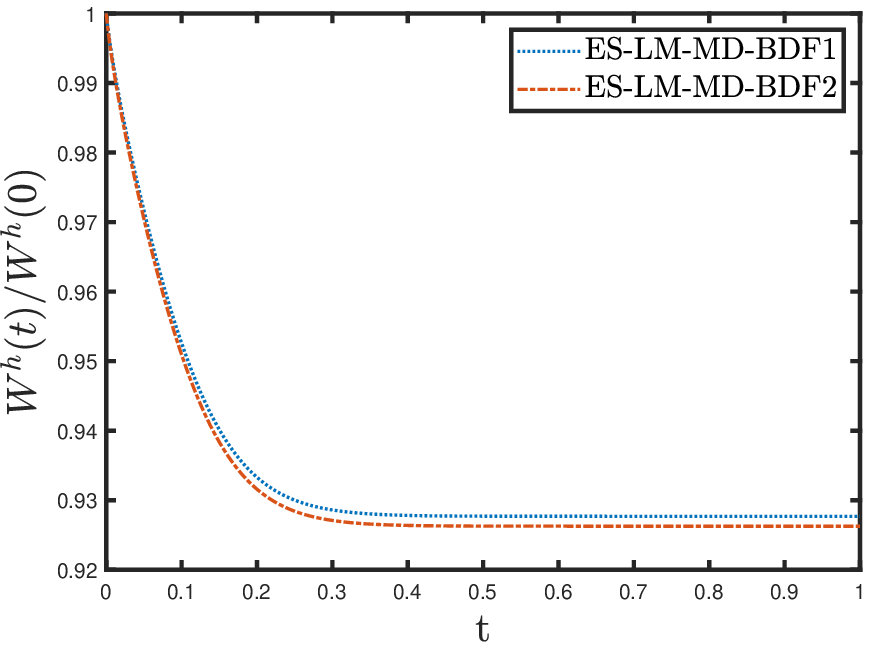}
    \includegraphics[width=0.45\textwidth]{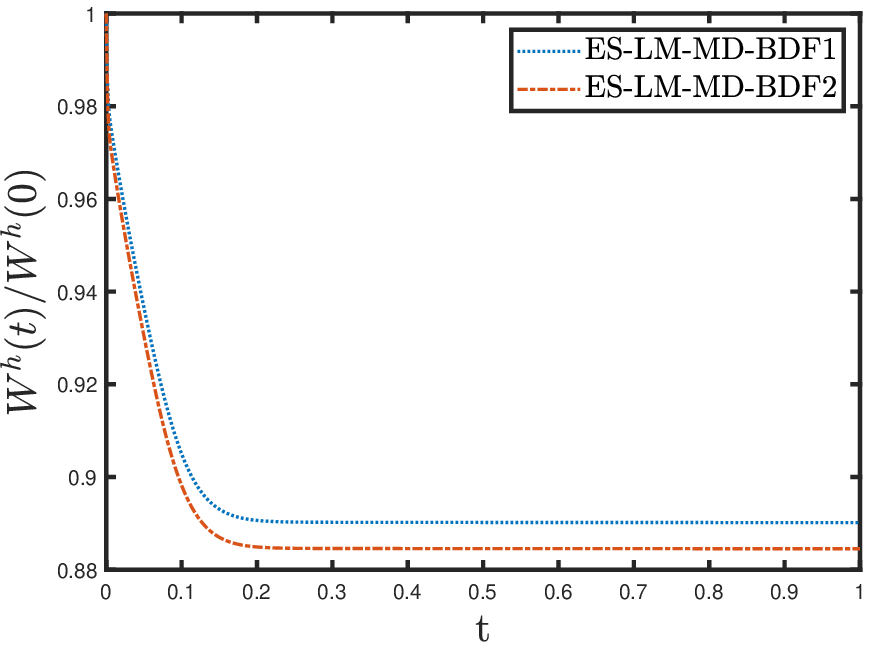}
    \caption{The normalized energy for ES-LM-MD-BDFk $(k = 1,2)$ with different surface energy densities: $\gamma(\vec{n})=1+0.05 \left(n_1^4+n_2^4+n_3^4\right)$ and $\gamma(\vec{n})=1+0.5 \left(n_1^4+n_2^4+n_3^4\right)$. Other parameters are chosen as $h = 1.5356\times 10^{-1}$ and $\Delta t = 10^{-3}$.}
    \label{fig:ESMDener_test}
\end{figure}
    
\begin{figure}[!h]
         \centering
    \includegraphics[width=0.30\textwidth]{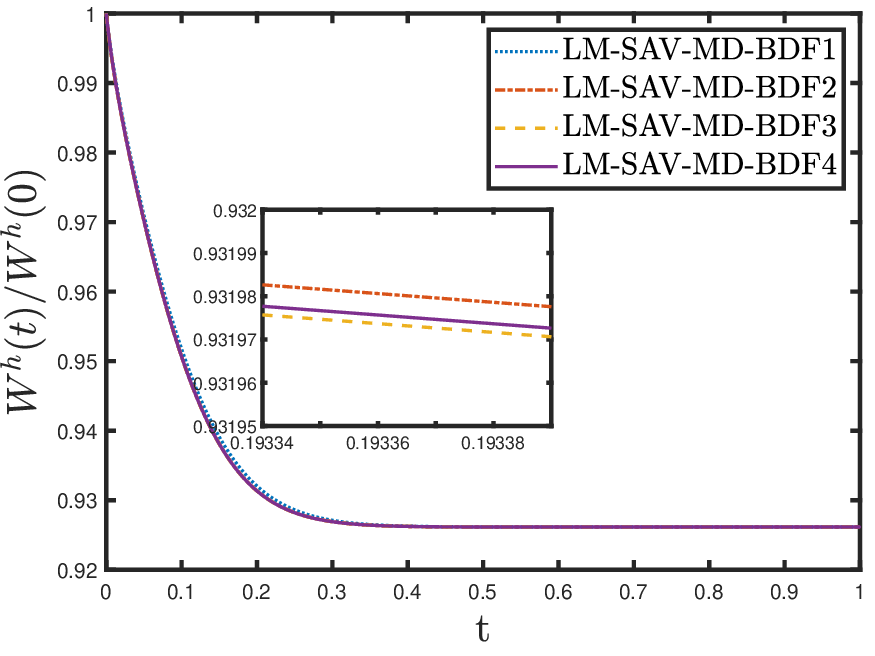}
    \includegraphics[width=0.30\textwidth]{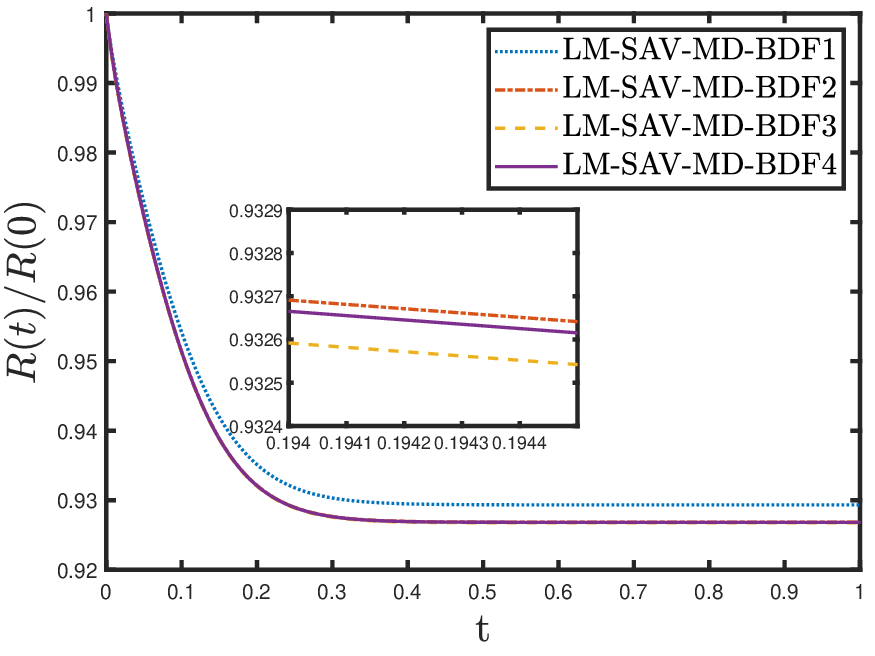}
    \includegraphics[width=0.30\textwidth]{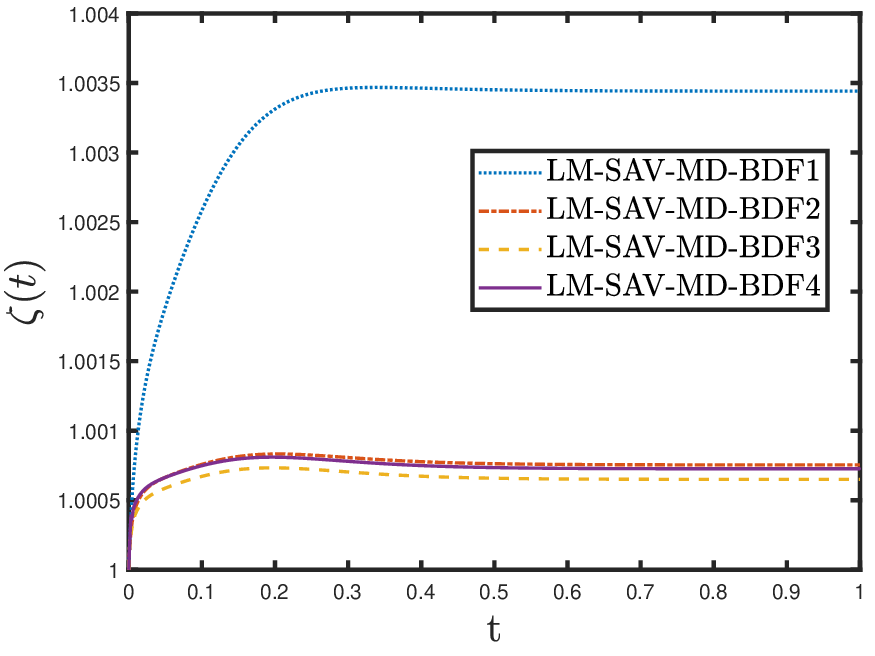}
    \includegraphics[width=0.30\textwidth]{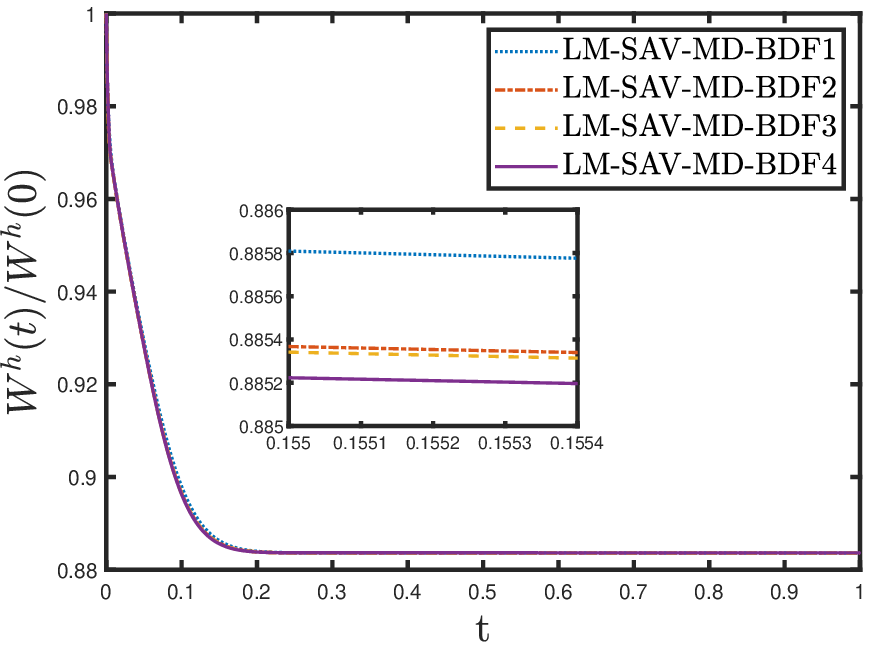}
    \includegraphics[width=0.30\textwidth]{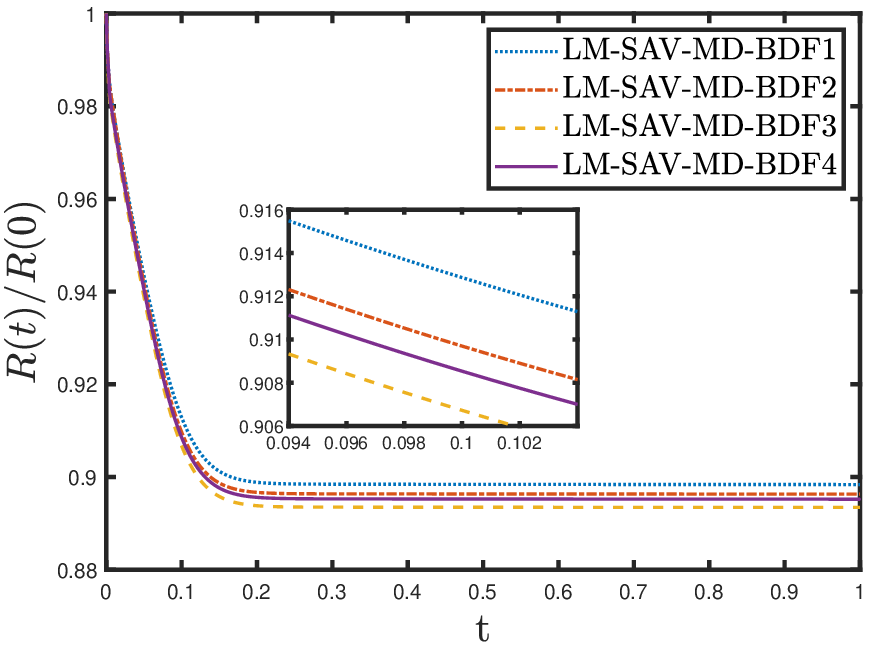}
    \includegraphics[width=0.30\textwidth]{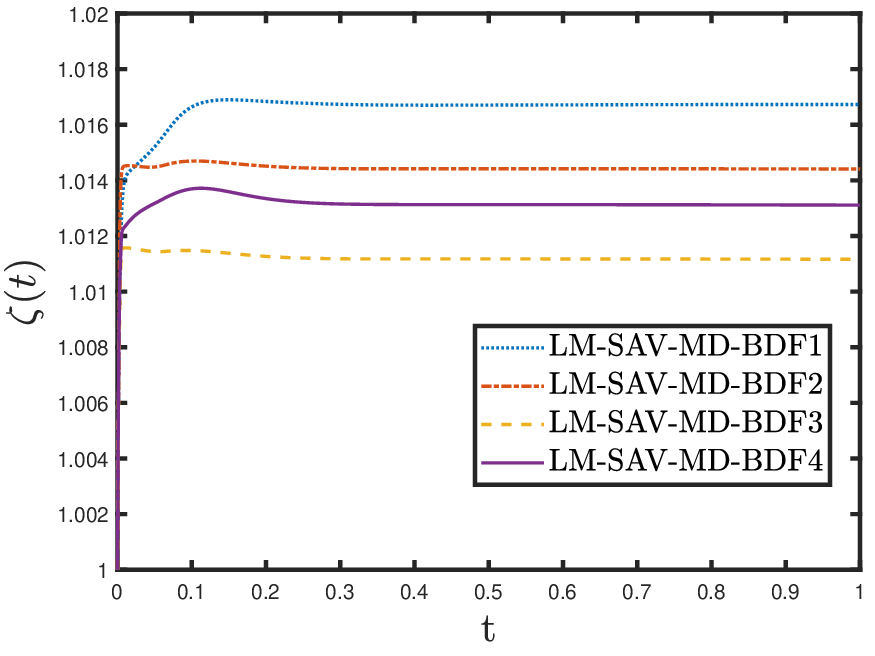}
    \caption{Energy-related quantities of LM-SAV-MD-BDFk methods with different energy densities: $\gamma(\vec{n})=1+0.05 \left(n_1^4+n_2^4+n_3^4\right)$ (the first row) and $\gamma(\vec{n})=1+0.5 \left(n_1^4+n_2^4+n_3^4\right)$ (the second row). The left colume: Normalized energy $W^h(t)/W^h(0)$; the middle colume: Normalized modified energy $R(t)/R(0)$; the right colume: Auxiliary scalar $\zeta(t)$. Other parameters are chosen as $h = 1.5356\times 10^{-1}$, $\Delta t = 10^{-3}$ and $r = 5$.}
    \label{fig:LMSAVMDener_test}
\end{figure}

\textbf{Example 4} (Mesh quality)
In this experiment, we test the mesh quality during the surface evolution of anisotropic SDF under different surface energy densities using the BGN and MD methods, with initial surfaces consisting of a \( 6 \times 1 \times 1 \) cuboid with a mesh size of \( h = 1.4633 \times 10^{-1} \) and a torus with a mesh size of \( h = 2.4263 \times 10^{-1} \).
Figures \ref{fig:meshellipse112_test}-\ref{fig:mesh4foldtorus_test} show that, for both ellipsoidal and 4-fold anisotropic surface energies, the BGN-BDF1 method experiences mesh degeneration during evolution when the time step is very small, a phenomenon that does not occur when the time step is increased.
In contrast, the MD-BDF1 method maintains good mesh quality throughout the entire evolution process.
In this example, we have observed that, when the time step is small, the MD method yields better mesh quality than the BGN method. Therefore, for simplicity, we will not further examine the mesh quality of the other methods.

\begin{figure}[!h]
         \centering
    \includegraphics[width=0.30\textwidth]{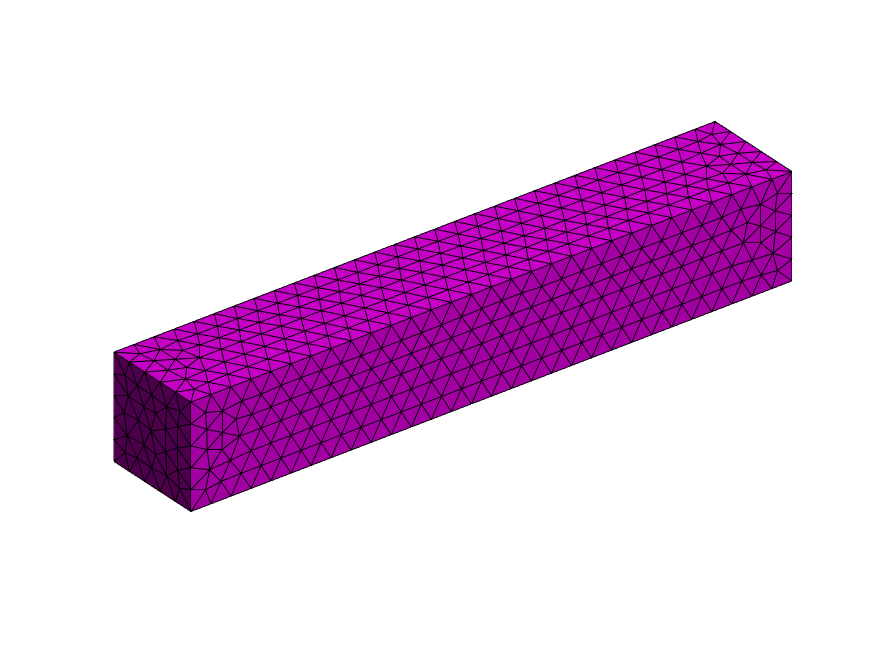}
    \includegraphics[width=0.30\textwidth]{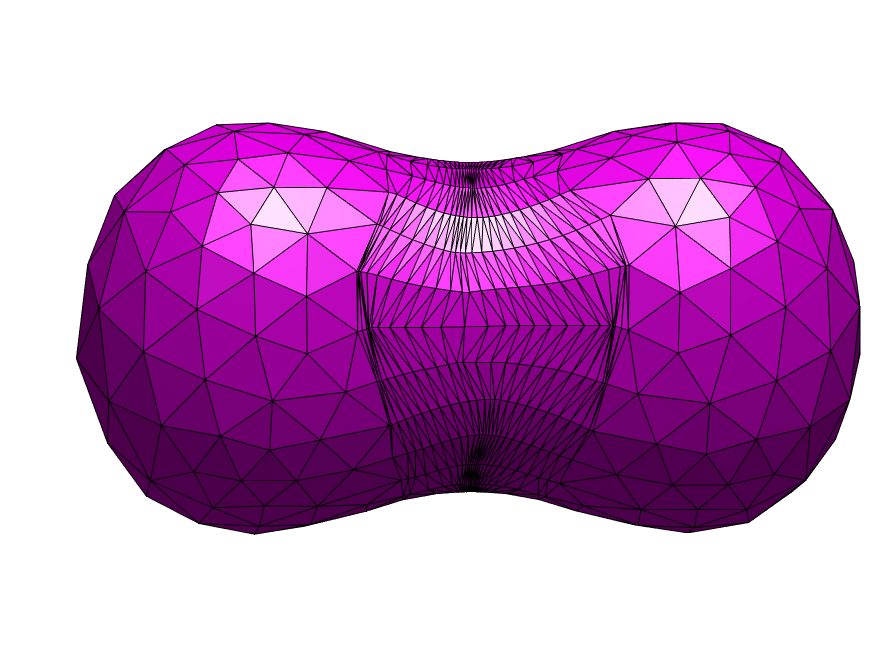}
    \includegraphics[width=0.30\textwidth]{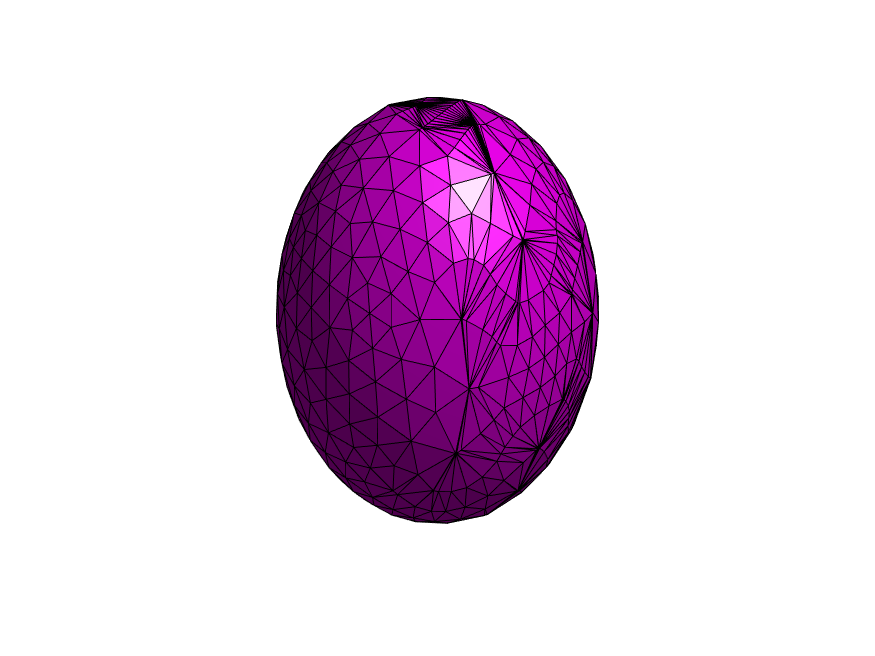}
    \includegraphics[width=0.30\textwidth]{example/mesh/initialshape611cube020.eps}
    \includegraphics[width=0.30\textwidth]{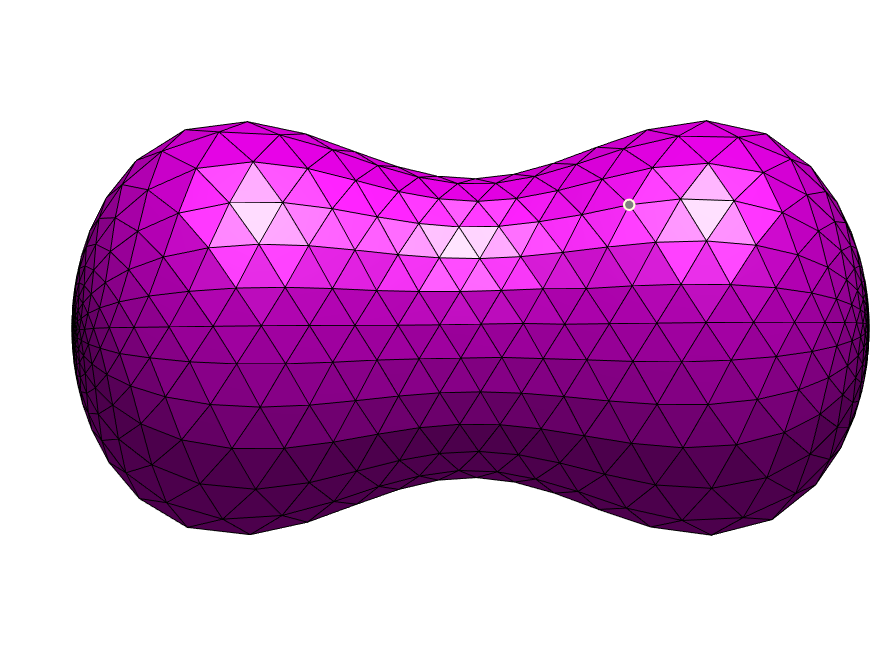}
    \includegraphics[width=0.30\textwidth]{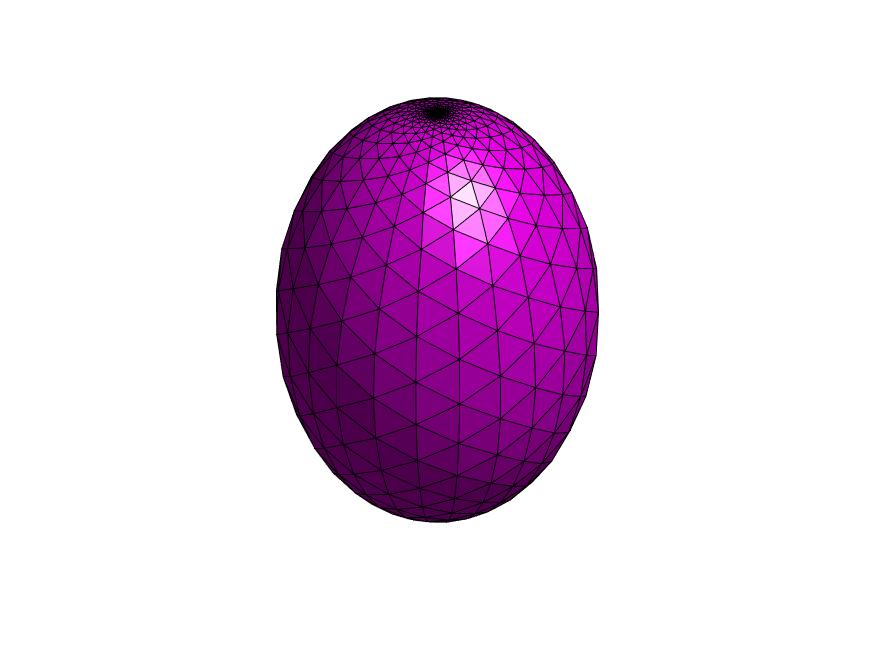}
    
    \caption{Surface evolution of BGN-BDF1 method and MD-BDF1 method under anisotropic surface energy density: $\gamma(\vec{n})=\sqrt{{n_1^2+n_2^2+2n_3^2}}$ at different times: t = 0, 0.3, 1. The first row: BGN-BDF1 method. The second row: MD-BDF1 method. Other parameters are chosen as $h = 1.4633 \times 10^{-1}$ and $\Delta t =10^{-4}$.}
    \label{fig:meshellipse112_test}
\end{figure}
\begin{figure}[!h]
         \centering
    \includegraphics[width=0.30\textwidth]{example/mesh/initialshape611cube020.eps}
    \includegraphics[width=0.30\textwidth]{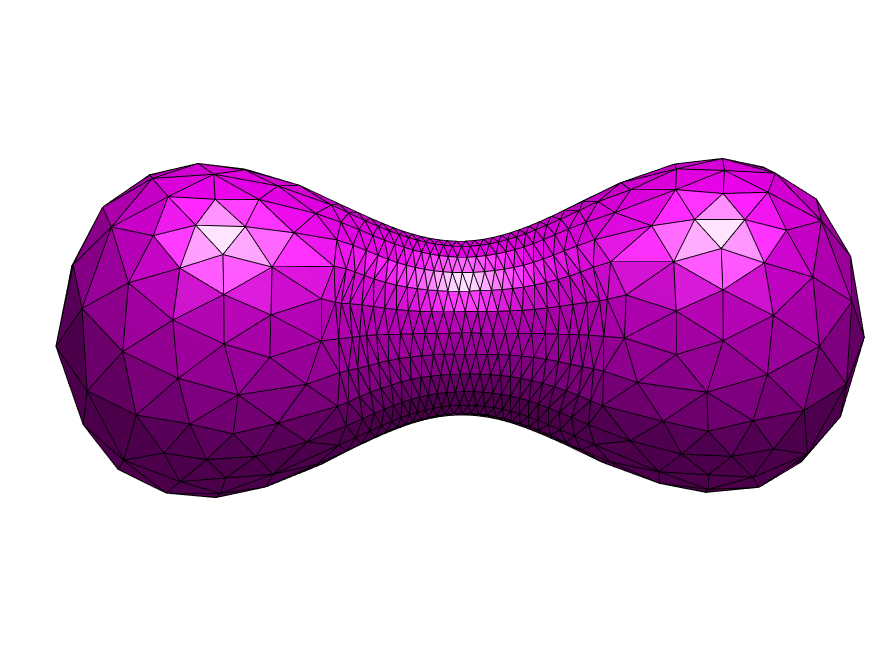}
    \includegraphics[width=0.30\textwidth]{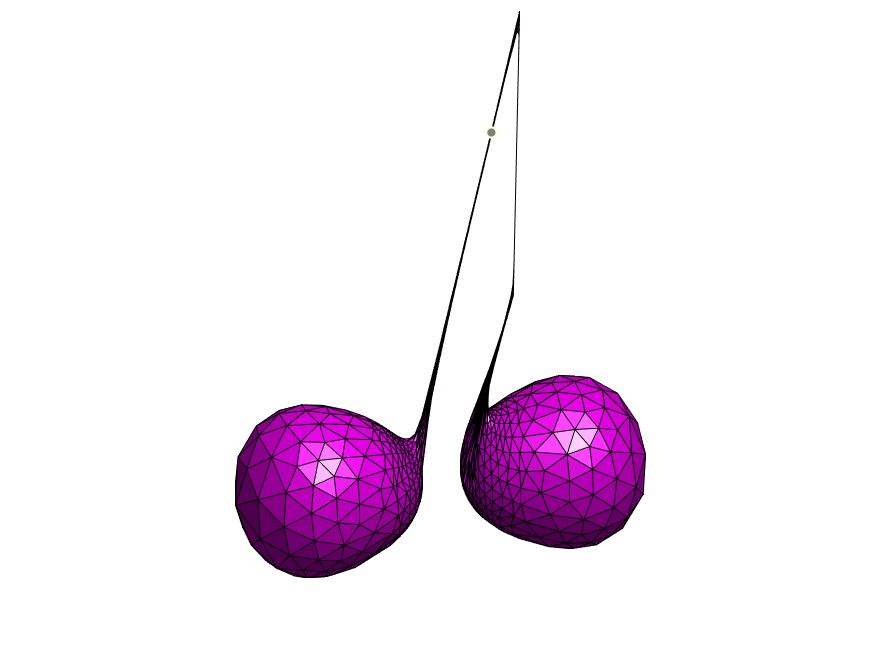}
    \includegraphics[width=0.30\textwidth]{example/mesh/initialshape611cube020.eps}
    \includegraphics[width=0.30\textwidth]{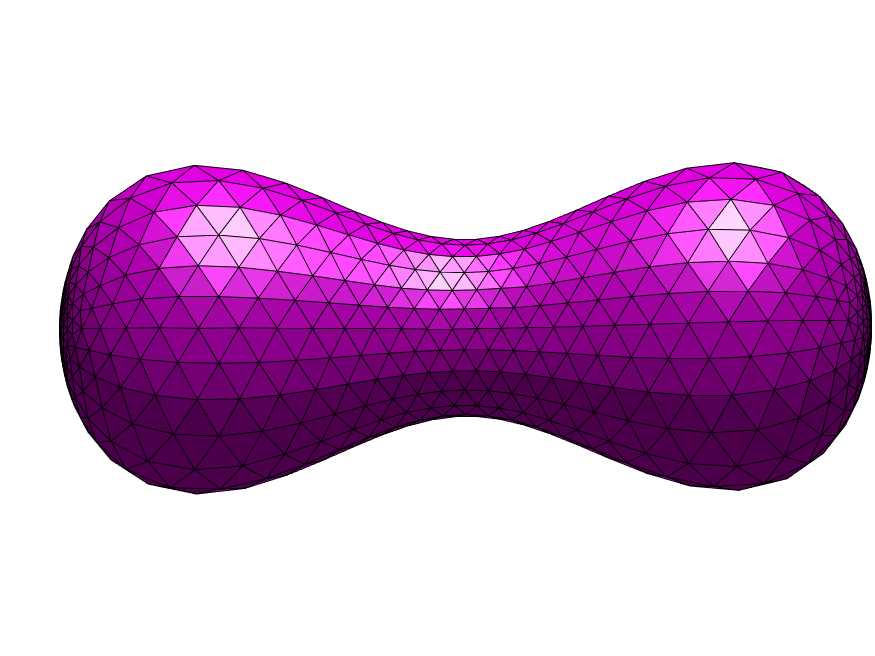}
    \includegraphics[width=0.30\textwidth]{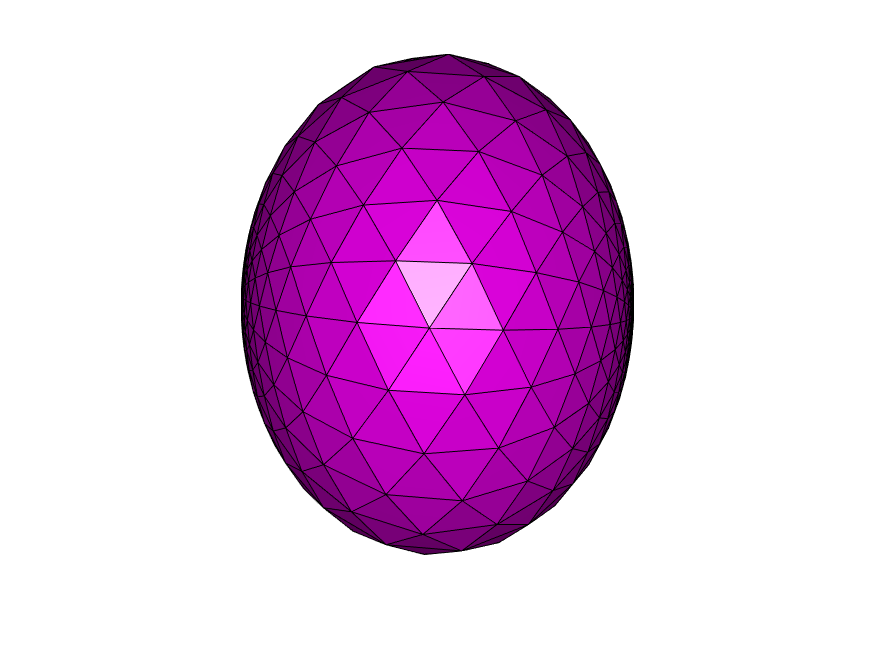}
    
    \caption{Surface evolution of BGN-BDF1 method and MD-BDF1 method under anisotropic surface energy density: $\gamma(\vec{n})=\sqrt{{n_1^2+2n_2^2+2n_3^2}}$ at different times. The first row: BGN-BDF1 method at time: $t = 0,0.132,0.142$. The second row: MD-BDF1 method at time: $t = 0,0.132,1$. Other parameters are chosen as $h = 1.4633 \times 10^{-1}$ and $\Delta t =10^{-4}$.}
    \label{fig:meshellipse122_test}
\end{figure}

\begin{figure}[!h]
         \centering
    \includegraphics[width=0.33\textwidth]{example/mesh/initialshape611cube020.eps}
    \includegraphics[width=0.33\textwidth]{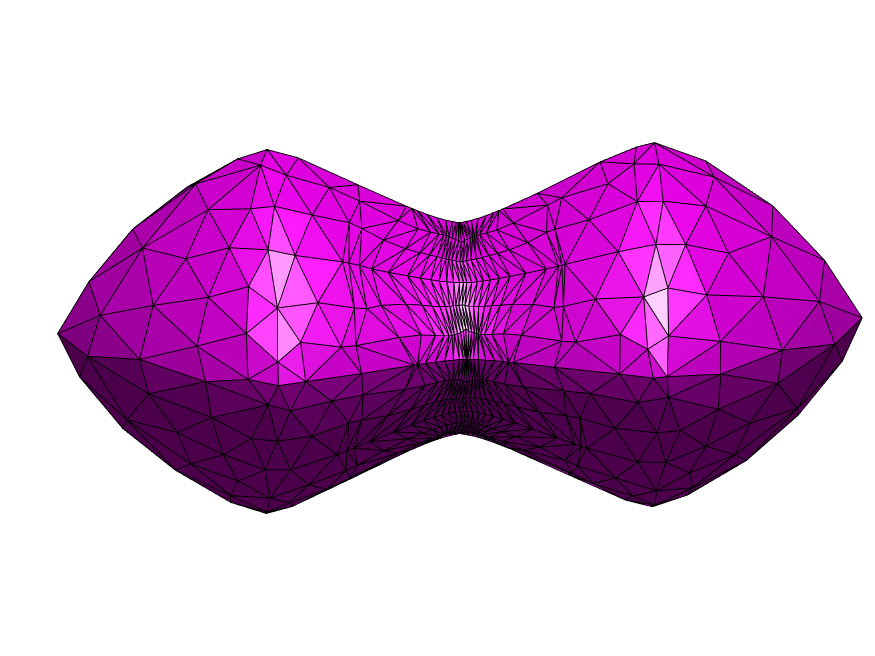}
    \includegraphics[width=0.33\textwidth]{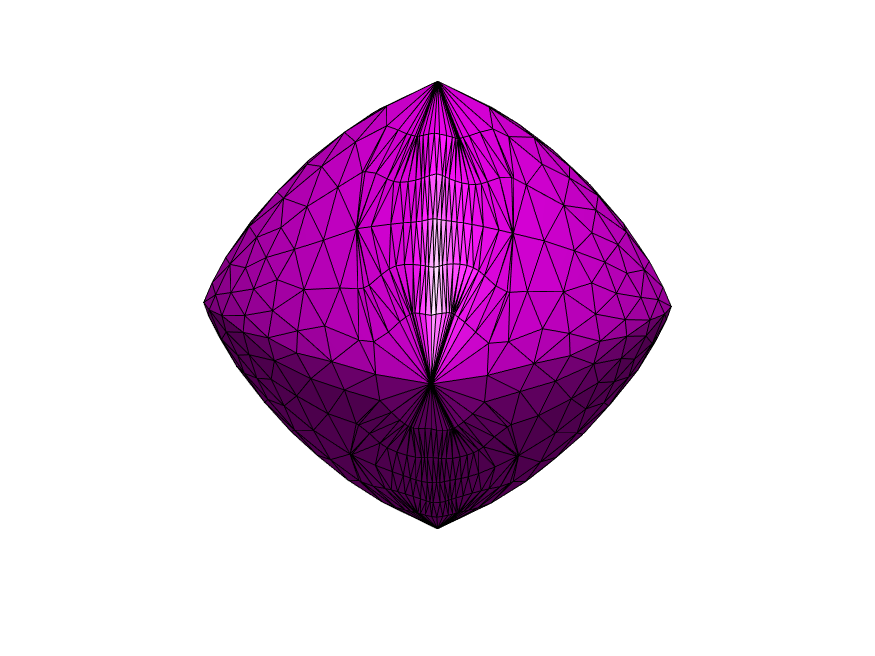}
    \includegraphics[width=0.33\textwidth]{example/mesh/initialshape611cube020.eps}
    \includegraphics[width=0.33\textwidth]{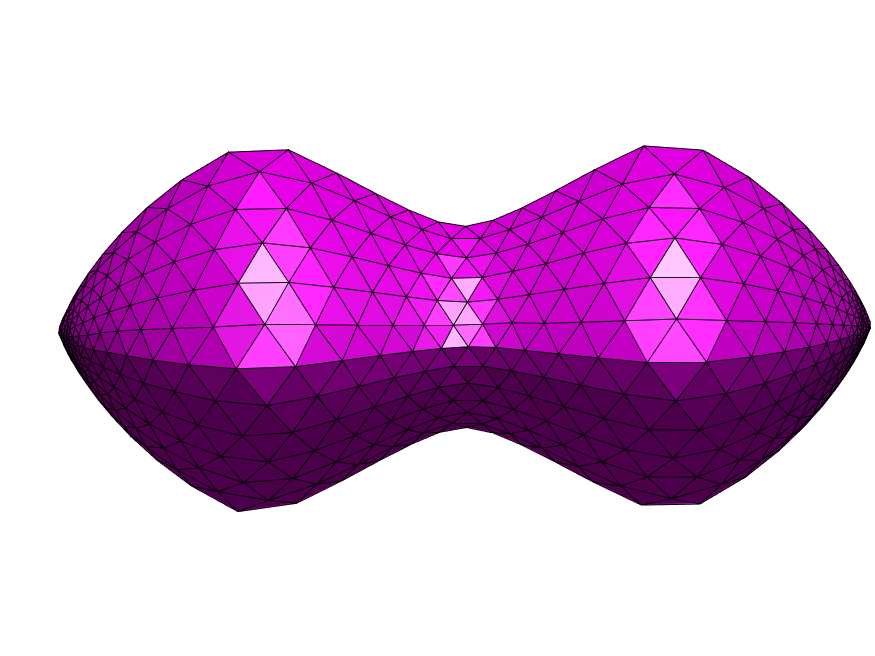}
    \includegraphics[width=0.33\textwidth]{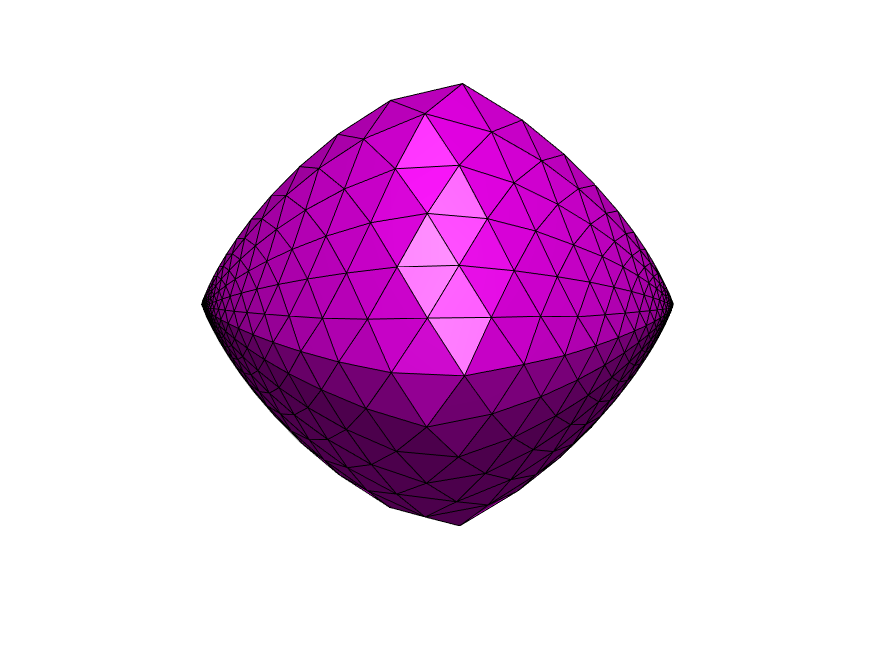}
    
    \caption{Surface evolution of BGN-BDF1 method and MD-BDF1 method under anisotropic surface energy density: $\gamma(\vec{n})=1+0.5 \left(n_1^4+n_2^4+n_3^4\right)$ at different times: t = 0, 0.3, 1. The first row: BGN-BDF1 method. The second row: MD-BDF1 method. Other parameters are chosen as $h = 1.4633 \times 10^{-1}$ and $\Delta t =10^{-4}$.}
    \label{fig:mesh4foldcube_test}
\end{figure}

\begin{figure}[!h]
         \centering
    \includegraphics[width=0.33\textwidth]{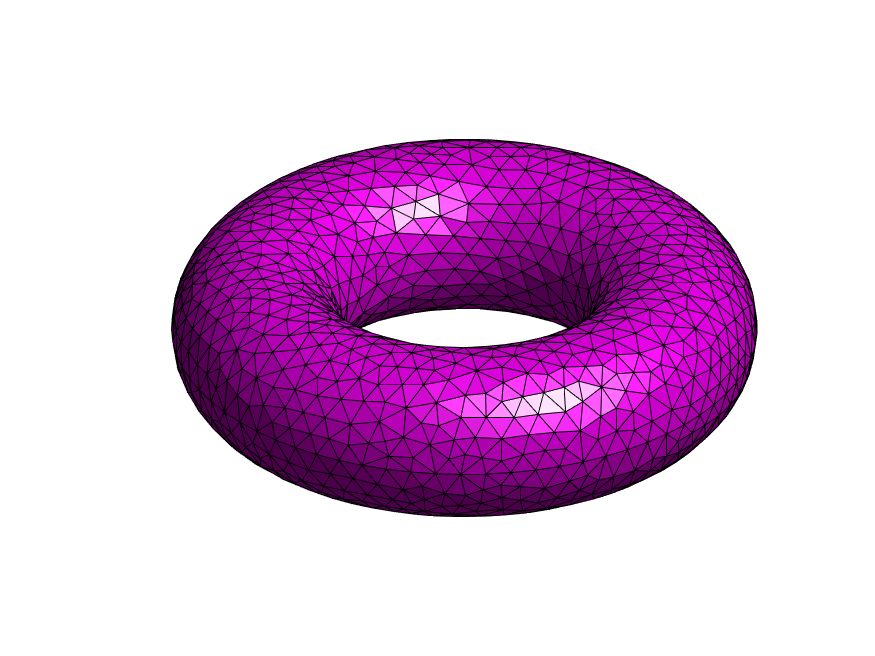}
    \includegraphics[width=0.33\textwidth]{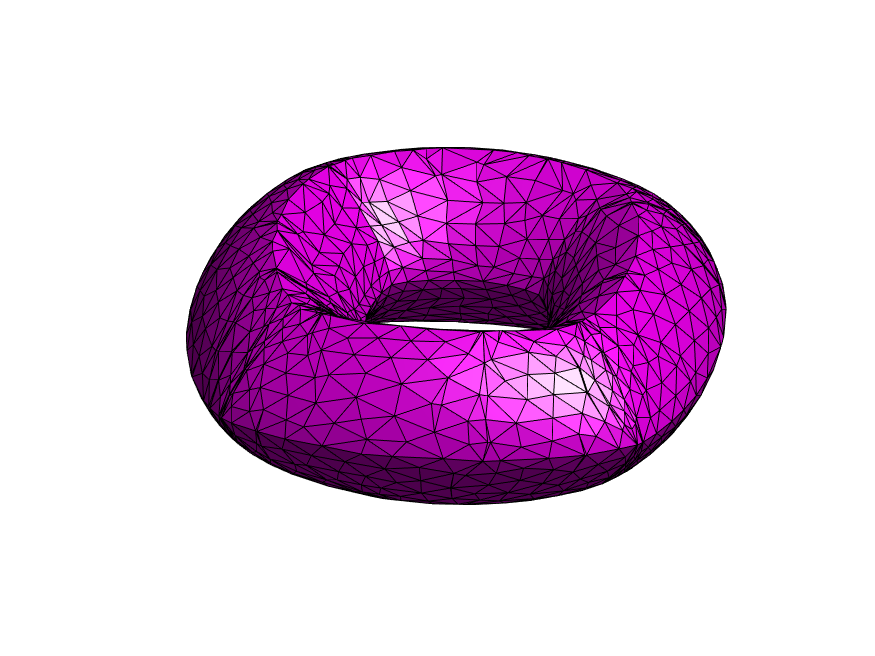}
    \includegraphics[width=0.33\textwidth]{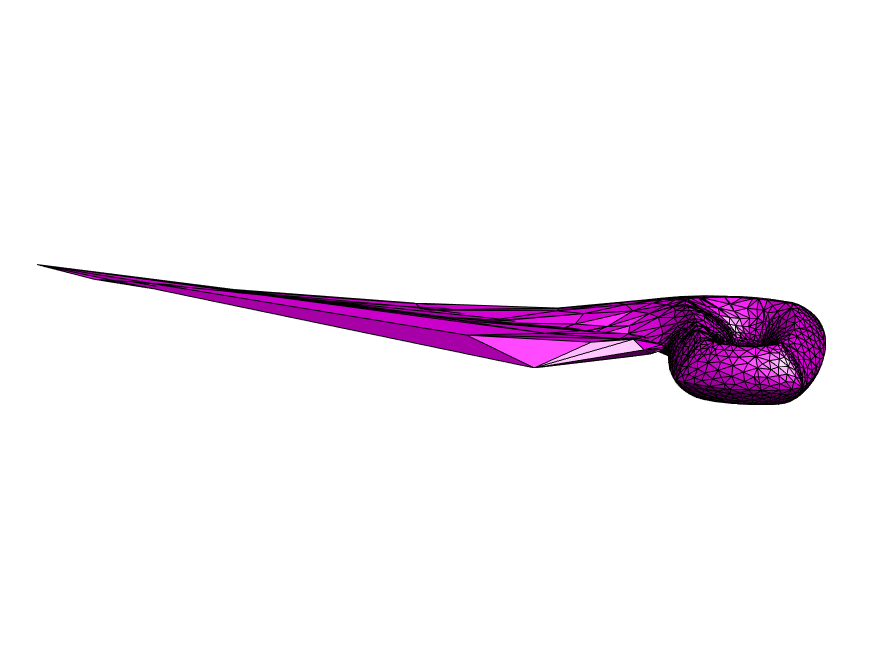}
    \includegraphics[width=0.33\textwidth]{example/mesh/initialshapetorus030.eps}
    \includegraphics[width=0.33\textwidth]{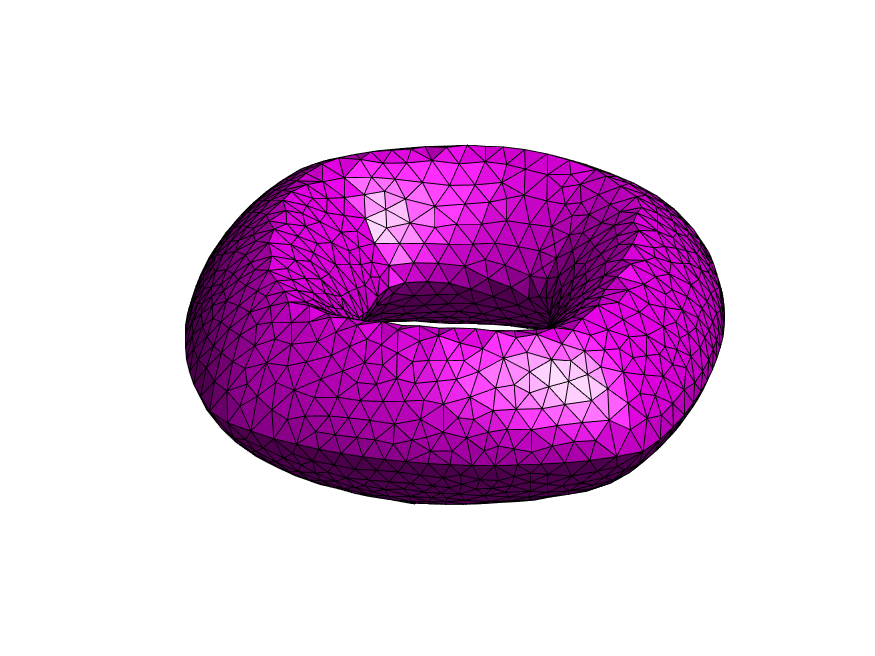}
    \includegraphics[width=0.33\textwidth]{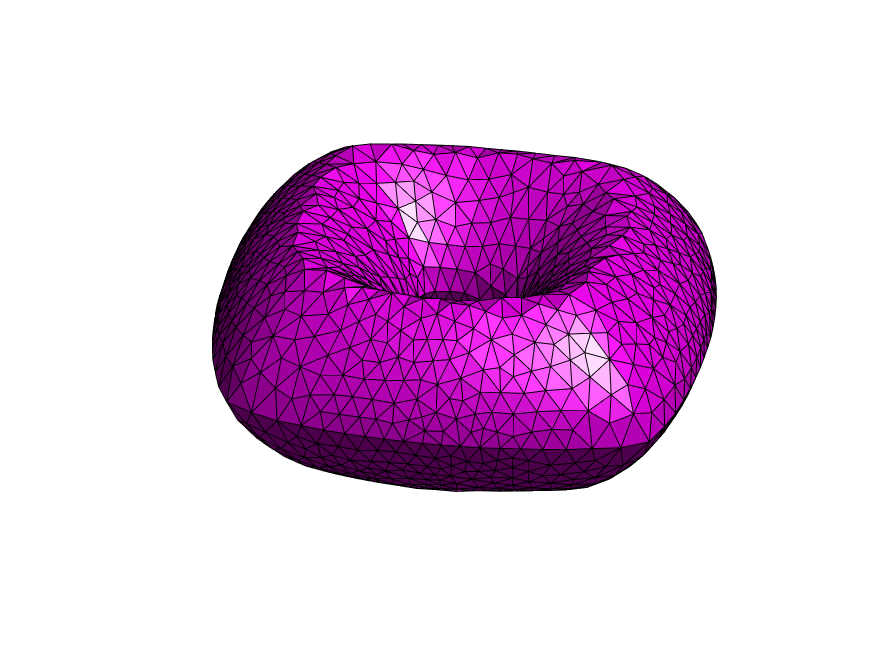}
    
    \caption{Surface evolution of BGN-BDF1 method and MD-BDF1 method under anisotropic surface energy density: $\gamma(\vec{n})=1+0.5 \left(n_1^4+n_2^4+n_3^4\right)$ at different times. The first row: BGN-BDF1 method at time: $t = 0, 0.3, 0.8509$. The second row: MD-BDF1 method at time: $t = 0, 0.3, 1$. Other parameters are chosen as $h = 2.4263 \times 10^{-1}$ and $\Delta t =10^{-4}$.}
    \label{fig:mesh4foldtorus_test}
\end{figure}

\section{Conclusion} \label{sec:conclusion}

In this study, we extend the MD formulation to anisotropic SDF, which, similar to the isotropic case, also maintains excellent mesh quality. Another key innovation of this work is the development of structure-preserving algorithms with high-order temporal accuracy.
Firstly, we develop the MD-BDFk method based on the MD formulation. By replacing the normal vector \( \vec{n}^m \) with \( \vec{n}^{m+\frac{1}{2}} \) in the MD-BDF1 method, we derive the volume-conservative VC-MD-BDF1 method.
Secondly, we design energy-stable SAV-MD-BDFk methods for the SDF by combining the SAV approach with the MD formulation. Similarly, by replacing \( \vec{n}^m \) with \( \vec{n}^{m+\frac{1}{2}} \) in the SAV-MD-BDF1 method, we construct the VC-SAV-MD-BDF1 method, which approximately conserves volume.
To construct high-order volume-conservative and energy-stable methods, we combine the LM approach, MD formulation, and BDFk schemes to develop the LM-MD-BDFk methods. This includes the volume-conservative VC-LM-MD-BDFk methods, the energy-stable ES-LM-MD-BDFk methods (for \(k=1, 2\)), and the structure-preserving SP-LM-MD-BDFk methods.
To address the instability issues in the SP-LM-MD-BDFk methods while maintaining temporal higher-order energy stability (\(k \geq 3\)), we further develop novel LM-SAV-MD-BDFk methods, which exhibit approximate volume conservation and energy stability.
Through extensive numerical experiments, we demonstrate that our methods successfully maintain mesh quality for anisotropic SDFs, achieving high-order temporal accuracy, volume conservation, and energy stability, with different methods exhibiting distinct characteristics.

\bibliographystyle{elsarticle-num}
\bibliography{references}
\end{document}